\documentclass{amsart}

\title{Two-complete stable motivic stems over finite fields}

\author{Glen Matthew Wilson}
\email{glenw@math.uio.no}
\address{Department of Mathematics, University of Oslo, Norway.  }

\author{Paul Arne {\O}stv{\ae}r}
\email{paularne@math.uio.no}
\address{Department of Mathematics, University of Oslo, Norway.  }



\usepackage{etex}
\usepackage{amssymb,amsfonts,amsmath,amsthm}
\usepackage{longtable} 
\usepackage{hyperref}
\usepackage{tikz}
\usepackage{bbm}
\usepackage{setspace}
\usepackage{colortbl}
\usepackage{enumerate}
\usepackage{graphicx}
\usepackage{pinlabel}
\usepackage{mathrsfs}

\newcommand{\Hcomp}{{}^{{\kern -.4pt}\wedge}_{{\kern -1.5pt}H}}
\newcommand{\MASS}{\mathfrak{M}}
\newcommand{\myol}[2][3]{{}\mkern#1mu\overline{\mkern-#1mu#2}}
\newcommand{\twocomp}{{}^{{\kern -.5pt}\wedge}_2}
\newcommand{\ellcomp}{{}^{{\kern -.5pt}\wedge}_{\ell}}
\newcommand{\unit}{\mathbbm{1}}
\newcommand{\SH}{\mathcal{SH}}
\newcommand{\Sm}{\mathrm{Sm}}

\newcommand{\Spc}{\mathrm{Spc}}
\newcommand{\Spt}{\mathrm{Spt}}
\newcommand{\Ev}{\mathrm{Ev}}

\newcommand{\Fbar}{\myol{F}}
\newcommand{\Kbar}{\myol{K}}
\newcommand{\kbar}{\myol{k}}
\newcommand{\st}{ \, \vert \, }
\newcommand{\pihat}{\hat{\pi}}
\newcommand{\del}{\partial}
\newcommand{\todo}[1]{\marginpar{{\footnotesize#1}}}
\renewcommand{\todo}[1]{}

\newcommand{\details}[1]
{
{\footnotesize {\it Details:}
#1
}
}

\newcommand{\bispt}{\Spt_{\bbG_m, S^1}}
\newcommand{\symspt}{\Spt^{\Sigma}}
\newcommand{\Fpbar}{\overline{\bbF}_p} 

\newcommand{\bbF}{\mathbb{F}} 
\newcommand{\calA}{\mathcal{A}} 
\newcommand{\calC}{\mathcal{C}} 
\newcommand{\calH}{\mathcal{H}}
\newcommand{\calO}{\mathcal{O}}
 
\newcommand{\calX}{\mathcal{X}}
\newcommand{\calY}{\mathcal{Y}}
\newcommand{\bbA}{\mathbb{A}} 
\newcommand{\bbC}{\mathbb{C}} 
\newcommand{\bbG}{\mathbb{G}}
\newcommand{\bbL}{\mathbb{L}} 
\newcommand{\bbP}{\mathbb{P}}
\newcommand{\bbR}{\mathbb{R}} 
\newcommand{\bbZ}{\mathbb{Z}}
\newcommand{\frakS}{\mathfrak{S}}
\newcommand{\Fqbar}{\overline{\bbF}_q}

\DeclareMathOperator{\Ext}{Ext}
\DeclareMathOperator{\sSet}{\underline{sSet}}
\DeclareMathOperator{\Hom}{Hom} 
\DeclareMathOperator{\im}{im}
\DeclareMathOperator{\colim}{colim}
\DeclareMathOperator{\holim}{holim}
\DeclareMathOperator{\Sq}{Sq}
\DeclareMathOperator{\Spec}{Spec} 
\DeclareMathOperator{\InverseLimit}{\underleftarrow{\lim}}

\usepackage[all, arc, ps]{xy} 

\theoremstyle{definition} 
\newtheorem{proposition}{Proposition}[section]
\newtheorem{definition}[proposition]{Definition} 
\newtheorem{theorem}[proposition]{Theorem}
\newtheorem{example}[proposition]{Example} 
\newtheorem{corollary}[proposition]{Corollary} 
\newtheorem{lemma}[proposition]{Lemma} 
\theoremstyle{remark} 
\newtheorem{remark}[proposition]{Remark} 
\numberwithin{equation}{section}

\begin{document}

\begin{abstract}
  Let $\ell$ be a prime and $q = p^{\nu}$ where $p$ is a prime
  different from $\ell$.  We show that the $\ell$-completion of the
  $n$th stable homotopy group of spheres is a summand of the
  $\ell$-completion of the $(n, 0)$ motivic stable homotopy group of
  spheres over the finite field with $q$ elements $\bbF_q$.  With
  this, and assisted by computer calculations, we are able to
  explicitly compute the two-complete stable motivic stems $\pi_{n,
    0}(\bbF_q)\twocomp$ for $0\leq n\leq 18$ for all finite fields and
  $\pi_{19, 0}(\bbF_q)\twocomp$ and $\pi_{20, 0}(\bbF_q)\twocomp$ when
  $q \equiv 1 \bmod 4$ assuming Morel's connectivity theorem for
  $\bbF_q$ holds.
\end{abstract}

\maketitle

\section{Introduction}

The homotopy groups of spheres belong to the most important and
puzzling invariants in topology.  See Kochman \cite{Kochman} and the
more recent works of Isaksen \cite{StableStems} and Wang and Xu
\cite{WX} for amazing computer assisted ways of computing these
invariants based on the Adams spectral sequence. The Adams spectral
sequence of topology is a well studied method to calculate the stable
homotopy groups of spheres, see Adams \cite{Adams} and Ravenel
\cite{Ravenel}.  With two-primary coefficients, the second page of the
Adams spectral sequence has a description in terms of Ext groups over
the mod $2$ Steenrod algebra
\begin{equation*} 
  E_2^{s,t} =
  \Ext_{\calA^{*}}^{s,t}(\bbF_{2},\bbF_{2})
\end{equation*} 
and converges to the two-complete stable homotopy groups of spheres
$(\pi_n^s)\twocomp$.  Extensive computer calculations of these Ext
groups have been carried out by Bruner in \cite{BB1} and \cite{BB2}.
However, even if one knew completely the answer for the Ext groups in
the Adams spectral sequence, one is still not finished with computing
the stable homotopy groups of spheres.  One needs to know in addition
the differentials and all the group extensions hidden in the
associated graded of the filtration.  Only partial results have been
obtained in spite of an enormous effort.

Given any field $k$ the stable motivic homotopy category $\SH_k$ over
$k$ has the structure of a triangulated category and encodes both
topological information and arithmetic information about $k$.  An
application of this framework is Voevodsky's proof of Milnor's
conjecture on Galois cohomology \cite{MR2031199}.  Just as for the
stable homotopy category $\SH$, it is an interesting and deep problem
to compute the stable motivic homotopy groups of spheres
$\pi_{s,w}(k)$ over $k$, that is, $\SH_k(\Sigma^{s,w}\unit, \unit)$,
where $\unit$ denotes the motivic sphere spectrum over $k$.  When $k$
has finite mod 2 cohomological dimension and $s\geq w \geq 0$, the
motivic Adams spectral sequence (MASS) converges to the two-completion
of the stable motivic stems
\begin{equation*} 
  E_2^{f,(s,w)} =
  \Ext_{\calA^{**}}^{f,(s+f,w)}(H^{**},H^{**})
  \Longrightarrow
  (\pi_{s,w}\unit)\twocomp.
\end{equation*} 
This is a trigraded spectral sequence, where $\calA^{**}$ is the
bigraded mod $2$ motivic Steenrod algebra (see the work of Hoyois,
Kelly and {\O}stv{\ae}r \cite{HKOst} and Voevodsky \cite{MR2031199}),
and $H^{**}$ is the bigraded mod $2$ motivic cohomology ring of $k$.
A construction of the motivic Adams spectral sequence is given in
section \ref{section:MASS}.  The calculational challenges are to: (1)
identify the motivic Ext groups, (2) determine the differentials, and
(3) reconstruct the abutment from the filtration quotients.

Based on the MASS, Dugger and Isaksen have carried out calculations of
the $2$-complete stable motivic homotopy groups of spheres up to the
34-stem over the complex numbers \cite{DI}.  Isaksen has extended this
work largely up to the 70-stem \cite{Charts, StableStems}.  We are led
to wonder, how do the stable motivic homotopy groups vary for
different base fields?  Morel has given a complete description of the
$0$-line $\pi_{n,n}(k)$ in terms of Milnor-Witt $K$-theory
\cite{Morel12}.  The $1$ line $\pi_{n+1,n}(k)$ is determined by
Hermitian and Milnor $K$-theory groups by the work of R{\"o}ndigs,
Spitzweck and {\O}stv{\ae}r \cite{RSO}, which generalizes the partial
results obtained by Ormsby and {\O}stv{\ae}r in \cite{LowDimFields}.
Ormsby has investigated the case of related invariants over $p$-adic
fields \cite{MotInvPadic} and the rationals \cite{MotBPInvQ}, and
Dugger and Isaksen have analyzed the case over the real numbers
\cite{DI-Real}.  It is now possible to perform similar calculations
over fields of positive characteristic, thanks to work on the motivic
Steenrod algebra in positive characteristic by Hoyois, Kelly and
{\O}stv{\ae}r \cite{HKOst}.  In this paper we use computer assisted
motivic Ext group calculations in tandem with theoretical arguments to
determine stable motivic stems $\pi_{n,0}$ in weight zero over finite
fields.

We now state our main results. For a prime $\ell$ and an abelian group
$G$, we write the $\ell$-completion of $G$ by $G\ellcomp$.

\begin{theorem}
  \label{thm:alg_closed} %
  \todo{theorems now follow numbering convention given throughout} %
  Let $\Fbar$ be an algebraically closed field of positive
  characteristic $p$. For all $s\geq w \geq 0$ or $s<w$, there are
  isomorphisms $\pi_{s,w}(\Fbar)[p^{-1}] \cong
  \pi_{s,w}(\bbC)[p^{-1}]$.
\end{theorem}
\begin{proof}
  When $s> w \geq 0$, the groups $\pi_{s,w}(\Fbar)$ and
  $\pi_{s,w}(\bbC)$ are torsion by proposition \ref{prop:finiteness}.
  The isomorphism $\pi_{s,w}(\Fbar)[p^{-1}] \cong
  \pi_{s,w}(\bbC)[p^{-1}]$ follows when $s > w \geq 0$ from
  theorem \ref{thm:ellcomp_iso} by summing up the $\ell$-primary
  parts. When $s = w \geq 0$ the result follows by Morel's
  identification of the $0$ line in \cite{Morel12}. If $s<w$ then
  Morel's connectivity theorem implies that both groups are trivial by
  corollary \ref{sphere_connective}.
\end{proof}

Let $\pi_n^s$ denote the $n$th topological stable stem.  Over the
complex numbers, Levine showed there is an isomorphism $\pi_n^s \cong
\pi_{n,0}(\bbC)$ \cite[Corollary 2]{Levine}.  We obtain a similar result
over any algebraically closed field of positive characteristic $p$
after inverting $p$. %

\begin{corollary}
  \label{cor:invert_p_iso}
  Let $\Fbar$ be an algebraically closed field of positive
  characteristic $p$. For all $n\geq 0$ the homomorphism $\bbL c :
  (\pi_n^s)[p^{-1}] \to \pi_{n,0}(\Fbar)[p^{-1}]$ is an
  isomorphism.
\end{corollary} 

We do not expect Levine's theorem to hold over a field which is not
algebraically closed. Write $\bbF_q$ for the finite field with
$q=p^{\nu}$ elements where $p$ is a prime and $\widetilde{\bbF}_q$ for
the union of the field extensions $\bbF_{q^i}$ over $\bbF_q$ with $i$
odd. In this paper, we will see how the groups $\pi_{n,0}(\bbF_q)$
differ from $\pi_n^s$ using motivic Adams spectral sequence
calculations. Corollary \ref{cor:invert_p_iso} allows us to identify
differentials in the mod 2 motivic Adams spectral sequence over a
finite field and identify the two-complete groups
$\pi_{n,0}(\bbF_q)\twocomp$ in a range.  The analogous calculations
with the mod $\ell$ motivic Adams spectral sequence for $\ell$ an odd
prime are given by Wilson in \cite{WThesis}. The groups take the
following form.

\begin{theorem}
\label{thm:ellcomp_iso} 
If Morel's connectivity theorem holds for the finite field $\bbF_q$,
then for all $0\leq n\leq 18$ there is an isomorphism
\begin{equation*}
  \pi_{n,0}(\bbF_q)[p^{-1}] \cong (\pi_n^s \oplus
  \pi_{n+1}^s)[p^{-1}].
\end{equation*}
In particular, the group $\pi_{4,0}(\bbF_q)[p^{-1}]$ is trivial.
\end{theorem} 
\begin{proof}
  Propositions \ref{prop:1mod4} and \ref{prop:3mod4} calculate the
  two-completion of $\pi_{n,0}(\bbF_q)$ for $0 \leq n \leq 18$. For
  primes $\ell \neq 2$, the calculations are similar and given by
  Wilson in \cite[\S\S6--7]{WThesis}. The $\ell$-completions of
  $\pi_{n,0}(\bbF_q)$ are shown to agree with the $\ell$-primary part
  of $\pi_{n,0}(\bbF_q)$ for $n>0$ in proposition
  \ref{prop:finiteness}. When $n=0$, the result follows by Morel's
  identification of $\pi_{0,0}(\bbF_q)$ with the Grothendieck-Witt
  ring of $\bbF_q$, since $GW(\bbF_q) \cong \bbZ \oplus \bbZ/2$ as
  shown by Scharlau in \cite[Chapter 2, 3.3]{Scharlau}.
\end{proof}

\begin{remark}
  The above theorem depends on Morel's connectivity theorem to prove
  that the motivic Adams spectral sequence converges to the homotopy
  groups of the $\ell$-completion of the sphere spectrum. The
  published proof of the theorem by Morel in \cite{Morel12} holds for
  infinite fields. A private message from Panin gives a new proof of
  Morel's connectivity theorem which is valid for finite fields. We
  therefore state our results under the assumption that Morel's
  connectivity theorem holds for finite fields. However, our argument
  for theorem \ref{thm:ellcomp_iso} goes through with the field
  $\bbF_q$ replaced by $\widetilde{\bbF}_q$ where Morel's connectivity
  theorem holds by proposition \ref{prop:q-tilde}. The uneasy reader
  may replace $\bbF_q$ with $\widetilde{\bbF}_q$ throughout.
\end{remark}

In the case of a finite field $\bbF_q$ with $q \equiv 3 \bmod 4$, we
use the $\rho$-Bockstein spectral sequence to identify the additive
structure of the $E_2$ page of the MASS. Some hidden products in the
$\rho$-Bockstein spectral sequence were identified with the help of
computer calculations by Fu and Wilson, which can be found in
\cite{Fu-Wilson}.

It is interesting to note that the pattern $\pi_{n,0}(\bbF_q)\twocomp
\cong (\pi_n^s \oplus \pi_{n+1}^s)\twocomp$ obtained in theorem
\ref{thm:ellcomp_iso} does not hold in general.  We show that if $q
\equiv 1 \bmod 4$, then
\begin{equation*}
\pi_{19,0}(\bbF_q)\twocomp \cong (\bbZ/8 \oplus \bbZ/2) \oplus \bbZ/4
\text{ and } \pi_{20,0}(\bbF_q)\twocomp \cong \bbZ/8 \oplus \bbZ/2.
\end{equation*}
We shall leave open for further investigations the question of whether
or not an isomorphism $\pi_{n,0}(\bbF_q)\twocomp \cong (\pi_n^s \oplus
\pi_{n+1}^s)\twocomp$ holds when $q \equiv 3 \bmod 4$ and $n = 19,
20$.

\section*{Acknowledgments}  

The work on this paper was initiated during the program ``A Special
Semester in Motivic Homotopy Theory'' organized by Marc Levine in
2014; we gratefully acknowledge the hospitality and support from
Universit{\"a}t Duisburg-Essen. 

We thank the referee for several valuable suggestions and pointing out
an error in our calculation of the 19 and 20 stem. Thanks go to Panin
\todo{Check with Panin.}%
for private correspondence on Morel's stable connectivity theorem. The
first author would like to thank Charles Weibel, Dan Isaksen and Kyle
Ormsby for helpful conversations regarding this work. This paper
constitutes a portion of the first author's PhD thesis.  The second
author gratefully acknowledges support from the RCN project ``Special
Geometries,'' no.~239015 and ``Motivic Hopf Equations,'' no.~250399.%

\section{The stable motivic homotopy category}

We first sketch a construction of the stable motivic homotopy category
that will be convenient for our purposes, and in the process, set our
notation.  Treatments of stable motivic homotopy theory can be found
in Voevodsky \cite{Voev98}, Jardine \cite{Jardine}, Hu \cite{SMod},
Dundas, R{\"o}ndigs and {\O}stv{\ae}r \cite{MotFunctors}, Morel
\cite{Mor03}, Ayoub \cite{Ayoub2}, and the Nordfjordeid lectures
\cite{Nordfjordeid}. 

\subsection{The unstable motivic homotopy category}

A base scheme $S$ is a Noetherian separated scheme of finite Krull
dimension.  We write $\Sm/S$ for the category of smooth schemes of
finite type over $S$. A space over $S$ is a simplicial presheaf on
$\Sm/S$.  The collection of spaces over $S$ forms the category
$\Spc(S)$, where morphisms are natural transformations of functors.
We write $\Spc_{*}(S)$ for the category of pointed spaces.  

The first model category structure we endow $\Spc(S)$ with is the
projective model structure, see, for example, Blander
\cite[1.4]{Blander}, Dundas, R{\"o}ndigs and {\O}stv{\ae}r
\cite[2.7]{MotFunctors}, Hirschhorn \cite[11.6.1]{Hhorn}.

\begin{definition}
  A map $f : X \to Y$ in $\Spc(S)$ is a (global) weak equivalence if
  for any $U \in \Sm/S$ the map $f(U): X(U) \to Y(U)$ of simplicial
  sets is a weak equivalence.  The projective fibrations are those
  maps $f :X \to Y$ for which $f(U) : X(U) \to Y(U)$ is a Kan
  fibration for any $U \in \Sm/S$.  The projective cofibrations are
  those maps in $\Spc(S)$ which satisfy the left lifting property for
  trivial projective fibrations.  The projective model structure on
  $\Spc(S)$ consists of the global weak equivalences, the projective
  fibrations and the projective cofibrations.
\end{definition}

The category $\Spc(S)$ equipped with the projective model structure is
cellular, proper and simplicial; see Blander \cite[1.4]{Blander}.
Furthermore, $\Spc(S)$ has the structure of a simplicial monoidal
model category, with product $\times$ and internal hom
$\underline{\Hom}$.

The constant presheaf functor $c : \sSet \to \Spc(S)$ associates to a
simplicial set $A$ the presheaf $cA$ defined by $cA(U) = A$ for any $U
\in \Sm/S$.  The functor $c$ is a left Quillen functor when $\Spc(S)$
is equipped with the projective model structure.  Its right adjoint
$\Ev_S : \Spc(S) \to \sSet$ satisfies $\Ev_S(X) = X(S)$.  One can show
that representable presheaves and constant presheaves in $\Spc(S)$ are
cofibrant in the projective model structure.

For a smooth scheme $X$ over $S$, we write $h_X$ for the representable
presheaf of simplicial sets.  We will occasionally abuse notation and
write $X$ for $h_X$.  Although the representable pre-sheaf functor
embeds $\Sm/S$ into $\Spc(S)$, colimits which exist in $\Sm/S$ are not
necessarily preserved in $\Spc(S)$.  That is, if $X = \colim X_i$ in
$\Sm/S$, it need not be true that $h_X = \colim h_{X_i}$, for example,
$\colim (h_{\bbA^1} \leftarrow h_{\bbG_m} \to h_{\bbA^1}) \neq
h_{\bbP^1}$, as one can check by applying the Picard group functor.
To fix this, one introduces the Nisnevich topology on $\Sm/S$.

Morel and Voevodsky proved in \cite[3.1.4]{MV99} that the Nisnevich
topology is generated by covers coming from the elementary
distinguished squares. Recall that an elementary distinguished square
is a pull-back square in $\Sm/S$
\begin{equation*}
  \xymatrix@R=.5cm{ V' \ar[r] \ar[d] & X' \ar[d]^f \\ V
    \ar[r]^j & X }
\end{equation*} 
for which $f$ is an {\'e}tale map, $j$ is an open embedding and
$f^{-1}(X-V) \to X-V$ is an isomorphism, where these subschemes are
given the reduced structure.  Hence a presheaf of sets $F$ on $\Sm/S$
is a Nisnevich sheaf if and only if for any elementary distinguished
square the resulting square after applying $F$ is a pull-back square.

\begin{definition} 
  For a pointed space $\calX$ and $n\geq 0$, the $n$th simplicial
  homotopy sheaf $\pi_n\calX$ of $\calX$ is the Nisnevich
  sheafification of the presheaf $U \mapsto \pi_n(\calX(U))$.

  Write $W_{Nis}$ for the class of maps $f : \calX \to \calY$ for
  which $f_* : \pi_n \calX \to \pi_n \calY$ is an isomorphism of
  Nisnevich sheaves for all $n \geq 0$.  The Nisnevich local model
  structure on $\Spc_*(S)$ is the left Bousfield localization of the
  projective model structure with respect to $W_{Nis}$.
\end{definition}

\begin{definition}
  \label{motivic_model_cat} 
  Let $W_{\bbA^1}$ be the class of maps $\pi_{X}: (X \times \bbA^1)_+
  \to X_+$ for $X \in \Sm/S$.  The motivic model structure on
  $\Spc_*(S)$ is the left Bousfield localization of the projective
  model structure with respect to $W_{Nis} \cup W_{\bbA^1}$.  We write
  $\Spc_*^{\bbA^1}(S)$ for the category of pointed spaces
  equipped with the motivic model structure.  The pointed motivic
  homotopy category $\calH_*^{\bbA^1}(S)$ is the homotopy category of
  $\Spc_*^{\bbA^1}(S)$.

  For pointed spaces $\calX$ and $\calY$, we write $[\calX, \calY]$
  for the set of maps $\calH_*^{\bbA^1}(S)(\calX, \calY)$. The $n$th
  motivic homotopy sheaf of a pointed space $\calX$ over $S$ is the
  sheaf $\pi_n\calX$ associated to the presheaf $U \mapsto [S^n\wedge
  U_+, \calX]$.
\end{definition}

There are two circles in the category of pointed spaces: the constant
simplicial presheaf $S^1$ pointed at its $0$-simplex and the
representable presheaf $\bbG_m = \bbA^1 \setminus \{0\}$ pointed at
$1$.  These determine a bigraded family of spheres
$S^{i,j} = (S^1)^{\wedge i-j}\wedge \bbG_m^{\wedge j}$.

\begin{definition} 
  For a pointed space $X$ over $S$ and natural numbers $i$ and $j$
  with $i\geq j$, write $\pi_{i,j}X$ for the set of maps
  $[S^{i,j},X]$.
\end{definition}

The category of pointed spaces $\Spc_{*}(S)$ equipped with the induced
motivic model category structure has many good properties which make
it amenable to Bousfield localization.  In particular, $\Spc_*(S)$ is
closed symmetric monoidal, pointed simplicial, left proper and
cellular.

\subsection{The stable Nisnevich local model structure}

With the unstable motivic model category in hand, we now construct the
stable motivic model category using the general framework laid out by
Hovey in \cite{H-Spectra}.

Let $T$ be a cofibrant replacement of $\bbA^1/\bbA^1-\{0\}$.  Morel
and Voevodsky have shown that $T$ is weak equivalent to $S^{2,1}$ in
$\Spc_{*}^{\bbA^1}(S)$ \cite[3.2.15]{MV99}.  The functor $T \wedge -$
on $\Spc_*^{\bbA^1}(S)$ is a left Quillen functor, and we may invert
it by creating a category of $T$-spectra.

\begin{definition}
  \label{T-spectra} 
  A $T$-spectrum $X$ is a sequence of spaces $X_n \in
  \Spc_*^{\bbA^1}(S)$ equipped with structure maps $\sigma_n : T
  \wedge X_n \to X_{n+1}$.  A map of $T$-spectra $f: X \to Y$ is a
  collection of maps $f_n : X_n \to Y_n$ which are compatible with the
  structure maps.  We write $\Spt_T(S)$ for the category of
  $T$-spectra of spaces.
\end{definition}

To start, the level model category structure on $\Spt_T(S)$ is defined
by declaring a map $f : X \to Y$ to be a weak equivalence (respectively
fibration) if every map $f_n :X_n \to Y_n$ is a weak equivalence
(respectively fibration) in the motivic model structure on
$\Spc_*(S)$. The cofibrations for the level model structure are
determined by the left lifting property for trivial level fibrations.

\begin{definition} 
  Let $X$ be a $T$-spectrum.  For integers $i$ and $j$, the $(i,j)$
  stable homotopy sheaf of $X$, written as $\pi_{i,j}X$, is the
  Nisnevich sheafification of the presheaf $U \mapsto \colim_n
  \pi_{i+2n,j+n}X_n(U)$.  A map $f:X \to Y$ is a stable weak
  equivalence if for all integers $i$ and $j$ the induced maps $f_* :
  \pi_{i,j}X \to \pi_{i,j}Y$ are isomorphisms. %
\end{definition}
  
\begin{definition} 
  \label{def:stable_model}
  The stable model structure on $\Spt_T(S)$ is the model category
  where the weak equivalences are the stable weak equivalences and the
  cofibrations are the cofibrations in the level model structure.  The
  fibrations are those maps with the right lifting property with
  respect to trivial cofibrations.  We write $\SH_S$ for the homotopy
  category of $\Spt_T(S)$ equipped with the stable model structure.
\end{definition}

The stable model structure on $\Spt_T(S)$ can be realized as a left
Bousfield localization of the level-wise model structure, as defined
by Hovey \cite[3.3]{H-Spectra}.

Just as for the category $\Spt_{S^1}$ of simplicial $S^1$-spectra,
there is not a symmetric monoidal category structure on $\Spt_T(S)$
which lifts the smash product $\wedge$ in $\SH_S$.  One remedy is to
use a category of symmetric $T$-spectra $\Spt_T^{\Sigma}(S)$.  The
construction of this category is given by Hovey in
\cite[7.7]{H-Spectra} and Jardine in \cite{Jardine}.  It is proven in
\cite[9.1]{H-Spectra} that there is a zig-zag of Quillen equivalences
from $\symspt_T(S)$ to $\Spt_T(S)$, hence $\SH_S$ is equivalent to the
homotopy category of $\Spt_T^{\Sigma}(S)$ as well.  Since Quillen
equivalences induce equivalences of homotopy categories, the category
$\SH_S$ is a symmetric monoidal triangulated category with shift
functor $[1] = S^{1,0}\wedge -$.

\begin{definition}
  If $E$ is a $T$-spectrum over $S$, write $\pi_{i,j}E$ for
  $\SH_S(\Sigma^{i,j}\unit, E)$. In the case where $E = \unit$ and $S
  = \Spec(R)$ for a ring $R$, we simply write $\pi_{i,j}(R)$ for
  $\SH_S(\Sigma^{i,j}\unit,\unit)$.
\end{definition}

In addition to the category of $T$-spectra, we will find it convenient
to work with the category of $(\bbG_m, S^1)$-bispectra, see Jardine
\cite{Jardine} or the Nordfjordeid lectures \cite{Nordfjordeid}.

\begin{definition} 
  \label{def:s1_spectra}
  Consider the simplicial circle $S^1$ as a space over $S$ given by
  the constant presheaf. An $S^1$-spectrum over $S$ is a
  sequence of spaces $X_n \in \Spc_*(S)$ equipped with structure maps
  $\sigma_n : S^1 \wedge X_n \to X_{n+1}$.  A map of $S^1$-spectra
  over $S$ is a sequence of maps $f_n :X_n \to Y_n$ that are
  compatible with the structure maps.  The collection of $S^1$-spectra
  over $S$ with compatible maps between them forms a category
  $\Spt_{S^1}(S)$.

  First equip $\Spt_{S^1}(S)$ with the level model structure with
  respect to the Nisnevich local model structure on $\Spc_*(S)$. The
  $n$th stable homotopy sheaf of an $S^1$-spectrum $E$ over $S$ is the
  Nisnevich sheaf $\pi_n E = \colim \pi_{n+j} E_j$.  A map $f : E \to
  F$ of $S^1$-spectra over $S$ is a simplicial stable weak equivalence
  if for all $n\in \bbZ$ the induced map $f_* :\pi_n E \to \pi_n F$ is
  an isomorphism of sheaves. The stable Nisnevich local model category
  structure on $\Spt_{S^1}(S)$ is obtained by localizing at the class
  of simplicial stable equivalences, as in definition
  \ref{def:stable_model}.

  The motivic stable model category structure on $\Spt_{S^1}(S)$ is
  obtained from the simplicial stable model category structure by left
  Bousfield localization at the class of maps $W_{\bbA^1} = \{
  \Sigma^{\infty} X_{+}\wedge \bbA^1 \to \Sigma^{\infty} X_+ \st X \in
  \Sm/S\}$.  Write $\Spt_{S^1}^{\bbA^1}(S)$ for the motivic stable
  model category $L_{W_{\bbA^1}}\Spt_{S^1}(S)$ and write
  $\SH^{\bbA^1}_{S^1}(S)$ for its homotopy category.  The $n$th
  motivic stable homotopy sheaf of an $S^1$-spectrum $E$ is the
  Nisnevich sheaf $\pi_n^{\bbA^1} E$ associated to the presheaf $U
  \mapsto \SH^{\bbA^1}_{S^1}(S^n \wedge \Sigma^{\infty}U_+, E)$.
\end{definition}

\begin{definition} 
  In the projective model structure on $\Spc_*(S)$, the space $\bbG_m$
  pointed at $1$ is not cofibrant.  We abuse notation and write
  $\bbG_m$ for a cofibrant replacement of $\bbG_m$.  A $(\bbG_m,
  S^1)$-bispectrum over $S$ is a $\bbG_m$-spectrum of $S^1$-spectra.
  We write $\bispt(S)$ for the category of $(\bbG_m, S^1)$-bispectra
  over $S$.  Viewing $\bispt(S)$ as the category of $\bbG_m$-spectra
  of $S^1$-spectra, we first equip $\bispt(S)$ with the level model
  category structure with respect to the motivic stable model category
  structure on $\Spt_{S^1}(S)$.  The motivic stable model category
  structure on $\bispt(S)$ is the left Bousfield localization at the
  class of stable equivalences.
\end{definition}

There are left Quillen functors $\Sigma^{\infty}_{S^1} : \Spc_*(S) \to
\Spt_{S^1}(S)$ and $\Sigma^{\infty}_{\bbG_m} : \Spt_{S^1}(S) \to
\bispt(S)$.  Additionally, the category $\bispt(S)$ equipped with the
motivic stable model structure is Quillen equivalent to the stable
model category structure on $\Spt_T(S)$; see the Nordfjordeid lectures
\cite[page 216]{Nordfjordeid}.

\begin{definition} 
  To any spectrum of simplicial sets $E \in \Spt_{S^1}$ we may
  associate the constant $S^1$-spectrum $cE$ over $S$ with value $E$.
  That is, $cE$ is the sequence of spaces $cE_n$ with the evident
  bonding maps.  For a simplicial spectrum $E$, we also write $cE$ for
  the $(\bbG_m,S^1)$-bispectrum $\Sigma^{\infty}_{\bbG_m} cE$.  This
  defines a left Quillen functor $c : \Spt_{S^1} \to \bispt(B)$ with
  right adjoint given by evaluation at $S$.  Compare with Levine
  \cite[6.5]{Levine}.
\end{definition}

\subsection{Base change of stable model categories}
\label{sec:base_change}

\begin{definition} 
  Let $f : R \to S$ be a map of base schemes.  Pull-back along $f$
  determines a functor $f^{-1} : \Sm/S \to \Sm/R$ which induces
  Quillen adjunctions
  $(f^*, f_*) : \Spc^{\bbA^1}_*(S) \to \Spc^{\bbA^1}_*(R)$ and
  $(f^*, f_*) : \Spt_T(S) \to \Spt_T(R)$.
\end{definition}

We now discuss some of the properties of base change. A more thorough
treatment is given by Morel in \cite[\S5]{Morel-Con}. The map $f_*$
sends a space $\calX$ over $R$ to the space $\calX \circ f^{-1}$ over
$S$. The adjoint $f^{*}$ is given by the formula $(f^* \calY)(U) =
\colim_{U \to f^{-1}V} \calY(V)$. For a smooth scheme $X$ over $S$, a
standard calculation shows $f^*X = f^{-1}X$. Additionally, if $cA$ is
a constant simplicial presheaf on $\Sm/S$, it follows that $f^*(cA) =
cA$.

The Quillen adjunction $(f^*, f_*)$ extends to both the model category
of $T$-spectra and $(\bbG_m, S^1)$-bispectra by applying the maps
$f^*$, and respectively $f_*$, term-wise to a given spectrum. In the
case of $f^*$ for $T$-spectra, for instance, the bonding maps of
$f^*E$ are given by $T\wedge f^*E_n \cong f^*(T\wedge E_n) \to
f^*(E_{n+1})$ as $f^*T = T$. The same reasoning shows that the
adjunction $(f^*, f_*)$ extends to $(\bbG_m, S^1)$-bispectra.

Write $Q$ (respectively $R$) for the cofibrant (respectively fibrant)
replacement functor in $\Spt_T(S)$.  The derived functors $\bbL f^*$
and $\bbR f_*$ are given by the formulas $\bbL f^* = f^* Q$ and
$\bbR f_* = f_* R$.

Let $f : C \to B$ be a smooth map. The functor
$f_{\#} : \Sm/C \to \Sm/B$ sends  $\alpha: X \to C$ to
$f\circ\alpha: X \to B$ and induces a functor
$f_{\#} : \Spc^{\bbA^1}_*(B) \to \Spc^{\bbA^1}_*(C)$ by restricting a
presheaf on $\Sm/B$ to a presheaf on $\Sm/C$. The functor $f^*$ is
canonically equivalent to $f_{\#}$ on the level of spaces and spectra.

\subsection{The connectivity theorem}
\label{sec:connectivity_thm}

Morel establishes the connectivity of the sphere spectrum over
fields $F$ by studying the effect of Bousfield localization at
$W_{\bbA^1}$ of the stable Nisnevich local model category structure on
$\Spt_{S^1}(F)$ (see definition \ref{def:s1_spectra}).

An $S^1$-spectrum $E$ over $S$ is said to be simplicially
$k$-connected if for any $n\leq k$, the simplicial stable homotopy
sheaves $\pi_n E$ are trivial.  An $S^1$-spectrum $E$ is $k$-connected
if for all $n\leq k$ the motivic stable homotopy sheaves
$\pi_n^{\bbA^1} E$ are trivial.

\begin{theorem}[Morel's connectivity theorem]
  \label{thm:morel_connectivity}
  If $E$ is a simplicially $k$-connected $S^1$-spectrum over an
  infinite field $F$, then $E$ is also $k$-connected.
\end{theorem}

Morel's connectivity theorem has been proven when $F$ is an infinite
field in \cite{Morel12}, but the argument there does not hold for
finite fields. Private correspondence with Panin gives a new argument
to prove Morel's connectivity theorem for finite fields as well.

The connectivity theorem along with the work of Morel in
\cite[\S5]{Mor03} yield the following. This also follows from Voevodsky
\cite[4.14]{Voev98}.
\begin{corollary}
  \label{sphere_connective}
  Over a field $F$ where Morel's connectivity theorem holds, the
  sphere spectrum $\unit$ is $(-1)$-connected. In particular, for all
  $s-w<0$ the groups $\pi_{s,w}(F)$ are trivial.
\end{corollary}

\section{Comparison to the stable homotopy category}

The following result of Levine is crucial for our calculations
\cite[Theorem 1]{Levine}.
\begin{theorem}
  \label{levine_comparison}
  If $S = \Spec(\bbC)$, the functor $\bbL c : \SH \to \SH_S$ is fully
  faithful.
\end{theorem}

\begin{proposition}
  \label{const_comm} 
  Let $f : R \to S$ be a map of base schemes.  The following diagram
  of stable homotopy categories commutes.
  \begin{equation*} 
    \xymatrix@R=.5cm{
      &\SH \ar[dr]^{\bbL c} \ar[dl]_{\bbL c} & \\ \SH_S \ar[rr]^{\bbL
        f^*} && \SH_R
    }
  \end{equation*}
\end{proposition}

\begin{proof}
  The result follows by establishing $f^* \circ c = c$ on the level of
  model categories. For a constant space $cA \in \Spc(S)$, we have
  $f^* cA = cA$ by the calculation
  \begin{equation*}
    (f^*cA)(U) = \underset{U \to f^{-1}V}{\colim} cA(V) = A
  \end{equation*} 
  given the formula for $f^*$ in section \ref{sec:base_change}. As the
  base change map is extended to $T$-spectra by applying $f^*$
  term-wise, the claim follows.
\end{proof}

\begin{proposition} 
  \label{prop:char0_comparison}
  Let $S$ be a base scheme equipped with a map $\Spec(\bbC) \to S$.
  Then $\bbL c : \SH \to \SH_S$ is faithful.
\end{proposition}
\begin{proof} 
  For symmetric spectra $X$ and $Y$, the map $\bbL c : \SH(X,Y) \to
  \SH(\bbC)(cX,cY)$ factors through $\SH_S(cX, cY)$ by proposition
  \ref{const_comm}.  Theorem \ref{levine_comparison} implies that the map $\bbL
  c : \SH(X,Y) \to \SH_S(cX,cY)$ must be injective.
\end{proof}

\begin{corollary} 
  Write $W(\Fpbar)$ for the ring of Witt vectors of $\Fpbar$ and $K$
  for the fraction field of $W(\Fpbar)$ (see Serre \cite[II
  \S6]{Serre} for a definition).  The map $\bbL c : \pi_n^s \to
  \pi_{n,0}(W(\Fpbar))$ is an injection, because we have maps
  $W(\Fpbar) \to K \to \bbC$.
\end{corollary}

\section{Motivic cohomology}
\label{mot_coh}

Spitzweck has constructed a spectrum $H\bbZ$ in $\symspt_T(S)$ which
represents motivic cohomology $H^{a,b}(X; \bbZ)$ defined using Bloch's
cycle complex when $S$ is the Zariski spectrum of a Dedekind domain
\cite{Spitzweck}.  Spitzweck establishes enough nice properties of
$H\bbZ$ so that we may construct the motivic Adams spectral sequence
over general base schemes and establish comparisons between the
motivic Adams spectral sequence over a Hensel local ring in which
$\ell$ is invertible and its residue field.

\subsection{Integral motivic cohomology}

\begin{definition} 
  Over the base scheme $\Spec(\bbZ)$, the spectrum
  $H\bbZ_{\Spec(\bbZ)}$ is defined by Spitzweck in
  \cite[4.27]{Spitzweck}.  For a general base scheme $S$, we define
  $H\bbZ_S$ to be $f^* H\bbZ_{\Spec(\bbZ)}$ where $f : S \to
  \Spec(\bbZ)$ is the unique map.
\end{definition}

Let $S = \Spec(D)$ for $D$ a Dedekind domain.  For $X \in \Sm/S$,
Spitzweck shows there is a canonical isomorphism
$\SH_S(\Sigma^{\infty}X_+, \Sigma^{a,b}H\bbZ) \cong H^{a,b}(X;\bbZ)$,
where $H^{a,b}(-;\bbZ)$ denotes Levine's motivic cohomology defined
using Bloch's cycle complex \cite[7.19]{Spitzweck}.  The isomorphism
is functorial with respect to maps in $\Sm/S$.  Additionally, if $i :
\{s\} \to S$ is the inclusion of a closed point with residue field
$k(s)$, there is a commutative diagram for $X \in \Sm/S$.
\begin{equation*} 
  \xymatrix@R=.5cm{
    \SH_S(\Sigma^{\infty}X_+, \Sigma^{a, b}H\bbZ) \ar[r]^-{\cong}
    \ar[d] & H^{a,b}(X; \bbZ) \ar[d] \\ \SH(k(s))(\bbL i^*
    \Sigma^{\infty} X_+, \Sigma^{a,b} H\bbZ) \ar[r]^-{\cong} &
    H^{a,b}(\bbL i^* X;\bbZ)
  }
\end{equation*}

If the residue field $k(s)$ has positive characteristic, there is a
canonical isomorphism of ring spectra $\bbL i^* H\bbZ_S \cong
H\bbZ_{k(s)}$ by Spitzweck \cite[9.16]{Spitzweck}.  For a smooth map
of base schemes $f : R \to S$, there is an isomorphism $\bbL f^*
H\bbZ_S \cong H\bbZ_R$, because when $f$ is smooth we have $\bbL f^* =
f^*$, see Morel \cite[page 44]{Morel-Con}. It is then straightforward
to see that $f^{*}H\bbZ_S \cong H\bbZ_R$.

\subsection{Motivic cohomology with coefficients $\bbZ/\ell$}

For a prime $\ell$, write $H\bbZ/\ell$ for the cofiber of the map
$\ell\cdot: H\bbZ \to H\bbZ$ in $\SH_S$.  The spectrum $H\bbZ/\ell$
represents motivic cohomology with $\bbZ/\ell$ coefficients.  For a
smooth scheme $X$ over $S$, we write $H^{**}(X;\bbZ/\ell)$ for the
motivic cohomology of $X$ with $\bbZ/\ell$ coefficients.  When $S$ is
the Zariski spectrum of a ring $R$, we write $H^{**}(R;\bbZ/\ell)$ for
$H^{**}(\Spec(R); \bbZ/\ell)$.  We will frequently omit $\Spec$ from
our notation when the meaning is clear in other cases as well.

The now resolved Beilinson-Lichtenbaum conjecture allows us to
calculate the mod 2 motivic cohomology of a finite field $\bbF_q$ of
odd characteristic. In particular, there is an isomorphism
$H^{**}(\bbF_q ; \bbZ/2) \cong K^M_*(\bbF_q)/2[\tau]$ where $\tau$ has
bidegree $(0,1)$ and elements of $K^M_n(\bbF_q)/2$ have bidegree
$(n,n)$. The group $K^M_1(\bbF_q)/2 \cong
\bbF_q^{\times}/\bbF_q^{\times\,2}$ is isomorphic to $\bbZ/2$. We
write $u$ for the non-trivial element of
$\bbF_q^{\times}/\bbF_q^{\times\,2}$ and $\rho$ for the class of
$-1$. It is well known that $-1$ is a square in $\bbF_q$ if and only
if $q \equiv 1 \bmod 4$. Hence $H^{**}(\bbF_q ; \bbZ/2) \cong
\bbZ/2[\tau, u]/(u^2)$ and $u = \rho$ if and only if $q \equiv 3 \bmod
4$. 

The mod 2 Bockstein homomorphism $\beta$ is the motivic cohomology
operation given by the connecting homomorphism in the long exact
sequence of cohomology associated to the short exact sequence of
coefficient groups
\begin{equation*}
  0 \to \bbZ/2 \to \bbZ/4 \to \bbZ/2 \to 0.
\end{equation*}
The Bockstein is a cohomology operation of bidegree $(1,0)$. On the
mod 2 motivic cohomology of a finite field $\bbF_q$, the Bockstein is
determined by $\beta(\tau) = \rho$ and $\beta(u)=0$ as it is a
derivation. We remark that the Bockstein is trivial on the mod 2
motivic cohomology of a finite field $\bbF_q$ if and only if $q \equiv
1 \bmod 4$.

\begin{proposition}
  \label{prop:motcoh_dedekind}
  Let $D$ be a Hensel local ring in which $\ell$ is invertible. Write
  $F$ for the residue field of $D$ and write $\pi : D \to F$ for the
  quotient map. Then the map
  $\pi^* : H^{**}(D; \bbZ/\ell) \to H^{**}(F; \bbZ/\ell)$ is an
  isomorphism of $\bbZ/\ell$--algebras. Furthermore, the action of the
  Bockstein is the same in either case.
\end{proposition}

\begin{proof}
  The rigidity theorem for motivic cohomology in Geisser
  \cite[1.2(3)]{Geisser} gives the isomorphism. The map $\bbL\pi^*$
  gives comparison maps for the long exact sequences which define the
  Bockstein over $D$ and $F$. The rigidity theorem shows the long
  exact sequences are isomorphic, so the action of the Bockstein is
  the same in either case.
\end{proof}

\subsection{Mod 2 motivic cohomology operations and cooperations}

The mod $2$ motivic Steenrod algebra over a base scheme $S$, which we
write as $\calA^{**}(S)$, is the algebra of bistable mod $2$ motivic
cohomology operations. A bistable cohomology operation is a family of
operations $\theta_{**} : H^{**}(-;\bbZ/2) \to H^{*+a,*+b}(-;\bbZ/2)$
which are compatible with the suspension isomorphism for both the
simplicial circle $S^1$ and the Tate circle $\bbG_m$.

When $S$ is the Zariski spectrum of a characteristic $0$ field,
Voevodsky identified the structure of this algebra in \cite{MR2031198,
  MR2737977}. Voevodsky's calculation was extended to hold where the
base is the Zariski spectrum of a field of positive characteristic
$p\neq 2$ by Hoyois, Kelly and {\O}stv{\ae}r in \cite{HKOst}. In
particular, the algebra $\calA^{**}(S)$ is generated over $\bbF_2$ by
the Steenrod squaring operations $\Sq^i$ of bidegree $(i, \lfloor i/2
\rfloor)$ and the operations given by cup products $x\cup -$ where $x
\in H^{**}(S; \bbZ/2)$. The Steenrod squaring operations satisfy
motivic Adem relations, which are given by Voevodsky in
\cite[\S10]{MR2031198} (a minor modification is needed in the case
$a+b \equiv 1 \bmod 2$).

We record the structure of the mod $2$ dual Steenrod algebra
$\calA_{**}(\bbF_q)$ for a finite field $\bbF_q$ of characteristic
different from $2$ in the following proposition.
\begin{proposition}
  \label{prop:algebroid}
  Let $\bbF_q$ be a finite field of odd characteristic. The mod $2$
  dual Steenrod algebra is an associative commutative algebra of the
  following form.
  \begin{equation*}
    \calA_{**}(\bbF_q) \cong H_{**}(\bbF_q)[\tau_i, \xi_j \st i\geq 0, 
    j\geq 1]/(\tau_i^2 - \tau \xi_{i+1} - \rho \tau_{i+1} -
    \rho\tau_0\xi_{i+1})
  \end{equation*}
  Here $\tau_i$ has bidegree $(2^{i+1} -1, 2^i-1)$ and $\xi_i$ has
  bidegree $(2^{i+1}-2, 2^i -1)$. Note that if $q \equiv 1 \bmod 4$,
  the relation for $\tau_i^2$ simplifies to $\tau_i^2 = \tau
  \xi_{i+1}$ as $\rho = 0$. 

  The structure maps for the Hopf algebroid $(H_{**}(\bbF_q),
  \calA_{**}(\bbF_q))$, which we write simply as $(H_{**},
  \calA_{**})$, are as follows. %
  \begin{enumerate}[(a)]
  \item The left unit $\eta_L : H_{**} \to \calA_{**}$ is given by
    $\eta_L(x) = x$.
  \item The right unit $\eta_R : H_{**} \to \calA_{**}$ is determined
    as a map of $\bbZ/2$--algebras by $\eta_R(\rho) = \rho$ and
    $\eta_R(\tau) = \tau + \rho \tau_0$.  In the case where $\rho$ is
    trivial, that is, $q \equiv 1 \bmod 4$, the right and left unit agree
    $\eta_R = \eta_L$.
  \item The augmentation $\epsilon : \calA_{**} \to H_{**}$ kills
    $\tau_i$ and $\xi_i$, and for $x \in H_{**}$, it follows that
    $\epsilon(x) = x$.
  \item The coproduct $\Delta : \calA_{**} \to
    \calA_{**}\otimes_{H_{**}} \calA_{**}$ is a map of graded
    $\bbZ/2$--algebras determined by $\Delta(x) = x \otimes 1$ for $x
    \in H_{**}$, $\Delta(\tau_i) = \tau_i \otimes 1 + 1 \otimes \tau_i
    + \sum_{j= 0}^{i-1} \xi_{i-j}^{2^j} \otimes \tau_j$ and
    $\Delta(\xi_i) = \xi_i \otimes 1 + 1 \otimes \xi_i +
    \sum_{j=1}^{i-1} \xi_{i-j}^{2^j}\otimes \xi_j$.
  \item The antipode $c$ is a map of $\bbZ/2$--algebras determined by
    $c(\rho) = \rho$, $c(\tau) = \tau + \rho \tau_0$, $c(\tau_i) =
    \tau_i + \sum_{j=0}^{i-1} \xi_{i-j}^{2^j}c(\tau_j)$ and
    $c(\xi_i) = \xi_i + \sum_{j=1}^{i-1} \xi_{i-j}^{2^j}c(\xi_j)$.
  \end{enumerate}
\end{proposition}
\begin{proof}
  The calculation can be found in the work of Hoyois, Kelly and
  {\O}stv{\ae}r \cite{HKOst} and Voevodsky \cite{MR2031198}.
\end{proof}

We now investigate the structure of the Hopf algebroid of mod 2
cohomology cooperations over a Dedekind domain.

\begin{definition} 
  Let $D$ be a Dedekind domain, and let $C$ denote the set of
  sequences $(\epsilon_0, r_1, \epsilon_1, r_2, \ldots)$ with
  $\epsilon_i \in \{0,1\}$, $r_i \geq 0$ and only finitely many
  non-zero terms.  The elements $\tau_i \in
  \calA_{2^{i+1}-1,2^i-1}(D)$ and $\xi_i \in
  \calA_{2^{i+1}-2,2^i-1}(D)$ are constructed by Spitzweck in
  \cite[11.23]{Spitzweck}.  For any sequence $I = (\epsilon_0, r_1,
  \epsilon_1, r_2, \ldots)$ in $C$, write $\omega(I)$ for the element
  $\tau_0^{\epsilon_0}\xi_1^{r_1}\cdots$ and $(p(I),q(I))$ for the
  bidegree of the operation $\omega(I)$.
\end{definition}

Spitzweck calculates in \cite[11.24]{Spitzweck} that the dual Steenrod
algebra is generated by the elements $\tau_i$ and $\xi_j$ but does not
identify the relations for $\tau_i^2$. We record Spitzweck's
calculation in the following proposition.

\begin{proposition}
  \label{dual_steenrod_algebra} 
  Let $D$ be a Dedekind domain.  As an $H\bbZ/2$ module, there is a
  weak equivalence $\bigvee_{I\in B} \Sigma^{p(I),q(I)} H\bbZ/2 \to
  H\bbZ/2 \wedge H\bbZ/2$.  The map is given by $\omega(I)$ on the
  factor $\Sigma^{p(I),q(I)} H\bbZ/2$.
\end{proposition}

To obtain the relations for $\tau_i^2$, we find an analog of the
result of Voevodsky \cite[6.10]{MR2031198} when $D$ is a Hensel local
ring.

\begin{proposition}
  Let $D$ be a Hensel local ring in which $2$ is invertible and let
  $F$ denote the residue field of $D$. Then the following isomorphism
  holds.
  \begin{equation*}
    H^{**}(B\mu_{2}, \bbZ/2) \cong 
    H^{**}(D, \bbZ/2)[[u,v]]/(u^2 = \tau v + \rho u)
  \end{equation*}
  Here $v$ is the class $v_{2}\in H^{2,1}(B\mu_{2})$ defined by
  Spitzweck in \cite[page 81]{Spitzweck} and $u \in
  H^{1,1}(B\mu_{2};\bbZ/2)$ is the unique class satisfying
  $\tilde{\beta}(u) = v$, where $\tilde{\beta}$ is the integral
  Bockstein determined by the coefficient sequence $\bbZ \to \bbZ \to
  \bbZ/2$.
\end{proposition}

\begin{proof}
  The motivic classifying space $B\mu_2$ over $D$ (respectively $F$) fits
  into a triangle
  $B\mu_{2+} \to (\calO(-2)_{\bbP^{\infty}})_+ \to
  \mathrm{Th}(\calO(-2))$
  by \cite[(6.2)]{MR2031198} and \cite[(25)]{Spitzweck}.  From this
  triangle, we obtain a long exact sequence in mod $2$ motivic
  cohomology \cite[(6.3)]{MR2031198} and \cite[(26)]{Spitzweck}. The
  comparison map $\bbL\pi^* : \SH_D \to \SH_F$ induces a homomorphism
  of these long exact sequences. The rigidity theorem
  \ref{prop:motcoh_dedekind} and the 5-lemma then show that the
  comparison maps are all isomorphisms. As the desired relation holds
  in the motivic cohomology of $B\mu_{2}$ over $F$ and the choices of
  $u$ and $v$ are compatible with base change, the result follows.
\end{proof}

With this result, the relations $\tau_i^2 =\tau \xi_{i+1} + \rho
\tau_{i+1} + \rho \tau_0 \xi_{i+1}$ in $\calA_{**}(D)$ follow when $D$
is a Hensel local ring in which $2$ is invertible by the argument
given by Voevodsky in \cite[12.6]{MR2031198}. Furthermore, the
calculation of Spitzweck in \cite[11.23]{Spitzweck} shows that the
coproduct $\Delta$ is the same as in proposition
\ref{prop:algebroid}(d). The action of the Steenrod squaring
operations $H^{**}(D)$ and $H^{**}(F)$ agree by the naturality of
these cohomology operations, since these cohomology groups are
isomorphic. This shows that the right unit $\eta_R$ and the antipode
$c$ are given by the formulas in proposition
\ref{prop:algebroid}(b,e).

\begin{remark}
  \label{tau_xi_image} 
  Let $D$ be a Dedekind domain in which $2$ is invertible and consider
  the map $f : \bbZ[1/2] \to D$.  A key observation of Spitzweck in
  the proof of \cite[11.24]{Spitzweck} is that the map $\bbL
  f^*:\calA_{**}(\bbZ[1/2])\to\calA_{**}(D)$ satisfies $\bbL f^*
  \tau_i=\tau_i$ and $\bbL f^* \xi_i=\xi_i$ for all $i$.  For a map
  $j: D \to \widetilde{D}$ of Dedekind domains in which $2$ is
  invertible, it follows that $\bbL j^* \tau_i = \tau_i$ and $\bbL j^*
  \xi_i = \xi_i$ for all $i$.
\end{remark}

\begin{proposition}
  \label{prop:algebroid_iso}
  Let $D$ be a Hensel local ring in which $2$ is invertible and let
  $F$ denote the residue field of $D$. Then the comparison map
  $\pi^* : \calA_{**}(D) \to \calA_{**}(F)$ is an isomorphism of Hopf
  algebroids.
\end{proposition}

\begin{proof}
  Remark \ref{tau_xi_image} shows that the map $\pi^* : \calA_{**}(D)
  \to \calA_{**}(F)$ is an isomorphism of left
  $H_{**}(F)$ modules. The compatibility of the isomorphism with the
  coproduct, right unit and antipode was established above.
\end{proof}

The following definition is taken from Dugger and Isaksen
\cite[2.11]{DI}.

\begin{definition}
  \label{def:mot_finite}
  A set of bigraded objects $X = \{x_{(a, b)}\}$ is said to be
  motivically finite if for any bigrading $(a,b)$ there are only
  finitely many objects $y_{(a',b')} \in X$ for which $a \geq a'$ and
  $2b-a \geq 2b'-a'$. We say a bigraded algebra or module is
  motivically finite if it has a generating set which is motivically
  finite.
\end{definition}

To motivate the preceding definition, observe that if $H^{**}(X)$ is
a motivically finite $H^{**}(F)$ module, then $H^{**}(X)$ is a finite
dimensional $\bbF_{\ell}$ vector space in each bidegree.

For a Hensel local ring $D$, the isomorphism $\calA_{**}(D) \cong
\calA_{**}(F)$ of motivically finite algebras gives an isomorphism of
their duals $\calA^{**}(D) \cong \calA^{**}(F)$. See Hoyois, Kelly and
{\O}stv{\ae}r \cite[5.2]{HKOst} and Spitzweck \cite[11.25]{Spitzweck}
for the proof that the dual of the Hopf algebroid of cooperations is
the Steenrod algebra.

The analogous results of this section hold for mod $\ell$ motivic
cohomology over a base field or a Hensel local ring in which $\ell$ is
invertible for odd primes $\ell$. Precise statements can be found in
Wilson \cite{WThesis}.

\section{Motivic Adams spectral sequence}
\label{section:MASS}
 
The motivic Adams spectral sequence over a base scheme $S$ may be
defined using the appropriate notion of an Adams resolution; see Adams
\cite{Adams}, Switzer \cite{Switzer}, or Ravenel \cite{Ravenel} for
treatments in the topological case. We recount the definition for
completeness and establish some basic properties of the motivic Adams
spectral sequence under base change.  We follow Dugger and Isaksen
\cite[\S3]{DI} for the definition of the motivic Adams spectral
sequence.  See also the work of Hu, Kriz and Ormsby \cite[\S6]{HKO}.

Let $p$ and $\ell$ be distinct primes and let $q=p^{\nu}$ for some
integer $\nu\geq 1$. We will be interested in the specific case of the
motivic Adams spectral sequence over a field and over a Hensel
discrete valuation ring with residue field of characteristic $p$.  We
write $H$ for the spectrum $H\bbZ/\ell$ over the base scheme $S$ and
$H^{**}(S)$ for the motivic cohomology of $S$ with $\bbZ/\ell$
coefficients.  The spectrum $H$ is a ring spectrum and is cellular in
the sense of Dugger and Isaksen \cite{DI-Cell} by work of Spitzweck
\cite[11.4]{Spitzweck}.

\subsection{Construction of the mod $\ell$ MASS}

\begin{definition} 
  \label{def:standard_resolution}
  Consider a spectrum $X$ over the base scheme $S$ and let $\myol{H}$
  denote the spectrum in the cofibration sequence $\myol{H} \to \unit
  \to H \to \Sigma \myol{H}$.  The standard $H$-Adams resolution of
  $X$ is the tower of cofibration sequences $X_{f+1} \to X_f \to W_f$
  given by $X_f = \myol{H}^{\wedge f} \wedge X$ and $W_f = H \wedge
  X_f$; compare this with \cite[\S15]{Adams}.
  \begin{equation*}
    \xymatrix@R=.5cm{
      X_0 = X \ar[rd]_{j_0} && \myol{H}\wedge X \ar[ll]_{i_1}
      \ar[dr]_{j_1} && \myol{H}\wedge \myol{H} \wedge X
      \ar[ll]_{i_2} & \cdots \ar[l] \\ &H \wedge X
      \ar[ur]^{\bullet}_{\del_0} && H \wedge \myol{H} \wedge X
      \ar[ur]^{\bullet}_{\del_1} &&
    }
  \end{equation*}
\end{definition}

\begin{definition}
  \label{def:mass}
  Let $X$ be a $T$-spectrum over $S$ and let $\{X_f, W_f\}$ be the
  standard $H$-Adams resolution of $X$.  The motivic Adams spectral
  sequence for $X$ with respect to $H$ is the spectral sequence
  determined by the following exact couple.
  \begin{equation*}
    \xymatrix@R=.5cm{
      \oplus \pi_{**}X_f \ar[rr]^{i_*}& & \oplus \pi_{**} X_f
      \ar[dl]^{j_*} \\ &\oplus \pi_{**} W_f \ar[ul]^{\del_*}&
    }
  \end{equation*} 
  The $E_1$ term of the motivic Adams spectral sequence is
  $E_1^{f,(s,w)} = \pi_{s,w}W_f$.  The index $f$ is called the Adams
  filtration, $s$ is the stem and $w$ is the motivic weight.  The
  Adams filtration of $\pi_{**}X$ is given by $F_i\pi_{**}X =
  \im(\pi_{**}X_i \to \pi_{**}X)$.
\end{definition}

\begin{proposition} 
  \label{prop:functoriality}
  Let $\frakS$ denote the category of spectral sequences in the
  category of abelian groups.  The associated spectral sequence to the
  standard $H$-Adams resolution defines a functor
  $\MASS : \SH_S \to \frakS$.  Furthermore, the motivic Adams spectral
  sequence is natural with respect to base change.
\end{proposition}

\begin{proof} 
  The construction of the standard $H$-Adams resolution is functorial
  because $\SH_S$ is symmetric monoidal.  Given $X \to X'$ we get
  induced maps of standard $H$-Adams resolutions
  $\{ X_f, W_f \} \to \{X_f', W_f'\}$.  As $\pi_{**}(-)$ is a
  triangulated functor, we get an induced map of the associated exact
  couples and hence of spectral sequences $\MASS(X) \to \MASS(X')$.

  Let $f:R \to S$ be a map of base schemes.  The claim is that there
  is a natural transformation between $\MASS : \SH_S \to \frakS$ and
  $\MASS \circ \bbL f^* : \SH_S \to \SH_R \to \frakS$.  Let $X\in
  \SH_S$ and let $\{X_f, W_f \}$ be the standard $H_S$-Adams
  resolution of $X$ in $\SH_S$.  We may as well assume $X$ is
  cofibrant, in which case $QX = X$ where $Q$ is the cofibrant
  replacement functor.  Let $\{X_f', W_f'\}$ denote the standard
  $H_R$-Adams resolution of $\bbL f^*X = f^*X$.  Observe that
  $\{f^*X_f , f^*W_f\} = \{X_f',W_f'\}$, since $f^* \unit = \unit$,
  $f^* H_S = H_R$ and $\bbL f^*$ is a monoidal functor.  We therefore
  have a map $\{\bbL f^* X_f, \bbL f^* W_f \} \to \{ X_f',W_f'\}$.
  Applying $\bbL f^* : \SH_S(\Sigma^{s,w}\unit,-) \to
  \SH_R(\Sigma^{s,w}\unit, \bbL f^* -)$ to $\{X_f, W_f\}$ gives a map
  of exact couples and therefore a map $\Phi_X : \MASS_S(X) \to
  \MASS_R(\bbL f^* X)$.  It is straightforward to verify that $\Phi$
  determines a natural transformation.
\end{proof}

\begin{corollary} 
  For a map of base schemes $f : R \to S$, there is a map of motivic
  Adams spectral sequences $\Phi : \MASS_S(\unit) \to \MASS_R(\unit)$.
  The map $\Phi$ is furthermore compatible with the induced map
  $\pi_{**}(S) \to \pi_{**}(R)$.
\end{corollary}

\begin{definition}
  \label{def:cellular}
  A particularly well behaved family of spectra in $\SH_S$ are the
  cellular spectra in the sense of Dugger and Isaksen
  \cite[2.10]{DI-Cell}. A spectrum $E \in \SH_S$ is cellular if it can
  be constructed out of the spheres $\Sigma^{\infty}S^{a,b}$ for any
  integers $a$ and $b$ by homotopy colimits. A cellular spectrum is of
  finite type if for some $k$ it has a cell decomposition with no
  cells $S^{a,b}$ for $a-b<k$ and at most finitely many cells
  $S^{a,b}$ for any $a$ and $b$, see Hu, Kriz and Ormsby
  \cite[\S2]{HKO}.
\end{definition}

In the following proposition, $\Ext$ is taken in the category of
$\calA_{**}$--comodules. The homological algebra of comodules is
investigated thoroughly in Adams \cite{Adams}, Switzer \cite{Switzer}
and Ravenel \cite{Ravenel}.

\begin{proposition}
  \label{prop:homological_E2}
  Suppose $X$ is a cellular spectrum over the base scheme $S$. The
  motivic Adams spectral sequence for $X$ has $E_2$ page given by
  \begin{equation*}
    E_2^{f,(s,w)} 
    \cong \Ext_{\calA_{**}(S)}^{f, (s+f,w)}(H_{**}S, H_{**}X). 
  \end{equation*}
  with differentials $d_r : E_r^{f,(s,w)} \to E_r^{f+r,(s-1,w)}$ for
  $r\geq 2$.  
\end{proposition}

\begin{proof}
  Spitzweck proves that $H$ is a cellular spectrum in
  \cite[11.4]{Spitzweck}.  The argument given for \cite[7.10]{DI} by
  Dugger and Isaksen then goes through. The cellularity of $X$ and $H$
  is sufficient to ensure that the K{\"u}nneth theorem holds, which is
  needed in the argument.
\end{proof}

\begin{corollary} 
  If $X$ and $X'$ are cellular spectra over $S$ and $X \to X'$ induces
  an isomorphism $H_{**} X \to H_{**}X'$, then the induced map
  $\MASS(X) \to \MASS(X')$ is an isomorphism of spectral sequences
  from the $E_2$ page onwards.
\end{corollary}

\begin{corollary}
  \label{mass_iso}
  Let $f : R \to S$ be a map of base schemes and consider a cellular
  spectrum $X$ over $S$.  Suppose $f^* : H_{**}(S) \to H_{**}(R)$,
  $f^* : \calA_{**}(S) \to \calA_{**}(R)$ and $ f^* : H_{**}X \to
  H_{**} (\bbL f^* X)$ are all isomorphisms.  Then $\MASS_S(X) \to
  \MASS_R(\bbL f^* X)$ is an isomorphism of spectral sequences from
  the $E_2$ page onwards.
\end{corollary}

\begin{corollary}
  \label{prop:E2iso}
  Let $D$ be a Hensel local ring in which $\ell$ is invertible and
  write $F$ for the residue field of $D$. Then the comparison map
  $\MASS(D) \to \MASS(F)$ is an isomorphism at the $E_2$ page.
\end{corollary}

\begin{proof}
  Propositions \ref{prop:motcoh_dedekind}, \ref{prop:algebroid_iso}
  and corollary \ref{mass_iso} give the result when $X = \unit$.
\end{proof}

\subsection{Convergence of the motivic Adams spectral sequence}

To simplify the notation, write $\Ext(R)$ for
$\Ext_{\calA^{**}(R)}(H^{**}(R), H^{**}(R))$ when working over the
base scheme $S=\Spec(R)$. For any abelian group $G$ and any prime
$\ell$, we write $G_{(\ell)}$ for the $\ell$-primary part of $G$ and
$G\ellcomp = \InverseLimit G/\ell^{\nu}$ for the $\ell$-completion of
$G$. If $\{X_f, W_f\}$ is the standard $H$-Adams resolution of a
spectrum $X$, the $H$-nilpotent completion of $X$ is the spectrum
$X\Hcomp = \holim_f X/X_f$ defined by Bousfiled in
\cite[\S5]{Bousfield}. The $H$-nilpotent completion has a tower given
by $C_i = \holim_f (X_i/X_f)$.

\begin{proposition}
  \label{prop:convergence}
  Let $S$ be the Zariski spectrum of a field $F$ with characteristic
  $p \neq \ell$ and let $X$ be a cellular spectrum $X$ over $S$ of
  finite type (definition \ref{def:cellular}). If either $\ell>2$ and
  $F$ has finite mod $\ell$ cohomological dimension, or $\ell =2$ and
  $F[\sqrt{-1}]$ has finite mod $2$ cohomological dimension, the
  motivic Adams spectral sequence converges to the homotopy groups of
  the $H$-nilpotent completion of $X$
  \begin{equation*}
    E_2^{f,(s,w)} \Rightarrow \pi_{s,w}(X\Hcomp).
  \end{equation*}
  Furthermore, there is a weak equivalence $X\Hcomp \cong X\ellcomp$.
\end{proposition}

\begin{proof}
  The argument given by Hu, Kriz and Ormsby in \cite{HKO}, which
  requires Morel's connectivity theorem for $F$, carries over to the
  positive characteristic case from the work of Hoyois, Kelly and
  {\O}stv{\ae}r \cite{HKOst}.  See Ormsby and {\O}stv{\ae}r
  \cite[3.1]{LowDimFields} for the analogous argument for the motivic
  Adams-Novikov spectral sequence.
\end{proof}

We say a line $s = mf + b$ in the $(f,s)$-plane is a vanishing line
for a bigraded group $G^{f,s}$ if $G^{f,s}$ is zero whenever $0 < s <
mf + b$.

\begin{proposition}
  \label{prop:vanishing_line}
  If $\Fbar$ is an algebraically closed field of characteristic
  $p\neq \ell$, then a vanishing line for
  $\Ext^{**}(\Fbar)\cong \Ext^{**}(W(\Fbar))$ at the prime $\ell$ is
  $s = (2\ell - 3)f$. If $\bbF_q$ is a finite field of characteristic
  $p\neq \ell$, then a vanishing line for
  $\Ext^{**}(\bbF_q)\cong \Ext^{**}(W(\bbF_q))$ at the prime $\ell$ is
  $s = (2\ell - 3)f - 1$.
\end{proposition}

\begin{proof}
  A vanishing line exists for $\Ext(\Fbar)\cong \Ext(W(\Fbar))$ when
  $\Fbar$ is an algebraically closed fields by comparison with $\bbC$
  and the topological case by work of Dugger and Isaksen
  \cite{DI}. The vanishing line $s = f(2\ell -3)$ from topology by
  Adams \cite{A61} is therefore a vanishing line for $\Ext(\Fbar)
  \cong \Ext(W(\Fbar))$.

  For a finite field $\bbF_q$, the line $ s = f(2\ell - 3) - 1$ is a
  vanishing line for $\Ext(\bbF_q)\cong \Ext(W(\bbF_q))$ by the
  identification of the $E_2$ page of the motivic Adams spectral
  sequence. When $\ell=2$ this is given in proposition
  \ref{prop:e2_1mod4} when $q \equiv 1 \bmod 4$ and the calculation of
  the $\rho$-BSS when $q\equiv 3 \bmod 4$. For odd $\ell$, see Wilson
  \cite{WThesis}.
\end{proof}

We now discuss the convergence of the motivic Adams spectral sequence
over the ring of Witt vectors associated to a finite field or an
algebraically closed field. Consult Serre \cite[II \S 6]{Serre} for a
construction of the ring of Witt vectors associated to a field of
positive characteristic.

\begin{proposition}
  \label{prop:conv_hat}
  Let $W(F)$ be the ring of Witt vectors of a field $F$ that is either
  a finite field or an algebraically closed field of characteristic
  $p$ and let $\ell$ be a prime different from $p$.  The motivic Adams
  spectral sequence for $\unit$ over $W(F)$ converges to
  $\pi_{**}(\unit\Hcomp)$ filtered by the Adams filtration, where
  $\unit\Hcomp$ is the $H$-nilpotent completion of $\unit$.
\end{proposition}

\begin{proof}
  The convergence $\MASS_{W(F)}(\unit) \Rightarrow
  \pi_{**}(\unit\Hcomp)$ follows by the argument given by Dugger and
  Isaksen \cite[7.15]{DI} given the vanishing line in the motivic
  Adams spectral sequence by proposition \ref{prop:vanishing_line}.
\end{proof}

\begin{proposition}
  \label{prop:Hcomp_compatibility}
  Let $R$ and $S$ be base schemes for which the motivic Adams spectral
  sequence for $\unit$ converges to $\pi_{**}(\unit\Hcomp)$; see
  propositions \ref{prop:convergence} and \ref{prop:conv_hat} for
  examples. A map of base schemes $f: R \to S$ yields a comparison map
  $\MASS_S(\unit\Hcomp) \to \MASS_R(\unit\Hcomp)$ which is compatible
  with the induced map
  \begin{equation*}
    \pi_{**}(\unit\Hcomp(S)) \to \pi_{**}(\bbL f^* \unit\Hcomp(S)) \to
    \pi_{**}(\unit\Hcomp(R)).
  \end{equation*}
\end{proposition}

\begin{proof}
  Let $\{X_f(S),W_f(S)\}$ denote the standard $H$-Adams resolution of
  $\unit$ over $S$. We now construct a map $\pi_{**}(\unit\Hcomp(S))
  \to \pi_{**}(\unit\Hcomp(R))$.  Recall from proposition
  \ref{prop:functoriality} that $f^* X_f(S) = X_f(R)$.  Since $\bbL
  f^*$ is a triangulated functor, there are maps $\bbL f^*
  (\unit/X_f(S)) \to \unit/X_f(R)$ and so a map $\bbL f^*
  \unit\Hcomp(S) \to \unit\Hcomp(R)$ by the universal property for
  $\unit\Hcomp(R) = \holim \unit/X_f(R)$.  Write $C_i(S)$ for the
  tower of $\unit\Hcomp(S)$ over $S$ defined above (and in Bousfield
  \cite[\S5]{Bousfield}). Similar considerations give a map of towers
  $\bbL f^* C_i(S) \to C_i(R)$.  Hence $\MASS_S(\unit\Hcomp) \to
  \MASS_R(\unit\Hcomp)$ is compatible with the induced map
  $\pi_{**}(\unit\Hcomp(S)) \to \pi_{**}(\unit\Hcomp(R))$.
\end{proof}

\begin{proposition}
  \label{prop:finiteness} 
  Let $F$ be a field of characteristic $p$ with finite mod $\ell$
  cohomological dimension for all primes $\ell\neq p$ and suppose
  $H^{s,w}(F;\bbZ/\ell)$ is a finite dimensional vector space over
  $\bbF_{\ell}$ for all $s$ and $w$. Furthermore, assume that the mod
  $\ell$ motivic Adams spectral sequence for $\unit$ over $F$ has a
  vanishing line, such as when $F$ is a finite field or an
  algebraically closed field. Then the $\ell$-primary part of
  $\pi_{s,w}(F)$ is finite whenever $s>w\geq 0$.
\end{proposition}

\begin{proof}
  Ananyevsky, Levine and Panin show that the groups $\pi_{s,w}(F)$
  are torsion for $s>w\geq 0$ in \cite{ALP}. It follows that the group
  $\pi_{s,w}(F)$ is the sum of its $\ell$-primary subgroups
  $\pi_{s,w}(F)_{(\ell)}$. We set out to show that
  $\pi_{s,w}(F)_{(\ell)}$ is finite when $\ell\neq p$.

  The motivic Adams spectral sequence converges to
  $\pi_{**}(\unit\ellcomp)$ by proposition \ref{prop:convergence}
  (this requires Morel's connectivity theorem). The vanishing line in
  the motivic Adams spectral sequence shows that the Adams filtration
  of $\pi_{s,w}(\unit\ellcomp)$ has finite length, and as each group
  $E_2^{f,(s,w)}$ is a finite dimensional $\bbF_{\ell}$ vector space,
  we conclude the groups $\pi_{s,w}(\unit\ellcomp)$ are finite.  From
  the long exact sequence of homotopy groups associated to the
  triangle $\unit\ellcomp \to \prod \unit/\ell^{\nu} \to \prod
  \unit/\ell^{\nu}$ defining $\unit\ellcomp$, we extract the following
  short exact sequence of finite groups.
  \begin{equation}
    \label{eq:finite}
    0 \to \InverseLimit^{1} \pi_{s+1,w}(\unit/\ell^{\nu}) 
    \to \pi_{s,w}(\unit \ellcomp) 
    \to \InverseLimit \pi_{s,w}(\unit/\ell^{\nu}) \to 0
  \end{equation}
  Similarly, from the triangles
  $\unit \xrightarrow{\ell^{\nu}\cdot} \unit \to \unit/\ell^{\nu}$ we
  extract the short exact sequences
  \begin{equation*}
    0 \to \pi_{s,w}(\unit)/\ell^{\nu} 
    \to \pi_{s,w}(\unit/\ell^{\nu}) 
    \to {}_{\ell^{\nu}}\pi_{s-1,w}(\unit) \to 0,
  \end{equation*}
  which form a short exact sequence of towers. The maps in the tower
  $\{ \pi_{s,w}(\unit)/\ell^{\nu} \}$ are given by the reduction maps
  $\pi_{s,w}(\unit)/\ell^{\nu} \to \pi_{s,w}(\unit)/\ell^{\nu-1}$.
  Since the tower $\{ \pi_{s,w}(\unit)/\ell^{\nu} \}$ satisfies the
  Mittag-Leffler condition, we have $\InverseLimit^{1}
  \pi_{s,w}(\unit)/\ell^{\nu} = 0$.  The associated long exact
  sequence for the inverse limit gives the exact sequence
  \begin{equation}
    \label{eq:lim_sequence}
    0 \to \pi_{s,w}(\unit)\ellcomp
    \to \InverseLimit \pi_{s,w}(\unit/\ell^{\nu})
    \to \InverseLimit {}_{\ell^{\nu}}\pi_{s-1,w}(\unit) \to 0.
  \end{equation}
  The group $\InverseLimit {}_{\ell^{\nu}}\pi_{s-1,w}(\unit)$ is the
  $\ell$-adic Tate module of $\pi_{s-1,w}(\unit)$, which is
  torsion-free. Since $\InverseLimit \pi_{s,w}(\unit/\ell^{\nu})$ is
  finite by \ref{eq:finite}, the map $\InverseLimit
  \pi_{s,w}(\unit/\ell^{\nu}) \to \InverseLimit
  {}_{\ell^{\nu}}\pi_{s-1,w}(\unit)$ is trivial. But since the
  sequence \ref{eq:lim_sequence} is exact, the group $\InverseLimit
  {}_{\ell^{\nu}}\pi_{s-1,w}(\unit)$ is trivial,
  $\pi_{s,w}(\unit)\ellcomp \cong \InverseLimit
  \pi_{s,w}(\unit/\ell^{\nu})$ and $\pi_{s,w}(\unit)\ellcomp$ is
  finite.

  Write $K(i)$ for the kernel of the canonical map
  $\pi_{s,w}(\unit)\ellcomp \to \pi_{s,w}(\unit)/\ell^{i}$.  The tower
  $\cdots K(i) \subseteq K(i-1) \subseteq \cdots \subseteq K(1)$
  consists of finite groups and so it must stabilize. Hence the tower
  \begin{equation*}
    \cdots \to \pi_{s,w}(\unit)/\ell^{\nu} 
    \to \pi_{s,w}(\unit)/\ell^{\nu-1} 
    \to   \cdots \to \pi_{s,w}(\unit)/\ell
  \end{equation*}
  must also stabilize. There is then some $N$ for which
  ${\ell}^{N}\pi_{s,w}(\unit) = \ell^{\nu}\pi_{s,w}(\unit)$ for all
  $\nu \geq N$, and so $\ell^{N}\pi_{s,w}(\unit)$ is
  $\ell$-divisible. From the short exact sequence of towers
  $\ell^{\nu} \pi_{s,w}(\unit) \to \pi_{s,w}(\unit)\to
  \pi_{s,w}(\unit)/\ell^{\nu}$,
  taking the inverse limit yields the exact sequence
  \begin{equation*}
    0 \to \ell^{N}\pi_{s,w}(\unit)
    \to \pi_{s,w}(\unit)
    \to \pi_{s,w}(\unit)\ellcomp \to 0.
  \end{equation*}
  Since $\pi_{s,w}(\unit)\ellcomp$ is finite, it is $\ell$-primary and
  there is a short exact sequence
  \begin{equation*}
    0 \to \ell^{N}\pi_{s,w}(\unit)_{(\ell)}
    \to \pi_{s,w}(\unit)_{(\ell)}
    \to \pi_{s,w}(\unit)\ellcomp \to 0.
  \end{equation*}

  The group $\ell^{N}\pi_{s,w}(\unit)_{(\ell)}$ must be zero. Suppose
  for a contradiction that it is non-zero. Then
  $\ell^{N}\pi_{s,w}(\unit)_{(\ell)}$ must contain
  $\bbZ/\ell^{\infty}$ as a summand, which shows the $\ell$-adic Tate
  module of $\pi_{s,w}(\unit)$ is non-zero---a contradiction.
\end{proof}

We now identify the groups $\pi_{s,s}(\unit\ellcomp)$ for $s\geq 0$.

\begin{proposition}
  \label{prop:0_stem}
  Let $F$ be a finite field or an algebraically closed field of
  characteristic $p \neq \ell$. When $s=w\geq 0$ or $s<w$, the motivic
  Adams spectral sequence of $\unit$ over $F$ converges to the
  $\ell$-completion of $\pi_{s,w}(F)$. %
\end{proposition}

\begin{proof}
  If $s<w$, the convergence follows from Morel's connectivity theorem.
  When $s=w \geq 0$, proposition \ref{prop:convergence} implies that
  at bidegree $(s,w)$ the motivic Adams spectral sequence converges to
  the group $\pi_{s,w}(\unit\ellcomp)$.  Since $\pi_{s-1,s}(\unit)=0$
  by Morel's connectivity theorem, the short exact sequence (see, for
  example, Hu, Kriz, Ormsby \cite[(2)]{HKO})
  \begin{equation*}
    0\to \Ext(\bbZ/\ell^{\infty},\pi_{s,s}(\unit)) \to %
    \pi_{s,s}(\unit\ellcomp) \to %
    \Hom(\bbZ/\ell^{\infty}, \pi_{s-1,s}(\unit)) \to 0
  \end{equation*}
  gives an isomorphism $\Ext(\bbZ/\ell^{\infty},\pi_{s,s}(\unit))
  \cong \pi_{s,s}(\unit\ellcomp)$.  In \cite[1.25]{Morel12}, Morel has
  calculated $\pi_{0,0}(F) \cong GW(F)$ and $\pi_{s,s}(F)\cong W(F)$
  for $s>0$ where $W(F)$ is the Witt group of the field $F$. For the
  fields under consideration, $GW(F)$ and $W(F)$ is a finitely
  generated abelian group. But for any finitely generated abelian
  group $A$, there is an isomorphism $\Ext(\bbZ/\ell^{\infty}, A)
  \cong A\ellcomp$, given in Bousfield and Kan
  \cite[Chapter VI\S2.1]{BousfieldKan}, which concludes the proof.
\end{proof}

\section{Stable stems over an algebraically closed field}

Let $\Fbar$ be an algebraically closed field of positive
characteristic $p$. We write $W=W(\Fbar)$ for the ring of Witt vectors
of $\Fbar$, $K = K(\Fbar)$ for the field of fractions of $W$ and
$\Kbar = \Kbar(\Fbar)$ for the algebraic closure of $K$.  Note that
$K$ is a field of characteristic $0$.  The previous sections have set
us up with enough machinery to compare the motivic Adams spectral
sequences at a prime $\ell \neq p$ over the associated base schemes
$\Spec(\Fbar)$, $\Spec(W)$ and $\Spec(\Kbar)$.  We will often write
the ring instead of the Zariski spectrum of the ring in our
notation. For any Dedekind domain $R$, we write $\Ext(R)$ for the
trigraded ring $\Ext_{\calA^{**}(R)}(H^{**}(R),H^{**}(R))$.

\begin{proposition}
  \label{prop:Fpbar_iso}
  Let $\Fbar$ be an algebraically closed field of positive
  characteristic $p$, and let $\ell$ be a prime different from
  $p$. The $E_2$ page of the mod $\ell$ motivic Adams spectral
  sequence for $\unit$ over $W$, the ring of Witt vectors of $\Fbar$,
  is given by
\begin{equation*}
  E_2^{f,(s,w)}(W) \cong \Ext^{f,(s+f, w)}(W) \cong
  \Ext^{f,(s+f,w)}(\Fbar).
\end{equation*}
\end{proposition}
\begin{proof}
  Since $W$ is a Hensel local ring with residue field $\Fbar$,
  proposition \ref{prop:E2iso} applies.
\end{proof}

\begin{proposition}
  \label{kbar_iso} 
  Let $\Fbar$ be an algebraically closed field of characteristic
  $p$. The homomorphism $f : W \to \Kbar$ induces isomorphisms of
  graded rings $ f^* : H_{**}(W) \to H_{**}(\Kbar)$ and
  $ f^* : \calA_{**}(W) \to \calA_{**}(\Kbar)$.
\end{proposition}

\begin{proof} 
  It suffices to establish isomorphisms for motivic cohomology, as
  $H^{**}(S) \cong H_{-*,-*}(S)$.  Since $H^{**}(W)\cong
  H^{**}(\Fpbar)$, we have $H^{**}(W)\cong \bbF_{\ell}[\tau]$ where
  $\tau \in H^{0,1}(W)\cong \mu_{\ell}(W)$. We also have
  $H^{**}(\Kbar)\cong \bbF_{\ell}[\tau]$. To identify the ring map
  $f^* : H^{**}(W) \to H^{**}(R)$ it suffices to identify the value of
  $f^*(\tau)$. The homomorphism $f^* : H^{0,1}(W) \to H^{0,1}(\Kbar)$
  may be identified with $\mu_{\ell}(W) \to \mu_{\ell}(\Kbar)$, which
  is an isomorphism. Hence $f^* : H^{**}(W)\to H^{**}(\Kbar)$ is an
  isomorphism.  The argument given for proposition
  \ref{prop:algebroid_iso} establishes that $j^* : \calA_{**}(W) \to
  \calA_{**}(\Kbar)$ is an isomorphism.
\end{proof}

\begin{corollary}
  \label{E2_iso} 
  Let $\Fbar$ be an algebraically closed field of characteristic
  $p$. The homomorphisms $W \to \Kbar$ and $W \to \Fbar$ induce
  isomorphisms of motivic Adams spectral sequences for $\unit$ from
  the $E_2$ page onwards. In particular,
  $\Ext(\Fbar)\cong \Ext(W) \cong \Ext(\Kbar)$. %
\end{corollary}

\begin{lemma}
  \label{lem:char-0}
  Let $f: \kbar \to \Kbar$ be an extension of algebraically closed
  fields of characteristic $0$. For all $s$ and $w \geq 0$, base
  change induces an isomorphism $\pi_{s,w}(\kbar) \to
  \pi_{s,w}(\Kbar)$.
\end{lemma}

\begin{proof}
  Let $\ell$ be a prime. The maps $f^* : H_{**}(\kbar) \to
  H_{**}(\Kbar)$ and $f^* : \calA_{**}(\kbar) \to \calA_{**}(\Kbar)$
  are isomorphisms, hence the induced map of cobar complexes $f^* :
  \calC^{*}(\kbar) \to \calC^{*}(\Kbar)$ is an isomorphism. It follows
  that the map $\MASS_{\kbar}(\unit) \to \MASS_{\Kbar}(\unit)$ is an
  isomorphism from the $E_2$ page onwards.  The homomorphism $\bbL f^*
  : \pi_{**}(\unit\Hcomp(\kbar)) \to \pi_{**}(\unit\Hcomp(\Kbar))$ is
  therefore an isomorphism since it is compatible with the map of
  spectral sequences. Propositions \ref{prop:finiteness} and
  \ref{prop:0_stem} identify $\pi_{s,w}(\unit\Hcomp)$ with %
  $\pi_{s,w}(\unit)\ellcomp$ for all $s \geq w \geq 0$ over both
  $\kbar$ and $\Kbar$. By the work of Ananyevsky, Levine and Panin
  \cite{ALP}, the groups $\pi_{s,w}(\kbar)$ and $\pi_{s,w}(\Kbar)$ are
  torsion for $s > w \geq 0$ and so they are the sum of their
  $\ell$-primary parts. This establishes the result for $s> w \geq
  0$. When $s=w\geq 0$, the result follows by proposition
  \ref{prop:0_stem} and Morel's identification of the groups
  $\pi_{n,n}(F)$. If $s<w$, the connectivity theorem applies and gives
  the isomorphism.
\end{proof}

\begin{corollary}
  \label{cor:constant_iso}
  Let $\Kbar$ be an algebraically closed field of characteristic 0.
  For any $n\geq 0$, the map $\bbL c : \pi_n^s \to \pi_{n,0}(\Kbar)$
  is an isomorphism.
\end{corollary}

\begin{proof}
  The statement is true when $\Kbar = \bbC$ by Levine's theorem. The
  previous proposition extends the result to an arbitrary
  algebraically closed field of characteristic 0.
\end{proof}

\begin{theorem}
  \label{ellcomp_iso}
  Let $\Fbar$ be an algebraically closed field of characteristic $p$
  and let $\ell$ be a prime different from $p$. Then there is an
  isomorphism $\pi_{s,w}(\Fbar)\ellcomp \cong \pi_{s,w}(\bbC)\ellcomp$
  for all $s \geq w \geq 0$. 
\end{theorem}

\begin{proof}
  Consider the homomorphisms $\Fbar \leftarrow W \to \Kbar$.  The
  induced maps on the motivic Adams spectral sequence are compatible
  with the maps of homotopy groups
  \begin{equation*}
    \pi_{**}(\unit\Hcomp(\Fbar)) \leftarrow %
    \pi_{**}(\unit\Hcomp(W))
    \to \pi_{**}(\unit\Hcomp(\Kbar))
  \end{equation*}
  By corollary \ref{E2_iso}, the maps
  $\MASS_{\Fbar}(\unit) \leftarrow \MASS_W(\unit) \to
  \MASS_{\Kbar}(\unit)$
  are isomorphisms at the $E_2$ page, and so there are isomorphisms
  $\pi_{**}(\unit\Hcomp(\Fbar)) \cong \pi_{**}(\unit\Hcomp(W)) \cong
  \pi_{**}(\unit\Hcomp(\Kbar))$.
  For $s \geq w \geq 0$, propositions \ref{prop:finiteness} and
  \ref{prop:0_stem} give isomorphisms
  $\pi_{s,w}(\unit\Hcomp(\Fbar)) \cong \pi_{s,w}(\Fbar)\ellcomp$ and
  $\pi_{s,w}(\unit\Hcomp(\Kbar)) \cong \pi_{s,w}(\Kbar)\ellcomp$.  The
  result now follows from lemma \ref{lem:char-0}.
\end{proof}

\begin{corollary}
  \label{cor:ellcomp_iso}
  Let $\Fbar$ be an algebraically closed field of characteristic $p$
  and let $\ell$ be a prime different from $p$.  The homomorphism
  $\bbL c : (\pi_n^s)\ellcomp \to \pi_{n,0}(\Fbar)\ellcomp$ is an
  isomorphism for all $n\geq 0$.
\end{corollary}

\begin{proof}
  The previous theorem yields the following diagram for all $n\geq 0$.
  \begin{equation*}
    \xymatrix{
      & (\pi_n^s)\ellcomp %
      \ar[rd]_{\cong}^{\bbL c} %
      \ar[d]^{\bbL c} %
      \ar[ld]_{\bbL c} & & \\ %
      \pi_{n,0}(\Fbar)\ellcomp & \pi_{n,0}(\unit\Hcomp(W)) 
      \ar[r]^{\cong} \ar[l]_{\cong} & %
      \pi_{n,0}(\Kbar)\ellcomp 
    }
  \end{equation*}
  The map $\bbL c : (\pi_n^s)\ellcomp \to \pi_{n,0}(\Kbar)\ellcomp$ is
  an isomorphism by corollary \ref{cor:constant_iso}, and so all of
  the maps in the above diagram are isomorphisms.
\end{proof}

\begin{corollary} 
  \label{cor:fq_summand}
  For a finite field $\bbF_q$ with characteristic $p \neq \ell$, the
  group $(\pi_n^s)\ellcomp$ is a summand of
  $\pi_{n,0}(\bbF_q)\ellcomp$ for $n\geq 0$.
\end{corollary}
\begin{proof} 
  The map $\bbL c : \pi_n^s \to \pi_{n,0}(\Fpbar)$ factors through
  $\pi_{n,0}(\bbF_q)$.  Passing to the $\ell$-completion, corollary
  \ref{cor:ellcomp_iso} implies the composition $(\pi_n^s)\ellcomp \to
  \pi_{n,0}(\bbF_q)\ellcomp \to \pi_{n,0}(\Fpbar)\ellcomp$ is an
  isomorphism.  Hence the result.
\end{proof}

\section{The motivic Adams spectral sequence for finite fields}

We now analyze the two-complete stable stems $\hat{\pi}_{**}(\bbF_q) =
\pi_{**}(\bbF_q)\twocomp$ when $q$ is odd.  The results of the
previous section allow us to identify the $n$th topological
two-complete stable stem $\hat{\pi}_n^s = (\pi_n^s)\twocomp$ as a
summand of $\hat{\pi}_{n,0}(\bbF_q)$.  With this, we are able to
analyze the MASS for $\bbF_q$ in a range.  We remind the reader that
these results assume Morel's connectivity theorem hold for $\bbF_q$,
or the results hold without qualification for the fields
$\widetilde{\bbF}_q$.  For the remainder of this section, write $H$
for the mod 2 motivic cohomology spectrum.

\subsection{The $E_2$ page of MASS over $\bbF_q$ when $q \equiv 1 \bmod 4$}

We will make frequent use of the calculation $H^{**}(\bbF_q; \bbZ/2)
\cong \bbZ/2[\tau, u]/(u^2)$ which was given in section
\ref{mot_coh}. Recall $\tau$ and $u$ are in bidegree $(0,1)$ and
$(1,1)$ respectively.

\begin{proposition} 
  \label{prop:e2_1mod4}
  The $E_2$ page of the mod $2$ motivic Adams spectral sequence for
  the sphere spectrum over $\bbF_q$ with $q \equiv 1 \bmod 4$ is the
  trigraded algebra
  \begin{equation*}
    E_2 \cong \Ext(\bbF_q) \cong
    \bbF_2[\tau, u]/(u^2) \otimes_{\bbF_2[\tau]} \Ext(\Fpbar).
  \end{equation*}
  We abuse notation and write $\tau$ and $u$ for their duals. Hence in
  the above, $\tau$ and $u$ are of degree $(0,-1)$ and $(-1,-1)$
  respectively.
\end{proposition}

\begin{proof}
  Consult Dugger and Isaksen \cite[3.5]{DI} for a similar
  argument. Recall from proposition \ref{prop:algebroid} that
  $\calA^{**}(\bbF_q) \cong \calA^{**}(\Fpbar) \otimes_{\bbF_2[\tau]}
  \bbF_2[\tau,u]/(u^2)$ and $H^{**}(\bbF_q) \cong H^{**}(\Fpbar)
  \otimes \bbF_2[\tau,u]/(u^2) $.  Since $\bbF_2[\tau,u]/(u^2)$ is
  flat as a module over $\bbF_2[\tau]$, a free resolution
  $H^{**}(\Fpbar) \leftarrow P^{\bullet}$ by $\calA^{**}(\Fpbar)$
  modules determines a free resolution $H^{**}(\bbF_q) \leftarrow
  P^{\bullet} \otimes \bbF_2[\tau,u]/(u^2)$. It is necessary here that
  $\Sq^1(\tau)=0$ for $P^{\bullet} \otimes \bbF_2[\tau,u]/(u^2)$ to be
  a resolution of $\calA^{**}(\bbF_q)$ modules. %
  \todo{this addresses comment (6)}%
  The canonical map
  \begin{equation*}
    \Hom_{\calA^{**}(\Fpbar)}(-,
    H^{**}(\Fpbar))\otimes \bbF_2[\tau,u]/(u^2) \to
    \Hom_{\calA^{**}(\bbF_q)}(- \otimes \bbF_2[\tau,u]/(u^2) ,
    H^{**}(\bbF_q))
  \end{equation*} 
  is a natural isomorphism, since a generating set for a module $M$
  over $\calA^{**}(\Fpbar)$ is also a generating set for
  $M\otimes \bbF_2[\tau,u]/(u^2)$ over $\calA^{**}(\bbF_q)$ by
  proposition \ref{prop:algebroid}. We conclude that
  $\Ext(\Fpbar)\otimes \bbF_2[\tau,u]/(u^2) \cong \Ext(\bbF_q)$.
\end{proof}

By the previous proposition, the irreducible elements of $\Ext(\bbC)$
are also irreducible elements of $\Ext(\bbF_q)$ when $q \equiv 1 \bmod
4$.  The only additional irreducible element in $\Ext(\bbF_q)$ is the
class $u$.  The irreducible elements of $\Ext(\bbF_q)$ up to stem
$s=21$ can be found in table \ref{table:1mod4}. These were obtained by
consulting Isaksen \cite[Table 8]{StableStems} and independently
verified by computer calculation by Fu and Wilson \cite{Fu-Wilson}.
\begin{table}[ht!]
\begin{minipage}[t]{.32\textwidth}
\begin{center}
\begin{tabular}{ll}\hline
  Elt. & Filtr. $(f,s,w)$ \\ %
  \hline $u$ & $(0,-1,-1)$ \\ %
  $\tau$ & $(0,0,-1)$\\ %
  $h_0$ & $(1,0,0)$\\ %
  $h_1$ & $(1,1,1)$\\ %
  $h_2$ & $(1,3,2)$\\ %
  $h_3$ & $(1,7,4)$ \\
  \end{tabular}
\end{center}
\end{minipage}
\begin{minipage}[t]{.32\textwidth}
\begin{center}
\begin{tabular}{ll}\hline
  Elt. & Filtr. $(f,s,w)$ \\ \hline 
  $c_0$ & $(3,8,5)$ \\ 
  $Ph_1$ & $(5,9,5)$ \\ 
  $Ph_2$ & $(5,11,6)$ \\ 
  $d_0$ & $(4,14, 8)$ \\ 
  $h_4$ & $(1,15, 8)$ \\ 
  $Pc_0$ & $(7, 16, 9)$ \\
  \end{tabular}
\end{center}
\end{minipage}
\begin{minipage}[t]{.32\textwidth}
\begin{center}
  \begin{tabular}{ll}\hline
    Elt. & Filtr. $(f,s,w)$ \\ \hline 
    $e_0$ & $(4, 17, 10)$ \\
    $P^2h_1$ & $(9,17,9)$ \\ 
    $f_0$ & $(4, 18, 10)$ \\ 
    $P^2h_2$ & $(9,19,10)$\\ 
    $c_1$ & $(3,19,11)$ \\ 
    $[\tau g]$ & $(4, 20, 11)$ \\
\end{tabular}
\end{center}
\end{minipage}
\caption{The irreducible elements of $\Ext(\bbF_q)$ with $q \equiv 1 \bmod 4$ in stem $s\leq 21$}
\label{table:1mod4}
\end{table}

We now investigate the motivic May spectral sequence over the finite
field $\bbF_q$ when $q \equiv 1 \bmod 4$. We will find it useful for
calculating Massey products in the MASS.

\begin{definition}
  Write $J$ for the cokernel of the map $\eta_{L}: H_{**} \to
  \calA_{**}$ in the category of bigraded $\bbF_2$ vector spaces and
  consider the increasing filtration of $\calA_{**}$ given by
  \begin{equation*}
    F_n\calA_{**} = \ker ( \calA_{**} \xrightarrow{\Delta^n} \calA_{**}^{\otimes n+1} \to J^{\otimes n + 1}).
  \end{equation*}
  This filtration on $\calA_{**}$ induces a filtration on the cobar
  complex $(\calC, d)$ defined by Ravenel in \cite[A1.2.11]{Ravenel}. The
  filtration of the cobar complex is compatible leads to a spectral
  sequence \cite[A1.3.9]{Ravenel} called the motivic May spectral
  sequence.
\end{definition}

Following the work of Dugger and Isaksen \cite[\S5]{DI}, we are able
to identify the structure of the motivic May spectral sequence over a
finite field $\bbF_q$ when $q \equiv 1 \bmod 4$. 

\begin{proposition}
  The associated graded Hopf algebroid $E^0\calA_{**}$ to the
  filtration $F^*\calA_{**}$ of the motivic dual Steenrod algebra over
  a finite field $\bbF_q$ when $q \equiv 1 \bmod 4$ is the exterior
  algebra over $H_{**}(\bbF_q) \cong \bbF_2[\tau, u]/(u^2)$
  \begin{equation*}
    E^0\calA_{**} \cong E_{H_{**}(\bbF_q)}(\tau_i, \xi_j^{2^k} \st i \geq 0, j \geq 1, k \geq 0).
  \end{equation*}
  If each generator $\zeta_i$ of $E^0\calA_*^{top}$ is is assigned the
  weight of $\tau_{i-1}$ for $i\geq 1$ and $\zeta_i^{2^j}$ is assigned
  the weight of $\xi_{i}^{2^{j-1}}$ for $j\geq 1$, there is an
  isomorphism of trigraded algebras
  \begin{equation*}
    E^0\calA_{**} \cong \bbF_2[\tau, u]/(u^2) \otimes_{\bbF_2} E^0\calA_*
  \end{equation*}
  where $\calA_*$ denotes the topological dual Steenrod algebra, which
  was studied by Milnor in \cite{Milnor}.
\end{proposition}

\begin{proof}
  Since $u \in F^0\calA_{**}(\bbF_q)$ and $\calA_{**}(\bbF_q) \cong
  \bbF_2[\tau,u]/(u^2) \otimes_{\bbF_2[\tau]} \calA_{**}(\bbC)$, there
  are isomorphisms $F^n\calA_{**}(\bbF_q) \cong
  F^n\calA_{**}(\bbC)\otimes_{\bbF_2[\tau]}\bbF_2[\tau,u]/(u^2)$.
  Over $\bbC$, there is an isomorphism
  \begin{equation*}
    E^0\calA_{**}(\bbC) \cong \bbF_2[\tau] \otimes_{\bbF_2} E^0\calA_*
  \end{equation*}
  which follows by dualizing the result of Dugger and Isaksen in
  \cite[5.2(a)]{DI}. The result now follows as $\bbF_2[\tau] \to
  \bbF_2[\tau,u]/(u^2)$ is flat.
\end{proof}

\begin{proposition}
  The $E_2$ page of the motivic May spectral sequence over a finite
  field $\bbF_q$ with $q \equiv 1 \bmod 4$ is given by 
  \begin{align*}
    E_2^{m,f,s,w} &= \Ext_{E^0\calA_{**}(\bbF_q)}^{f,(s+f,w,m)}(H_{**}(\bbF_q),H_{**}(\bbF_q)) \\
    & \cong \bbF_2[\tau, u]/(u^2) \otimes_{\bbF_2[\tau]}
    \Ext_{E^0\calA_{**}(\bbC)}^{f,(s+f,w,m)}(H_{**}(\bbC),H_{**}(\bbC))
  \end{align*}
  where $f$ is the Adams filtration (or homological degree), $s$ is
  the stem, $w$ is the motivic weight and $m$ is the May
  filtration. The differential $d_r$ changes grading as $d_r : E_r^{m,
    f, s, w} \to E_r^{m+r-1, f+1, s-1, w}$. The motivic May spectral
  sequence converges to $\Ext_{\calA_{**}}(H_{**}, H_{**})$.

  To be consistent with the work of Dugger and Isaksen \cite{DI,
    StableStems}, we write the grading of an element in the May
  spectral sequence in the form $(m,f,s,w)$.
\end{proposition}

\begin{proof}
  The $E_2$ page of the motivic May spectral sequence is identified by
  Ravenel in \cite[A1.3.9]{Ravenel} in terms of the derived functors
  of the cotensor product $H_{**} \square_{\calA_{**}} -$. In this
  case, the natural isomorphism $\Hom_{\calA_{**}}(H_{**}, -) \cong
  H_{**} \square_{\calA_{**}} -$ identifies the cotor groups with the
  ext groups in the statement of the proposition. The second
  isomorphism follows formally from the result over $\bbC$ established
  by Dugger and Isaksen in \cite[5.2(b)]{DI} by the flatness of
  $\bbF_2[\tau,u]/(u^2)$ over $\bbF_2[\tau]$.
\end{proof}

A description of the motivic May spectral sequence $E_2$ page over
$\bbC$ is given by Dugger and Isaksen in \cite[\S5]{DI} up to the 36
stem, from which one obtains a description of the motivic May spectral
sequence $E_2$ page over $\bbF_q$ when $q \equiv 1 \bmod 4$ using the
previous proposition. One must simply add $u$ to the list of
generators of the $E_2$ page given in \cite[Table 1]{DI} and the
relation $u^2 =0$.

\subsection{The $E_2$ page of MASS over $\bbF_q$ when $q \equiv 3 \bmod 4$}

For a finite field $\bbF_q$ with $q \equiv 3 \bmod 4$, the $E_2$ page
of the MASS can be identified in a range using the $\rho$-Bockstein
spectral sequence ($\rho$-BSS) which was introduced by Hill in
\cite{Hill}.  Here $\rho=[-1]$ is the non-zero class in
$H^{1,1}(\bbF_q)\cong \bbF_q^{\times}/2$, since $-1$ is not a square
in $\bbF_q^{\times}$.  We briefly describe the construction of the
$\rho$-BSS and refer the reader to Dugger and Isaksen \cite{DI-Real}
or Ormsby \cite{MotBPInvQ, MotInvPadic} for more details.

Let $\calC$ be the cobar construction corresponding to the Hopf
algebroid
\begin{equation*}
  (\bbF_2[\tau, \rho]/(\rho^2), \calA_{**}(\bbF_q)).
\end{equation*}
The filtration of $\calC$ given by $0 \subseteq \rho \calC \subseteq
\calC$ determines a spectral sequence, which in this case is just the
long exact sequence associated to the short exact sequence of
complexes
\begin{equation*} 
  0 \to \rho \calC \to \calC \to \calC/\rho\calC \to   0.
\end{equation*}

Note that $\rho \calC$ and $\calC/\rho \calC$ are both isomorphic to
the cobar construction over $\bbC$.  Hence we have the following long exact
sequence.
\begin{equation*} 
  \xymatrix@C=.75cm{
    \cdots \rho \Ext^{i,(*,*)}(\bbC) \ar[r] & \Ext^{i,(*,*)}(\bbF_q)
    \ar[r] & \Ext^{i,(*,*)}(\bbC) \ar[r]^-{d_1} & \rho
    \Ext^{i+1,(*,*)}(\bbC) \cdots }
\end{equation*}
In spectral sequence notation, the $E_1$ page is given by
\begin{equation*}
  E_1^{\epsilon, f, (s, w)} \cong
  \begin{cases}
    \Ext^{f,(s,w)}(\bbC) & \text{ if } \epsilon = 0 \\
    \rho \Ext^{f, (s+1,w+1)}(\bbC) & \text{ if } \epsilon = 1 \\
    0 & \text{ otherwise }
  \end{cases}
\end{equation*}
with differential $d_1 : E_1^{\epsilon,f,(s,w)} \to E_1^{\epsilon + 1,
  f+1, (s-1,w)}$.  The differential $d_1$ satisfies the Leibniz rule,
so it suffices to identify the differential on irreducible
elements. We identify all differentials up to the 20 stem by hand in
the following proposition; these calculations have been verified by
computer calculations.

\begin{proposition} 
  \label{prop:rhoBSS_diffs}
  In the $\rho$-BSS for $\bbF_q$ with $q \equiv 3 \bmod 4$, every
  irreducible element $x$ of $\Ext(\bbC)$ in stem $s\leq 19$ other
  than $\tau$ has $d_1( x) = 0$.  Also, $d_1( \tau) = \rho h_0$ and
  $d_1 ([\tau g]) = \rho h_2 e_0$.  Here $[\tau g]$ is the irreducible
  element of $\Ext(\bbC)$ in stem $20$, weight $11$ and filtration
  $4$.
\end{proposition}
\begin{proof} 
  The differential $d_1$ vanishes on all irreducible classes in
  $\Ext(\bbC)$ up to stem 20 for degree reasons except for possibly
  $\tau$, $f_0$ and $[\tau g]$.  The class $\tau$ cannot survive the
  $\rho$-BSS, since if it did, it would contribute a nonzero element
  to $\Ext^{0,0,-1}(\bbF_q)\cong \Hom^{0,-1}_{\calA}(H^{**},H^{**})$
  which is trivial.  We conclude $d_1(\tau) = \rho h_0$, because this
  is the only possible nonzero value for $d_1(\tau)$.

  The two possibilities for $d_1(f_0)$ are $0$ and $\rho h_1
  e_0$. Since $h_1 f_0=0$ in $\Ext(\bbC)$, we must have $d_1(h_1 f_0)
  = h_1 d_1(f_0) = 0$; hence $d_1(f_0)$ is annihilated by $h_1$. But
  as $\rho h_1 e_0$ is not annihilated by $h_1$, we must have
  $d_1(f_0)=0$.

  The only possible nonzero value for $d_1([\tau g])$ is $\rho h_2
  e_0$. From the relation $h_0 [\tau g] = \tau h_2 e_0$, we calculate
  $d_1(\tau h_2 e_0) = \rho h_0h_2e_0$ and $d_1(h_0 [\tau g]) = h_0
  d_1([\tau g])$. Hence $h_0 d_1([\tau g]) = h_0 \rho h_2 e_0$, from
  which the result follows.
\end{proof}

\begin{example} 
  Since $d_1 (h_1) =0$, we conclude $d_1 (\tau h_1) = \rho h_0 h_1 = 0$,
  as $h_0h_1$ vanishes in $\Ext(\bbC)$.  Hence there is a class
  $[\tau h_1] \in \Ext^{1,(1,0)}(\bbF_q)$ which is irreducible.
\end{example}

With this analysis of the $\rho$-BSS for $\bbF_q$ with $q \equiv 3
\bmod 4$, the structure of $\Ext(\bbF_q)$ as a graded abelian group up
to stem $21$ follows immediately and we may further identify all
irreducible elements in this range. The results of this proposition
were verified by computer calculation by Fu and Wilson
\cite{Fu-Wilson}.

\begin{proposition}
  \label{irred_3mod4} 
  When $q \equiv 3\bmod 4$, the irreducible elements of $\Ext(\bbF_q)$
  up to stem $s=21$ are given in table \ref{table:3mod4}.
\begin{table}[ht!]
\begin{minipage}[t]{.32\textwidth}
\begin{center}
\begin{tabular}{ll}\hline
  Elt. & Filtr. $(f,s,w)$ \\ \hline 
  $\rho$ & $(0,-1,-1)$ \\
  $[\rho\tau]$ & $(0,-1,-2)$\\ 
  $[\tau^2]$ & $(0,0,-2)$\\ 
  $h_0$ & $(1,0,0)$\\ 
  $h_1$ & $(1,1,1)$\\ 
  $[\tau h_1]$ & $(1,1,0)$\\ 
  $h_2$ & $(1,3,2)$\\ 
  $[\tau h_2^2]$ & $(2,6,3)$\\ 
  $h_3$ & $(1,7,4)$ \\
  $[\tau h_0^3h_3]$ & $(4,7,3)$ \\
  $c_0$ & $(3,8,5)$\\ 
  \end{tabular}
\end{center}
\end{minipage}
\begin{minipage}[t]{.32\textwidth}
\begin{center}
\begin{tabular}{ll}\hline
  Elt. & Filtr. $(f,s,w)$ \\ \hline 
  $[\tau c_0]$ & $(3,8,4)$\\ 
  $Ph_1$ & $(5,9,5)$ \\ 
  $[\tau Ph_1]$ &$(5,9,4)$ \\ 
  $Ph_2$ & $(5,11,6)$ \\ 
  $[\tau h_0 h_3^2]$ & $(3, 14, 7)$\\ 
  $d_0$ & $(4,14, 8)$ \\ 
  $[\tau h_0^2 d_0]$ & $(6,14,7)$ \\
  $h_4$ & $(1,15,8)$ \\
  $[\tau h_0^7 h_4]$ & $(8,15,7)$ \\
  $Pc_0$ & $(7,16,9)$ \\%
  \phantom{asd} & \phantom{asd} \\
  \end{tabular}
\end{center}
\end{minipage}
\begin{minipage}[t]{.32\textwidth}
\begin{center}
  \begin{tabular}{ll}\hline
    Elt. & Filtr. $(f,s,w)$ \\ \hline 
    $[\tau Pc_0]$ & $(7, 16, 8)$\\%
    $e_0$ & $(4, 17, 10)$ \\%
    $P^2h_1$ & $(9,17,9)$\\%
    $[\tau P^2 h_1]$ & $(9,17,8)$\\%
    $f_0$ & $(4,18, 10)$ \\ %
    $P^2h_2$ & $(9,19,10)$\\ %
    $c_1$ & $(3,19,11)$ \\ %
    $[\tau c_1]$ & $(3,19,10)$\\%
    $[\rho\tau g]$ & $(4, 19, 10)$ \\%
    $[\tau^2 g]$ & $(4, 20, 10)$ \\
    \phantom{asd} & \phantom{asd} \\
\end{tabular}
\end{center}
\end{minipage}
\caption{The irreducible elements of $\Ext(\bbF_q)$ with $q \equiv 3 \bmod 4$ in stem $s\leq 21$}
\label{table:3mod4}
\end{table}
\end{proposition}
\begin{proof} 
  The structure of $\Ext(\bbF_q)$ as an abelian group follows directly
  from the $\rho$-BSS and the differentials calculated in proposition
  \ref{prop:rhoBSS_diffs}. We now explain why the tabulated elements
  comprise all of the irreducible elements in this range. If $y \in
  H^{**}(\rho\calC) \cong \rho \Ext(\bbC)$, then we may write $y =
  \rho \cdot x$ with $x \in H^{**}(\calC/\rho\calC) \cong \Ext(\bbC)$.
  So long as $x\neq 1$ and $d_1(x)=0$, the element $y$ is
  reducible. By proposition \ref{prop:rhoBSS_diffs} we conclude the
  only irreducible elements arising from $\rho\Ext(\bbC)$ in this
  range are $\rho$, $[\rho \tau]$ and $[\rho\tau g]$.

  Now consider an element $x$ of $H^{**}(\calC/\rho\calC) \cong
  \Ext(\bbC)$ which survives the $\rho$-BSS, that is, $d_1(x)=0$. Then
  $x$ is irreducible in $\Ext(\bbF_q)$ if and only if for any
  factorization $x = a\cdot b$ in $\Ext(\bbC)$ with $d_1(a)=d_1(b)=0$
  it follows $a=1$ or $b=1$. This observation identifies all of the
  remaining irreducible elements in $\Ext(\bbF_q)$ in the range $s\leq
  21$.
\end{proof}

\begin{remark}
  Although proposition \ref{irred_3mod4} lists all of the irreducible
  elements in $\Ext(\bbF_q)$ when $q \equiv 3 \bmod 4$ in a range,
  there are hidden products in the $\rho$-BSS. For example, the
  product $[\tau h_2^2]\cdot h_1 = \rho c_0$ is hidden in the
  $\rho$-BSS. We obtained this product by computer calculation,
  however the arguments by Dugger and Isaksen in \cite[6.2]{DI-Real}
  can be used to obtain some products by hand.
\end{remark}

\subsection{The Adams spectral sequence for $H\bbZ[p^{-1}]$}

We begin with the motivic Adams spectral sequence for
$X=H\bbZ[p^{-1}]$ over a finite field $\bbF_q$ of characteristic
$p$, as defined in \ref{def:mass}. In propositions
\ref{tau-differentials-1} and \ref{tau-differentials-3} we identify
the differentials for $\MASS_{\bbF_q}(H\bbZ[p^{-1}])$ which
converges to $\pi_{**} (H\bbZ[p^{-1}]\twocomp) \cong
H_{**}(\bbF_q ;\bbZ)\twocomp$.  We accomplish this by working
backwards from our knowledge of the target group $H^{**}(\bbF_q;
\bbZ)\twocomp$ which is isomorphic to $H_{et}^*(\bbF_q; \bbZ_2(*))$ as
a consequence of the Beilinson-Lichtenbaum conjecture.  Soul{\'e}'s
calculation of $ H_{et}^*(\bbF_q; \bbZ_2(*))$ in \cite[IV.2]{Soule}
then gives
  \begin{equation*}
    \pi_{s,w}(H\bbZ[p^{-1}]) \cong
    \begin{cases} %
      \bbZ_{\ell} & \text{if } s=w=0 \\ %
      \bbZ/(q^w-1)\twocomp & \text{if } s=-1, w\geq 1 \\ %
      0 & \text{otherwise}.
    \end{cases}
  \end{equation*}
  Although the spectrum $H\bbZ[p^{-1}]$ is cellular by the
  Hopkins-Morel theorem proven by Hoyois \cite[\S8.1]{Hoyois}, it is
  unclear if it is finite type. Instead of relying on proposition
  \ref{prop:convergence} for convergence, we establish a weak
  equivalence of the $H$-nilpotent completion of $H\bbZ[p^{-1}]$ with
  $H\bbZ\twocomp$.

\begin{lemma}
  \label{lem:hz_ellcomp}
  Let $\bbF_q$ be a finite field of characteristic $p\neq 2$.  The
  $H$-nilpotent completion of $H\bbZ[p^{-1}]$ is weak equivalent
  to $H\bbZ\twocomp$.
\end{lemma}

\begin{proof}
  We will show that the tower $H\bbZ/2 \leftarrow H\bbZ/2^2 \leftarrow
  H\bbZ/2^3 \leftarrow \cdots$ under $H\bbZ[p^{-1}]$ is an
  $H$-nilpotent resolution under $H\bbZ[p^{-1}]$ (defined by Bousfield
  in \cite[5.6]{Bousfield}). It will then follow that the homotopy
  limit of this tower is weak equivalent to the $H$-nilpotent
  completion of $H\bbZ[p^{-1}]$; that is, $H\bbZ\twocomp \cong
  H\bbZ[p^{-1}]\Hcomp$ by the observations of Dugger and Isaksen in
  \cite[\S7.7]{DI} which shows Bousfield's result
  \cite[5.8]{Bousfield} holds in the motivic stable homotopy category.

  The spectrum $H\bbZ[p^{-1}]$ is the homotopy colimit of the diagram
  $H\bbZ \xrightarrow{p\cdot} H\bbZ \xrightarrow{p\cdot} \cdots$.
  From the triangle $H\bbZ \xrightarrow{2^{\nu} \cdot} H\bbZ \to
  H\bbZ/2^{\nu}$, we obtain a triangle $H\bbZ[p^{-1}]
  \xrightarrow{2^{\nu}\cdot} H\bbZ[p^{-1}] \to H\bbZ/2^{\nu}$ after
  inverting $p$ since $p\neq 2$ and $H\bbZ/2^{\nu}
  \xrightarrow{p\cdot} H\bbZ/2^{\nu}$ is a homotopy equivalence.
  Consider the following cofibration sequence of towers.
  \begin{equation*}
    \xymatrix@R=.5cm{
    H\bbZ[p^{-1}] \ar[d]_{=} & 
    H\bbZ[p^{-1}] \ar[l]_{2\cdot} \ar[d]^{2\cdot } & 
    H\bbZ[p^{-1}] \ar[l]_{2\cdot} \ar[d]^{2^2 \cdot} & \cdots \ar[l]\\ 
    H\bbZ[p^{-1}] \ar[d] & 
    H\bbZ[p^{-1}] \ar[l]_{=} \ar[d] & 
    H\bbZ[p^{-1}] \ar[l]_{=} \ar[d] & \cdots \ar[l]\\
    \mathrm{pt} & 
    H\bbZ/2 \ar[l] & 
    H\bbZ/2^2 \ar[l] & \cdots \ar[l]
  }
  \end{equation*}
  It is clear that $H\bbZ/2^{\nu}$ is $H$-nilpotent for all $\nu\geq
  1$. For any $H$-nilpotent spectrum $N$ we show that the induced map
  $\colim_{\nu} \SH_{\bbF_q}(H\bbZ/2^{\nu}, N ) \to
  \SH_{\bbF_q}(H\bbZ[p^{-1}], N)$ is an isomorphism following the
  proof of Bousfield \cite[5.7]{Bousfield}. This isomorphism holds if
  and only if
  \begin{equation*}
    \colim \{ \SH_{\bbF_q}(H\bbZ[\tfrac{1}{p}], N) \xrightarrow{2\cdot} \SH_{\bbF_q}(H\bbZ[\tfrac{1}{p}], N) \} \cong 
    \SH_{\bbF_q}(H\bbZ[\tfrac{1}{p}], N)[\tfrac{1}{2}]
  \end{equation*}
  vanishes for all $H$-nilpotent $N$. This follows by an inductive
  proof with the following filtration of the $H$-nilpotent spectra
  given in \cite[3.8]{Bousfield}. Take $C_0$ to be the collection of
  spectra $H\wedge X$ for $X$ any spectrum, and let $C_{m+1}$ be the
  collection of the spectra $N$ for which either $N$ is a retract of
  an element of $C_m$ or there is a triangle $X \to N \to Z$ with $X$
  and $Z$ in $C_m$.

  If $N = H \wedge X$, it is clear that $\SH_{\bbF_q}(H\bbZ[p^{-1}],
  N) \xrightarrow{2} \SH_{\bbF_q}(H\bbZ[p^{-1}], N)$ is the zero map,
  which establishes the base case. If the claim holds for $N$ in
  filtration $C_m$, the claim holds for $N$ in filtration $C_{m+1}$ by
  a standard argument.  The claim now follows.
\end{proof}

\begin{proposition}
  \label{tau-differentials-1}
  The mod 2 motivic Adams spectral sequence for $X=H\bbZ[p^{-1}]$
  over $\bbF_q$ when $q \equiv 1 \bmod 4$ has $E_1$ page given by
  \begin{equation*}
    E_1 \cong \bbF_2[\tau,u, h_0]/(u^2)
  \end{equation*} 
  where $h_0 \in E_1^{1,(0,0)}$.

  Write $\nu_2$ for the 2-adic valuation and $\epsilon(q)$ for
  $\nu_2(q-1)$. For all $r\geq 1$ the differentials $d_r$ vanish on
  $u\tau^j$ and $h_0^j$. If $r < \epsilon(q) + \nu_{2}(j)$ the
  differentials $d_r\tau^j$ vanish and we have
  \begin{equation*}
    d_{\epsilon(q)+\nu_2(j)} \tau^{j} = u\tau^{j-1}
    h_0^{\epsilon(q) + \nu_2(j)}.  
  \end{equation*}
  In particular, the differential $d_1$ is trivial, so
  $E_2 \cong E_1$.
\end{proposition}

\begin{proof}
  We build the following $H^{**}$-Adams resolution of
  $H\bbZ[p^{-1}]$ utilizing the triangles constructed in
  \ref{lem:hz_ellcomp}.
  \begin{equation}
    \label{eq:H-adams-res}
    \xymatrix@R=.5cm{ H\bbZ[p^{-1}] \ar[rd]_{j_0} && H\bbZ[p^{-1}]
      \ar[ll]_{2\cdot} \ar[dr]_{j_1} && H\bbZ[p^{-1}] \ar[ll]_{2\cdot} & \cdots
      \ar[l] \\ &H \ar[ur]^-{\bullet}_{\del_0} && H
      \ar[ur]^-{\bullet}_{\del_1} && }
  \end{equation} 
  The spectrum $H\bbZ[p^{-1}]$ is cellular, so the motivic Adams
  spectral sequence for $X=H\bbZ[p^{-1}]$ converges to
  $\pi_{**}(H\bbZ[p^{-1}]\Hcomp)$ by proposition
  \ref{prop:convergence}. Lemma \ref{lem:hz_ellcomp} shows that
  $\pi_{**}(H\bbZ[p^{-1}]\Hcomp)\cong \pi_{**}(H\bbZ\twocomp)$,
  so the spectral sequence converges
  \begin{equation*}
    E_2^{f,(s,w)} \Rightarrow H^{-s,-w}(\bbF_q;\bbZ)\twocomp.
  \end{equation*}

  The groups $H^{s,w}(\bbF_q;\bbZ)\twocomp$ are isomorphic to
  $H_{et}^s(\bbF_q; \bbZ_2(w))$ which were calculated by Soul{\'e} in
  \cite[IV.2]{Soule}. If $q \equiv 1 \bmod 4$
  \begin{equation}
    \label{eq:ell-adic-cohom}
    H^{-s,-w}(\bbF_q;\bbZ)\twocomp \cong
    \begin{cases} %
      \bbZ_{\ell} & \text{if } s=w=0 \\ %
      \bbZ/(q^w-1)\twocomp & \text{if } s=-1, w\geq 1 \\ %
      0 & \text{otherwise}.
    \end{cases}
  \end{equation}
  Note that $\nu_2(q^{w}-1) = \epsilon(q) + \nu_2(w)$ for all natural
  numbers $w$. The formulas for the differentials on $\tau^j$ are the
  only choice to give $H^{**}(\bbF_q;\bbZ)\twocomp$ as the
  $E_{\infty}$ term.
\end{proof}

\begin{proposition}
  \label{tau-differentials-3}
  The mod 2 motivic Adams spectral sequence for $X=H\bbZ[p^{-1}]$
  over $\bbF_q$ when $q \equiv 3 \bmod 4$ has $E_1$ page given by
  \begin{equation*}
    E_1 \cong \bbF_2[\tau, \rho, h_0]/(\rho^2)
  \end{equation*} 
  where $h_0 \in E_1^{1,(0,0)}$.  

  For all $r\geq 1$ the differentials $d_r$ vanish on $\rho \tau^j$
  and $h_0^j$. For odd natural numbers $j$, we calculate $d_1 (\tau^j)
  = \rho \tau^{j-1} h_0$.  Write $\lambda(q)$ for $\nu_2(q^2-1)$.  If
  $r < \lambda(q) + \nu_{2}(n)$ the differentials $d_r \tau^{2n}$
  vanish and
  \begin{equation*}
    d_{\lambda(q)+\nu_2(n)} \tau^{2n} = \rho\tau^{2n-1}
    h_0^{\lambda(q) + \nu_2(n)}.
  \end{equation*}
\end{proposition}

\begin{proof}
  The proof of the previous proposition goes through, except the
  target groups $ H^{-s,-w}(\bbF_q;\bbZ)\twocomp$ force different
  differentials in the spectral sequence when $q \equiv 3 \bmod 4$.
  Soul{\'e}'s calculation in equation \ref{eq:ell-adic-cohom}
  shows the order of $H^{1,1}(\bbF_q;\bbZ)\twocomp$ is $\nu_2(q-1)=1$,
  so we conclude $d_1( \tau) = \rho h_0$. As we have $\nu_2(q^{2j}-1)
  = \lambda(q) + \nu_2(j)$ for all natural numbers $j$, the claimed
  formulas for the differentials on $\tau^{2n}$ hold.
\end{proof}

\begin{corollary}
  \label{sphere-tau-diffs}
  In the MASS of $\unit$ over a finite field $\bbF_q$ with $q \equiv 1
  \bmod 4$, the differentials $d_r(\tau^j)$ vanish when $r <
  \epsilon(q) + \nu_{2}(j)$ and
  \begin{equation*}
    d_{\epsilon(q)+\nu_2(j)} \tau^{j} = u\tau^{j-1}
    h_0^{\epsilon(q) + \nu_2(j)}.  
  \end{equation*}

  In the MASS of $\unit$ over a finite field $\bbF_q$ with $q \equiv 3
  \bmod 4$, the differentials $d_r ([\tau^{2}]^n)$ vanish when $r <
  \lambda(q) + \nu_{2}(n)$ and
  \begin{equation*}
    d_{\lambda(q)+\nu_2(n)} [\tau^{2}]^n = [\rho\tau][\tau^2]^{n-1}
    h_0^{\lambda(q) + \nu_2(n)}.
  \end{equation*}
\end{corollary}

\begin{proof}
  The unit map $\unit \to H\bbZ[p^{-1}]$ induces a map of motivic
  Adams spectral sequences $\MASS(\unit) \to \MASS(H\bbZ[p^{-1}])$. On
  the $E_2$ page, observe that when $q \equiv 1 \bmod 4$ the classes
  $\tau$ and $u$ map to $\tau$ and $u$ respectively. When $q \equiv 3
  \bmod 4$, the classes $[\tau^2]$, $\rho$, $[\rho \tau]$ map to
  $[\tau^2]$, $\rho$, $[\rho \tau]$ respectively. The identification
  of the differentials in the MASS for $H\bbZ[p^{-1}]$ in propositions
  \ref{tau-differentials-1} and \ref{tau-differentials-3} then force
  the differentials stated in the corollary.
\end{proof}

\begin{example}
  When $q \equiv 3 \bmod 4$, the Massey product $\langle \rho, \rho,
  h_0 \rangle$ in the mod 2 motivic Adams spectral sequence for
  $H\bbZ[p^{-1}]$ is $\rho\tau$. Since we have $\rho^2 =0$
  and $d_1(\tau) = \rho h_0$, it follows that $0 + \rho \tau$ is in
  the Massey product. It is straightforward to verify that the
  indeterminacy is trivial. 
\end{example}

\subsection{Stable stems over $\bbF_q$}

We now begin an analysis of the differentials in the MASS to identify
the two-complete stable stems over $\bbF_q$. To assist the reader with
the computations presented below, figure \ref{fig:1mod4} and figure
\ref{fig:3mod4} display $E_2$ page charts of the MASS over $\bbF_q$.
Further charts can be found in the work of Fu and Wilson
\cite{Fu-Wilson}. Throughout this section, $\bbF_q$ is a finite field
with $q$ elements where $q$ is odd, and we write $\hat{G}$ for the
two-completion of an abelian group $G$.

Corollary \ref{cor:fq_summand} shows that $\pihat_n^s$ is a summand of
$\pihat_{n,0}(\bbF_q)$ for all $n \geq 0$. We will soon see that for
small values of $n\geq 0$ we have $\hat{\pi}_{n,0}(\bbF_q) \cong
\hat{\pi}_n^s \oplus \hat{\pi}_{n+1}^s$.  However this pattern fails
when $n = 19$ and $q \equiv 1 \bmod 4$. %

\begin{lemma}
  \label{lem:0stem}
  For a finite field $\bbF_q$ with $q$ odd, there is an isomorphism
  $\pi_{0,0}(\bbF_q) \cong \pi_0^s \oplus \pi_1^s$.
\end{lemma}

\begin{proof}
  The stem $\pi_{0,0}(\bbF_q)$ is isomorphic to the Grothendieck-Witt
  group $GW(\bbF_q)$ by Morel \cite{Mor03}. The isomorphism
  $GW(\bbF_q)\cong \bbZ \oplus \bbZ/2$ was established by Scharlau in
  \cite[Chapter 2, 3.3]{Scharlau}. Recall that $\pi_0^s \cong \bbZ$ and
  $\pi_1^s \cong \bbZ/2$. Hence we conclude $\pi_{0,0}(\bbF_q) \cong
  \pi_0^s \oplus \pi_1^s$.
\end{proof}

Morel's calculation of $\pi_{0,0}(\bbF_q)$ shows that $2 =
(1-\epsilon) + \rho \eta$, hence multiplication by 2 in
$\pi_{**}(\bbF_q)$ is detected in the mod 2 motivic Adams spectral
sequence by the class $h_0 + \rho h_1$ in $\Ext(\bbF_q)$. This is
needed to solve the extension problems when passing from the Adams
spectral sequence $E_{\infty}$ page to the stable stems.

\begin{proposition}
  \label{prop:1mod4}
  When $q \equiv 1 \bmod 4$ and $0 \leq n \leq 18$, there is an
  isomorphism
  $\pihat_{n,0}(\bbF_q) \cong \pihat_n^s \oplus \pihat_{n+1}^s$.
\end{proposition}
\begin{proof}
  Lemma \ref{lem:0stem} takes care of the case when $n = 0$. We now
  focus on $0<n\leq 18$ where the mod 2 MASS over $\bbF_q$ converges
  to the groups $\pihat_{n,0}(\bbF_q)$ by propositions
  \ref{prop:convergence} and \ref{prop:finiteness}.

  The irreducible elements of $\Ext(\bbF_q)$ in this range are given
  in table \ref{table:1mod4}.  All differentials $d_r$ for $r\geq 2$
  vanish on $h_0, h_1, h_3, c_0, Ph_1, d_0, Pc_0, P^2h_1$ for degree
  reasons.  As $\pihat_{3,0}(\bbF_q)$ must contain
  $\pihat_3^s \cong \bbZ/8$ as a summand by proposition
  \ref{cor:fq_summand}, we conclude
  $d_2 (\tau^2 h_2) = \tau^2 d_2(h_2) = 0$. The only possible non-zero
  value for $d_2(h_2)$ is $uh_1^3$. If $d_2(h_2)=uh_1^3$, then
  $d_2(\tau^2 h_2) = u\tau^2 h_1^3$ would be non-zero by the product
  structure of $\Ext(\bbF_q)$ in proposition
  \ref{prop:e2_1mod4}---a contradiction. Hence $d_2(h_2)= 0$.

  The non-zero Massey product $Ph_2 = \langle h_3, h_0^4, h_2\rangle$
  has no indeterminacy, because $h_3 E_2^{4,(3,2)} +
  E_2^{4,(7,4)}h_2=0$.  Since $\pihat_{11}^s\cong \bbZ/8$ is a summand
  of $\pihat_{11,0}$, the differential $d_2(Ph_2)$ must vanish.  The
  non-zero Massey product $P^2h_2 = \langle h_3, h_0^4, h_2\rangle$
  has no indeterminacy, because $h_3 E_2^{8,(11,6)} +
  E_2^{4,(7,4)}Ph_2 = 0$. Since $d_2 (Ph_2) = 0$, the topological
  result of Moss \cite[1(ii)]{Moss} implies $d_2 (P^2h_2) = 0$.

  The comparison map $\MASS(\bbF_q) \to \MASS(\Fpbar)$ shows that
  $d_2(h_4)$ and $d_3(h_0h_4)$ must be non-zero, as these
  differentials are non-zero in $\MASS(\Fpbar)$ by corollary
  \ref{E2_iso} and calculations of Isaksen \cite[Table 8]{StableStems}
  over $\bbC$. The only possible choice for $d_2(h_4)$ is $h_0h_3^2$,
  but $d_3(h_0h_4)$ is either $h_0d_0$ or $h_0d_0 + uh_1d_0$. In order
  to have $\pihat_{14}^s\cong \bbZ/2 \oplus \bbZ/2$ as a summand of
  $\pihat_{14,0}$, we must have $d_3(h_0h_4) = h_0d_0$.  A similar
  argument establishes $d_2 (e_0) = h_1^2 d_0$ and $d_2 (f_0) = h_0^2
  e_0$.  Note that $d_4 (h_0^3 h_4) = 0$ for degree reasons.

  The elements in weight $0$ are all of the form $\tau^j x$ or
  $u\tau^{j-1}x$ where $x$ is not a multiple of $\tau$ and of weight
  $j$.  The differentials of the elements in weight $0$ are now
  readily identified by using the Leibniz rule from corollary
  \ref{sphere-tau-diffs}.  Since $\pihat_n^s$ is a summand of
  $\pihat_{n,0}(\bbF_q)$ for all $n\geq 0$, we see that there are no
  hidden 2-extensions for $0< n \leq 18$. 
\end{proof}

In the proof of the following proposition, we provide some technical
details in a smaller font for the convenience of the reader. We follow
the convention of Dugger and Isaksen \cite{DI} and write the grading
of an element in the motivic May spectral sequence as $(m,s,f,w)$
where $m$ is the May filtration, $s$ is the stem, $f$ is the Adams
filtration and $w$ is the motivic weight. However, we continue to
write the grading in the MASS as $(f, (s,w))$.

\begin{proposition}
  \label{19and20} 
  When $q \equiv 1 \bmod 4$, there are isomorphisms
  $\hat{\pi}_{19,0}(\bbF_q) \cong (\bbZ/8 \oplus \bbZ/2) \oplus
  \bbZ/4$ and $\hat{\pi}_{20,0}(\bbF_q) \cong \bbZ/8 \oplus
  \bbZ/2$. In particular, when $q \equiv 1 \bmod 8$ we find $d_2([\tau
  g])$ is trivial, and when $q \equiv 5 \bmod 8$ we calculate
  $d_2([\tau g]) =u h_0 h_2 e_0$.
\end{proposition}
\begin{proof} 
  When $q \equiv 1 \bmod 4$, it is possible that $d_2([\tau g]) $ is
  $uh_0h_2e_0$. We analyze this differential using Massey products
  obtained from the May spectral sequence.  We show that $\langle
  \tau, h_1^4, h_4 \rangle = \{[\tau g] \}$ in the $E_2$ page of the
  MASS using Massey products in the May spectral sequence and the May
  convergence theorem in Isaksen \cite[2.2.1]{StableStems}. %

  At the $E_4$ page of the May spectral sequence we calculate
  $d_4(b_{21}^2) = h_1^4 h_4$ and $d_4(0) = \tau h_1^4$, as $\tau
  h_1^4 =0$; hence $[\tau g] = \tau b_{21}^2 \in \langle \tau, h_1^4,
  h_4 \rangle$ in the May spectral sequence.  There are no crossing
  differentials, so the May convergence shows $[\tau g] \in \langle
  \tau, h_1^4, h_4 \rangle$ in the MASS.

  \details{In this case, we must check if there are crossing
    differentials $d_t$ for $t\geq 5$. To see $E_4^{*,5,3,3}=0$ over
    $\bbF_q$, we check $E_4^{*,5,3,3}=0$ and $E_4^{*,6,3,4}=0$ over
    $\bbC$ using the chart in \cite[Appendix C]{DI}.  All that is in
    $(*,5,3,3)$ is $h_1b_{20}$, but this does not survive to
    $E_4$. And nothing is in $(*,6,3,4)$ even at the $E_2$ page.

    To see $E_5^{*,20,4,12}$ is trivial over $\bbF_q$, observe that
    all that is in $E_4^{*,20,4,12}$ over $\bbC$ is $b_{21}^2$, which
    does not survive to the $E_5$ page. The group $E_4^{*,21,4,13}$
    over $\bbC$ is trivial. A potential contribution from $h_0h_3^3$
    or $h_0h_2^2h_4$ is ruled out by weight reasons, and because they
    do not survive to the $E_4$ page from the differentials
    $d_2(h_0(1))$ and $d_2(h_0b_{22})$.  }

  The indeterminacy $\tau E_2^{4, (20,12)}+E_2^{3, (5,3)}h_4$ in the
  MASS is trivial, so we conclude $\langle \tau, h_1^4, h_4 \rangle =
  \{[\tau g ]\}$.

  We now identify $d_2([\tau g])$ using the following formula due to
  Moss \cite[1(i)]{Moss}.
  \begin{equation}
    \label{eq:taug}
    d_2(\langle \tau, h_1^4, h_4 \rangle) \subseteq 
    \langle d_2(\tau), h_1^4,h_4\rangle + 
    \langle \tau, 0, h_4\rangle + 
    \langle \tau, h_1^4, h_0h_3^2 \rangle
  \end{equation}
  The Massey product $\langle \tau, 0, h_4\rangle$ contains $0$ and
  has no indeterminacy.

  \details{ Here $0 = d_2(h_1^4)$ is in grading $E_2^{6,(3,4)}$, so
    the indeterminacy is $ \tau E_3^{6,(19,12)} + E_3^{5,(4,3)}h_4.$
    The degree of $h_1^2e_0$ is $6,(19,12)$, but it does not survive
    to the $E_3$ page. The group $E_3^{5,(4,3)}$ is trivial by
    checking the $E_2$ page. }

  To calculate $ \langle \tau, h_1^4, h_0h_3^2 \rangle$ we again use
  the May spectral sequence and the May convergence theorem. We
  calculate this Massey product at the $E_2$ page using
  $d_2(h_2b_{20}) = \tau h_1^4$ and $h_1^4h_0h_3^2 =0$ and see that $0
  \in \langle \tau, h_1^4, h_0h_3^2 \rangle$. There are no crossing
  differentials, so $0$ is in this Massey product in the MASS.

  \details{Note that $a_{01} = h_1b_{20}$ is in degree $(5,5,3,3)$ and
    $a_{12}$ is in degree $(8,19,6,12)$. Then for $a_{01}$ crossing
    differentials occur in $(?, 5,3,3)$ which is trivial from the
    fourth page on. For $a_{12}$ crossing diffs occur in degree
    $(m',19,6,12)$ with $m'\geq 8$. The only thing in this filtration,
    stem and weight is $h_1^2 e_0$ which has May filtration 10. But
    note that both $h_1^2$ and $e_0$ are permanent cycles, so that
    $h_1^2e_0$ is as well.  So there are no crossing
    differentials in this case. }

  The indeterminacy for $ \langle \tau, h_1^4, h_0h_3^2 \rangle$ in
  the MASS is $\tau E_2^{6,(19,12)} + E_2^{3,(5,3)}h_0h_3^2$ which is
  trivial. The group $E_2^{6,(19,12)}$ is generated by $h_1^2 e_0$
  which is annihilated by $\tau$ while $E_2^{3,(5,3)}$ is trivial.

  We now handle the Massey product $ \langle d_2(\tau),
  h_1^4,h_4\rangle $ which depends on the base field. Let us suppose
  that $q \equiv 1 \bmod 8$ so that $d_2(\tau) =0$ by corollary
  \ref{sphere-tau-diffs}. If $a_{12}$ is in the $E_1$ page of the MASS
  with $d_1(a_{12}) = h_1^4h_4$, then the Massey product contains
  $0\cdot h_4 + 0\cdot a_{12} = 0$. It is straightforward to check
  that the indeterminacy $ 0\cdot E_2^{4,(20,12)} + E_2^{5,(4,3)}\cdot
  h_4$ is trivial. We conclude $d_2([\tau g])=0$ when $q \equiv 1
  \bmod 8$.

  When $q \equiv 5 \bmod 8$, corollary \ref{sphere-tau-diffs}
  establishes $d_2(\tau) = uh_0^2$. We identify the Massey product
  $\langle uh_0^2, h_1^4,h_4\rangle $ using the May spectral sequence
  and the May convergence theorem. At the $E_4$ page of the May
  spectral sequence we have $d_4(b_{21}^2) = h_1^4h_4$ and $uh_0^2
  h_1^4 = 0$. Hence $u h_0^2 b_{21}^2 + 0 h_4 = u h_0 b_{21} h_2
  h_0(1) = u h_0 h_2 e_0$ is in the Massey product under
  consideration. It is straightforward to verify that there are no
  crossing differentials in this case.

  \details{As $a_{01}=0$ in $E_4^{9,4,5,3}$ and $a_{12} = b_{12}^2$,
    we must check two conditions: (1) whenever $m' \geq 9$ and $m'-5 <
    t$ that $d_t$ is trivial on $E_t^{m',4,5,3}$, and (2) whenever $m'
    \geq 8$ and $m'-4 < t$ that $d_t$ is trivial on
    $E_t^{m',20,4,12}$. Condition (1) is easily verified as
    $E_4^{*,4,5,3} = 0$ over $\bbC$ and $E_4^{*,5,5,4} = 0 $ over
    $\bbC$ as well. We conclude $E_4^{*,4,5,3} = 0$ over $\bbF_q$ as
    only these two groups can contribute to this graded piece. We
    remark that $u h_1^5$ does not contribute any terms, since to get
    the weight correct one needs to multiply by $\tau$ which
    annihilates the element. For condition (2), we will check that for
    all $t \geq 6$ the differentials vanish on
    $E_t^{(*,20,4,12)}$. This graded piece contains $b_{21}^2$ at the
    $E_4$ page, but it does not survive to $E_5 = E_6$. The only other
    possible elements in this group arise from elements in
    $E_t^{*,21,4,13}$ over $\bbC$ which we have seen is trivial at the
    $E_4$ page.  This verifies the hypotheses of May's convergence
    theorem. }

  The indeterminacy of $ \langle uh_0^2, h_1^4,h_4\rangle $ in the
  MASS is $ u h_0^2 E_2^{4, (20,12)} + E_2^{5, (4, 3)} h_4 $ which is
  trivial. Thus the May convergence theorem shows the Massey product
  is exactly $\{u h_0 h_2 e_0\}$ and we conclude $d_2([\tau g]) =
  uh_0h_2e_0$ if $q \equiv 5 \bmod 8$.

  We now analyze the differentials in the MASS in the 19 and 20
  stem. Since $[\tau g]$ has weight 11, the class $\tau^{11} [\tau g]
  $ is in $E^{4,(20,0)}$.  If $q \equiv 1 \bmod 8$, we calculate
  $d_2(\tau^{11}[\tau g]) = \tau^{11} u h_0 h_2 e_0 = u\tau^{10} h_0^2
  [\tau g]$.  If $q \equiv 5 \bmod 8$, then $d_2(\tau^{11}[\tau g]) =
  u \tau^{10} h_0^2 [\tau g]$. This resolves all differentials in the
  19 and 20 stem, so the calculation of the 19 stem follows.

  As $\hat{\pi}_{20}^s \cong \bbZ/8$ must be a summand of
  $\hat{\pi}_{20,0}(\bbF_q)$, we conclude there is a hidden extension
  from $u \tau^{11} h_2^2 h_4 = u \tau^{11} h_3^3$ to $\tau^{12}h_2
  e_0$.  The calculation of the 20-stem now follows.
\end{proof} 

\begin{remark} 
  Note that over $\bbF_q$ with $q \equiv 5 \bmod 8$ the map $\bbL c
  \{g\}$ is detected by $u \tau^{11} h_3^3$ which is in Adams
  filtration 3.  But over $\Fqbar$, $\bbL c\{g\}$ is in Adams
  filtration 4.
\end{remark}

\begin{proposition}
  \label{prop:3mod4} 
  When $q \equiv 3 \bmod 4$ and $0\leq n \leq 18$, there is an
  isomorphism
  $\hat{\pi}_{n,0}(\bbF_q) \cong \hat{\pi}_n^s \oplus
  \hat{\pi}_{n+1}^s$.
\end{proposition}
\begin{proof}
  The case $n=0$ is resolved by lemma \ref{lem:0stem}, so we now
  consider $0<n\leq 18$, where we may use the motivic Adams spectral
  sequence as in proposition \ref{prop:1mod4}.

  The differentials $d_r$ for $r\geq 2$ vanish on the following
  generators for degree reasons: $[\rho \tau]$, $\rho$, $h_0$, $h_1$,
  $h_3$, $[\tau h_2^2]$, $[\tau c_0]$, $[\tau Ph_1]$, $d_0$, $[\tau P
  c_0]$, $[\tau P^2h_1]$.  Since $\hat{\pi}_1^s \cong \bbZ/2$ is a
  summand of $\pihat_{1,0}(\bbF_q)$, we must have $d_r([\tau h_1]) = 0$
  for all $r\geq 2$.  Since $\pihat_3^s \cong \bbZ/8$ is a summand of
  $\hat{\pi}_{3,0}(\bbF_q)$, we must have $d_2 (h_2) = 0$.  An
  argument similar to that given for proposition \ref{prop:1mod4}
  establishes $d_2 (h_4) = h_0 h_3^2$, $d_2 (e_0) = h_1^2 d_0$ and
  $d_2 (f_0) = h_0^2 e_0$ by comparison to $\MASS(\Fqbar)$. Also, we
  determine $d_r([\tau c_1]) = 0$ for $r\geq 2$ by comparing with
  $\MASS(\Fqbar)$, as the class $[\tau c_1]$ must be a permanent cycle.

  The one exceptional case is $d_3 (h_0h_4)$.  Here we must have $d_3
  (h_0h_4) = h_0d_0 + \rho h_1 d_0$ in order for $\pihat_{14}^s =
  \bbZ/2\oplus \bbZ/2$ to be a summand of $\pihat_{14,0}(\bbF_q)$.

  The elements in weight $0$ are all of the form $[\tau^2]^i x$ or
  $[\rho\tau][\tau^2]^{i-1}x$ where $x$ is not a multiple of $\tau^2$
  and weight $2i$, or of the form $\rho[\tau^2]^i x$ if $x$ is not a
  multiple of $\tau^2$ and of weight $2i+1$.  The differentials of the
  elements in weight $0$ are now determined by using the Leibniz rule.
  Since $\lambda(q) = \nu_2(q^2-1) \geq 3$, we have $d_2(\tau^2) = 0$.
  This is sufficient to ensure that for elements $x$ in stem
  $s \leq 19$ there are no non-trivial differentials of the form
  $d_r ([\tau^2]^i x) = \rho\tau^{2i-1} h_0^r x$ when $[\tau^2]^i x$ has
  weight $0$. This resolves all differentials in weight 0 for stems
  $s\leq 19$ and there are no hidden 2-extensions in this range. Hence
  for $0< n \leq 18$ there is an isomorphism
  $\pihat_{n,0}(\bbF_q) \cong \pihat_n^s \oplus \pihat_{n+1}^s$.
\end{proof}


\begin{remark} 
  When $q \equiv 3 \bmod 4$, it is unclear whether $d_2([\tau^2 g]) =
  [\rho \tau g]$ or $d_2([\tau^2 g]) = 0$.  This is all that obstructs
  the identification of the stems $\hat{\pi}_{19,0}(\bbF_q)$ and
  $\hat{\pi}_{20,0}(\bbF_q)$ in this case.
\end{remark}

\subsection{Base change for finite fields}

\begin{proposition}
  \label{odd-base-change} 
  Let $q = p^{\nu}$ where $p$ is an odd prime.  For a field extension
  $f : \bbF_q \to \bbF_{q^i}$ with $i$ odd, the induced map $\bbL f^*
  : H^{**}(\bbF_q) \to H^{**}(\bbF_{q^i})$ and $\bbL f^* :
  \calA^{**}(\bbF_q) \to \calA^{**}(\bbF_{q^i})$ are isomorphisms.
\end{proposition}
\begin{proof} 
  The claim follows by checking on \'etale cohomology.  The map on
  cohomology is determined by
  $H^1_{et}(\bbF_q;\mu_2) \to H^1_{et}(\bbF_{q^i};\mu_2)$ which is
  just the induced map $\bbF_q^{\times}/2 \to \bbF_{q^i}^{\times}/2$.
  So long as $i$ is odd, this map is an isomorphism.
\end{proof}

\begin{corollary} 
  For $q = p^{\nu}$ with $p$ an odd prime, the induced map
  $\MASS(\bbF_q) \to \MASS(\bbF_{q^i})$ is an isomorphism of spectral
  sequences whenever $i$ is odd.
\end{corollary} 

\begin{proposition} 
  \label{prop:q-tilde}
  Let $q = p^{\nu}$ with $p$ an odd prime.  Let $\widetilde{\bbF}_q$
  denote the union of the field extensions $\bbF_{q^i}$ over $\bbF_q$
  with $i$ odd.  The field extension $f : \bbF_q \to
  \widetilde{\bbF}_q$ induces isomorphisms $\bbL f^* : H^{**}(\bbF_q)
  \to H^{**}(\widetilde{\bbF}_q)$ and $\bbL f^* : \calA^{**}(\bbF_q)
  \to \calA^{**}(\widetilde{\bbF}_q)$.  Hence the map $\MASS(\bbF_q)
  \to \MASS(\widetilde{\bbF}_q)$ is an isomorphism of spectral
  sequences.
\end{proposition}
\begin{proof} 
  This follows by a colimit argument using proposition
  \ref{odd-base-change}.
\end{proof}

\begin{corollary} 
  For any integers $s$ and $w \geq 0$, there is an isomorphism
  $\hat{\pi}_{s,w}(\bbF_q) \cong \hat{\pi}_{s,w}(\widetilde{\bbF}_q)$.
\end{corollary}

\begin{proposition} 
  \label{even_extension}
  Let $q= p^{\nu}$ with $p$ an odd prime.  For a field extension $f :
  \bbF_q \to \bbF_{q^i}$ with $i$ even, the map
  $f^* : H^{1,*}(\bbF_q) \to H^{1, *}(\bbF_{q^j})$ is trivial and
  $f^* : H^{0,*}(\bbF_q) \to H^{0, *}(\bbF_{q^j})$ is injective.
\end{proposition}
\begin{proof} 
  The map is determined by $\bbL f^* : H^{1,1}(\bbF_q) \to \
  H^{1,1}(\bbF_{q^i})$ which is just the map $\bbF_q^{\times}/2 \to
  \bbF_{q^i}^{\times}/2$.  However, any non-square $x \in
  \bbF_q^{\times}$ will be a square in $\bbF_{q^i}^{\times}$ when $i$
  is even.
\end{proof}

\begin{corollary}
  Let $q= p^{\nu}$ with $p$ an odd prime.  For a field extension $f :
  \bbF_q \to \bbF_{q^i}$ with $i$ even, the induced map $\MASS(\bbF_q)
  \to \MASS(\bbF_{q^i})$ kills the class $u$ (respectively $\rho$ and
  $[\rho \tau]$) and all of their multiples at the $E_2$ page.
\end{corollary}

\begin{proof}
  The induced map of cobar complexes is determined from proposition
  \ref{even_extension} and shows the classes $u$ (respectively $\rho$
  and $[\rho \tau]$) are killed under base change.
\end{proof}

\section{Implementation of motivic Ext group calculations}
\label{section:implementation}

The computer calculations used in this paper were done with the
program available from Fu and Wilson \cite{Fu-Wilson} at
\url{https://github.com/glenwilson/MassProg}. The program is written
in python and calculates $\Ext(F)$ when $F$ is $\bbC$, $\bbR$ or
$\bbF_q$ by producing a minimal resolution of $H^{**}(F)$ by
$\calA^{**}(F)$ modules in a range. With this complex in hand, the
program then produces its dual and calculates cohomology in each
degree.

To calculate a free resolution of $H^{**}(F)$ of $\calA^{**}(F)$
modules, we first need the program to efficiently perform calculations
in $\calA^{**}(F)$. The mod 2 motivic Steenrod algebra is generated by
the squaring operations $\Sq^i$ and the cup products $\alpha \cup - $
for $\alpha \in H^{**}(F)$. These generators satisfy Adem relations,
which are recorded in \cite[5.1]{HKOst} by Hoyois, Kelly and
{\O}stv{\ae}r in \cite[5.1]{HKOst} and \cite[10.2]{MR2031198} by
Voevodsky.  Additionally, one needs the commutation relations
$\Sq^{2i}\tau = \tau \Sq^{2i} + \tau \rho \Sq^{2i-1}$ for $i >0$ and
$\Sq^{2i+1}\tau = \tau \Sq^{2i+1} + \rho \Sq^{2i} + \rho^2 \Sq^{2i-1}$
for $i\geq 0$ which are obtained from the Cartan formula. With these
relations, the program can calculate the canonical form of any element
of $\calA^{**}$, that is, as a sum of monomials $\alpha\cdot \Sq^I$
where $\alpha \in H^{**}(F)$ and $I$ is an admissible sequence.

With the algebra of $\calA^{**}(F)$ available to the program, it then
proceeds to calculate a minimal resolution of $H^{**}(F)$ by
$\calA^{**}(F)$ modules. This is where a great deal of computational
effort is spent. To clarify what a minimal resolution is in practice,
let $\prec$ denote the order on $\bbZ\times\bbZ$ given by
$(m_1,n_1)\prec(m_2,n_2)$ if and only if $m_1 + n_1 < m_2 + n_2$, or
$m_1 + n_1 = m_2 + n_2$ and $n_1 < n_2$. The reader is encouraged to
compare this definition with the definition of McCleary in
\cite[Definition 9.3]{McCleary} and consult Bruner \cite{BB3} for
detailed calculations of a minimal resolution for the Adams spectral
sequence of topology.

\begin{definition}
  A resolution of $H^{**}(F)$ by $\calA^{**}(F)$ modules
  $H^{**}(F) \leftarrow P^{\bullet}$ is a minimal resolution if the
  following conditions are satisfied.
\begin{enumerate}
\item Each module $P^i$ is equipped with ordered basis $\{h_i(j) \}$
  such that if $j\leq k$ then $\deg h_i(j) \preceq \deg h_i(k)$.
    \item $\im(h_i(k)) \notin \im(\langle h_i(j) \st j < k \rangle)$
    \item $\deg h_i(k)$ is minimal with respect to degree in the order
  $\prec$ over all elements in
  $P^{i-1} \setminus \im(\langle h_i(j) \st j < k\rangle)$.
\end{enumerate} 
\end{definition}

The computer program calculates the first $n$ maps and modules in a
minimal resolution up to bidegree $(2n, n)$. With this, it then
calculates the dual of the resolution by applying the functor
$\Hom_{\calA^{**}(F)}(-,H^{**}(F))$ to the resolution
$P^{\bullet}$. With the cochain complex
$\Hom_{\calA^{**}(F)}(P^{\bullet}, H^{**}(F))$ in hand, the program
calculates cohomology in each degree, that is,
$\Ext^{f,(s+f,w)}(\bbF_q)$. 

Because the program calculates an explicit resolution of $H^{**}(F)$,
the products of elements in $\Ext(F)$ can be obtained from the
composition product, see McCleary \cite[9.5]{McCleary}.

\section{Charts}
\label{section:charts}

The weight 0 part of the $E_2$ page of the mod 2 MASS over $\bbF_q$ is
depicted in figures \ref{fig:1mod4} and \ref{fig:3mod4} according to
the case $q \equiv 1 \bmod 4$ or $q \equiv 3 \bmod 4$. The weight 0
part of the $E_{\infty}$ page of the mod 2 MASS over $\bbF_q$ can be
found in figures \ref{fig:1mod4-infty} and \ref{fig:3mod4-infty}.

In each chart, a circular or square dot in grading $(s, f)$ represents
a generator of the $\bbF_2$ vector space in the graded piece of the
spectral sequence. The square dots are used to indicate that the given
element is divisible by $u$, $\rho$, or $\rho\tau$, depending on the
case. Circular dots denote elements which are not divisible by $u$,
$\rho$, or $\rho\tau$. In the chart \ref{fig:3mod4-infty}, there is an
oval dot which corresponds to the class with representative
$\tau^8 \rho h_1 d_0 \equiv \tau^8 h_0 d_0$, as the class
$\rho h_1 d_0 + h_0 d_0$ is killed.

We indicate the product of a given class by $h_0$ with a solid,
vertical line. In the case $q \equiv 3 \bmod 4$, multiplication by
$\rho h_1$ plays an important role, so non-zero products by $\rho h_1$
are indicated by dashed vertical lines. In particular, when $q \equiv
3 \bmod 4$, multiplication by $2$ in $\hat{\pi}_{**}(\bbF_q)$ is
detected by multiplication by $h_0 + \rho h_1$.  The lines of slope 1
indicate multiplication by $\tau h_1$ or $[\tau h_1]$ depending on the
case. We caution the reader that the product structure displayed in
this chart was obtained by computer calculation and not all products
were established by hand in this paper. For example, the products in the
8-stem by $h_0$ are hidden in the May spectral sequence.

Dotted lines are used in two separate instances in these charts. The
first use is in figure \ref{fig:1mod4-infty}, where dotted lines indicate
hidden extensions by $h_0$ and $\tau h_1$. The other instance is in
figure \ref{fig:3mod4-infty} to indicate an unknown $d_2$
differential.

Additional charts can be found in the work of Fu and Wilson
\cite{Fu-Wilson}.

\begin{figure}[ht!]
  \centering %
{\LARGE
\begin{tikzpicture}[x=1.65cm, y=2.15cm, scale=.47, rotate=90,transform shape]
\draw[gray!40] (-1, 0) -- (22, 0);
\foreach \n in {1,...,5}{\draw[line width=.25pt, gray!40] (-0.5, 2*\n) -- (22, 2*\n);}
\foreach \n in {0,...,22}{\draw[line width=.25pt, gray!40] (\n-.5,0) -- (\n-.5,11);}
\foreach \n in {0,..., 21}{\draw (\n, -.1) node [anchor=north] {\n};}
\draw (-0.5, .2) node [anchor=east] {0};
\foreach \n in {1,...,11}{\draw (-0.5, \n) node [anchor=east] {\n};}

\foreach \n in {0,...,10}{\draw [line width=.6pt] (.1,\n) -- (.1,{1+\n});}
\draw [line width=.6pt] (3+.1,1) -- (3+.1, 2);
\draw [line width=.6pt] (3+.1,2) -- (3+.1, 3);
\draw [line width=.6pt] (7+.1,1) -- (7+.1, 2);
\draw [line width=.6pt] (7+.1,2) -- (7+.1, 3);
\draw [line width=.6pt] (7+.1,3) -- (7+.1, 4);
\draw [line width=.6pt] (11+.1,5) -- (11+.1, 6);
\draw [line width=.6pt] (11+.1,6) -- (11+.1, 7);
\draw [line width=.6pt] (14+.1,2) -- (14+.1, 3);
\draw [line width=.6pt] (14+.1,4) -- (14+.1, 5);
\draw [line width=.6pt] (14+.1,5) -- (14+.1, 6);
\draw [line width=.6pt] (15+.1,1) -- (15+.1, 8);
\draw [line width=.6pt] (17+.1,4) -- (17+.1, 7);
\draw [line width=.6pt] (18+.1,2) -- (18+.1, 4);
\draw [line width=.6pt] (18+.3,4) -- (18+.1, 5);
\draw [line width=.6pt] (19+.1,9) -- (19+.1, 11);
\draw [line width=.6pt] (20+.1,4) -- (20+.1, 6);

\draw [line width=.6pt] (0+.1,0) -- (3+.1,3);
\draw [line width=.6pt] (7+.1,1) -- (9+.1,3);
\draw [line width=.6pt] (8+.1,3) -- (9+.1,4);
\draw [line width=.6pt] (9+.1,5) -- (11+.1,7);
\draw [line width=.6pt] (14+.1,4) -- (15+.3,5);
\draw [line width=.6pt] (15+.3,5) -- (16+.1,6);
\draw [line width=.6pt] (16+.1,6) -- (17+.1,7);
\draw [line width=.6pt] (15+.1,1) -- (18+.1,4);
\draw [line width=.6pt] (16+.1,7) -- (17+.1,8);
\draw [line width=.6pt] (17+.1,9) -- (19+.1,11);
\draw [line width=.6pt] (17+.1,4) -- (18+.1,5);
\draw [line width=.6pt] (20+.1,4) -- (21+.1,5);

\filldraw [black] 
\foreach \n in {0,...,11}{(0+.1,\n) circle (2pt)};
\filldraw [black]
(1+.1,1) circle (2pt)
(2+.1,2) circle (2pt)
(3+.1,1) circle (2pt)
(3+.1,2) circle (2pt)
(3+.1,3) circle (2pt)
(6+.1,2) circle (2pt)
(7+.1,1) circle (2pt)
(7+.1,2) circle (2pt)
(7+.1,3) circle (2pt)
(7+.1,4) circle (2pt)
(8+.1,2) circle (2pt)
(8+.1,3) circle (2pt)
(9+.1,3) circle (2pt)
(9+.1,4) circle (2pt)
(9+.1,5) circle (2pt)
(10+.1,6) circle (2pt)
(11+.1,5) circle (2pt)
(11+.1,6) circle (2pt)
(11+.1,7) circle (2pt)
(14+.1,2) circle (2pt)
(14+.1,3) circle (2pt)
(14+.1,4) circle (2pt)
(14+.1,5) circle (2pt)
(14+.1,6) circle (2pt)
(15+.1,1) circle (2pt)
(15+.1,2) circle (2pt)
(15+.1,3) circle (2pt)
(15+.1,4) circle (2pt)
(15+.1,5) circle (2pt)
(15.2+.1,5) circle (2pt)
(15+.1,6) circle (2pt)
(15+.1,7) circle (2pt)
(15+.1,8) circle (2pt)
(16+.1,2) circle (2pt)
(16+.1,6) circle (2pt)
(16+.1,7) circle (2pt)
(17+.1,3) circle (2pt)
(17+.1,4) circle (2pt)
(17+.1,5) circle (2pt)
(17+.1,6) circle (2pt)
(17+.1,7) circle (2pt)
(17+.1,8) circle (2pt)
(17+.1,9) circle (2pt)
(18+.1,2) circle (2pt)
(18+.1,3) circle (2pt)
(18+.1,4) circle (2pt)
(18.2+.1,4) circle (2pt)
(18+.1,5) circle (2pt)
(18+.1,10) circle (2pt)
(19+.1,3) circle (2pt)
(19+.1,9) circle (2pt)
(19+.1,10) circle (2pt)
(19+.1,11) circle (2pt)
(20+.1,4) circle (2pt)
(20+.1,5) circle (2pt)
(20+.1,6) circle (2pt)
(21+.1,3) circle (2pt)
(21+.1,5) circle (2pt)
;

\draw [line width=.6pt] (2-.1,1) -- (2-.1, 2);
\draw [line width=.6pt] (2-.1,2) -- (2-.1, 3);
\draw [line width=.6pt] (6-.1,1) -- (6-.1, 2);
\draw [line width=.6pt] (6-.1,2) -- (6-.1, 3);
\draw [line width=.6pt] (6-.1,3) -- (6-.1, 4);
\draw [line width=.6pt] (10-.1,5) -- (10-.1, 6);
\draw [line width=.6pt] (10-.1,6) -- (10-.1, 7);
\draw [line width=.6pt] (13-.1,2) -- (13-.1, 3);
\draw [line width=.6pt] (13-.1,4) -- (13-.1, 5);
\draw [line width=.6pt] (13-.1,5) -- (13-.1, 6);
\draw [line width=.6pt] (14-.1,1) -- (14-.1, 8);
\draw [line width=.6pt] (16-.1,4) -- (16-.1, 7);
\draw [line width=.6pt] (17-.1,2) -- (17-.1, 4);
\draw [line width=.6pt] (17-.3,4) -- (17-.1, 5);
\draw [line width=.6pt] (18-.1,9) -- (18-.1, 11);
\draw [line width=.6pt] (19-.1,4) -- (19-.1, 6);
\draw [line width=.6pt] (21-.1,8) -- (21-.1, 10);

\draw [line width=.6pt] (0-.1,1) -- (2-.1,3);
\draw [line width=.6pt] (6-.1,1) -- (8-.1,3);
\draw [line width=.6pt] (7-.1,3) -- (8-.1,4);
\draw [line width=.6pt] (8-.1,5) -- (10-.1,7);
\draw [line width=.6pt] (13-.1,4) -- (14-.3,5);
\draw [line width=.6pt] (14-.3,5) -- (15-.1,6);
\draw [line width=.6pt] (15-.1,6) -- (16-.1,7);
\draw [line width=.6pt] (14-.1,1) -- (17-.1,4);
\draw [line width=.6pt] (15-.1,7) -- (16-.1,8);
\draw [line width=.6pt] (16-.1,9) -- (18-.1,11);
\draw [line width=.6pt] (16-.1,4) -- (17-.1,5);
\draw [line width=.6pt] (19-.1,4) -- (20-.1,5);

\filldraw 
([xshift=-2.5pt,yshift=-2.5pt]0-.1,1) rectangle ++(5pt, 5pt)
([xshift=-2.5pt,yshift=-2.5pt]1-.1,2) rectangle ++(5pt, 5pt)
([xshift=-2.5pt,yshift=-2.5pt]2-.1,1) rectangle ++(5pt, 5pt)
([xshift=-2.5pt,yshift=-2.5pt]2-.1,2) rectangle ++(5pt, 5pt)
([xshift=-2.5pt,yshift=-2.5pt]2-.1,3) rectangle ++(5pt, 5pt)
([xshift=-2.5pt,yshift=-2.5pt]5-.1,2) rectangle ++(5pt, 5pt)
([xshift=-2.5pt,yshift=-2.5pt]6-.1,1) rectangle ++(5pt, 5pt)
([xshift=-2.5pt,yshift=-2.5pt]6-.1,2) rectangle ++(5pt, 5pt)
([xshift=-2.5pt,yshift=-2.5pt]6-.1,3) rectangle ++(5pt, 5pt)
([xshift=-2.5pt,yshift=-2.5pt]6-.1,4) rectangle ++(5pt, 5pt)
([xshift=-2.5pt,yshift=-2.5pt]7-.1,2) rectangle ++(5pt, 5pt)
([xshift=-2.5pt,yshift=-2.5pt]7-.1,3) rectangle ++(5pt, 5pt)
([xshift=-2.5pt,yshift=-2.5pt]8-.1,3) rectangle ++(5pt, 5pt)
([xshift=-2.5pt,yshift=-2.5pt]8-.1,4) rectangle ++(5pt, 5pt)
([xshift=-2.5pt,yshift=-2.5pt]8-.1,5) rectangle ++(5pt, 5pt)
([xshift=-2.5pt,yshift=-2.5pt]9-.1,6) rectangle ++(5pt, 5pt)
([xshift=-2.5pt,yshift=-2.5pt]10-.1,5) rectangle ++(5pt, 5pt)
([xshift=-2.5pt,yshift=-2.5pt]10-.1,6) rectangle ++(5pt, 5pt)
([xshift=-2.5pt,yshift=-2.5pt]10-.1,7) rectangle ++(5pt, 5pt)
([xshift=-2.5pt,yshift=-2.5pt]13-.1,2) rectangle ++(5pt, 5pt)
([xshift=-2.5pt,yshift=-2.5pt]13-.1,3) rectangle ++(5pt, 5pt)
([xshift=-2.5pt,yshift=-2.5pt]13-.1,4) rectangle ++(5pt, 5pt)
([xshift=-2.5pt,yshift=-2.5pt]13-.1,5) rectangle ++(5pt, 5pt)
([xshift=-2.5pt,yshift=-2.5pt]13-.1,6) rectangle ++(5pt, 5pt)
([xshift=-2.5pt,yshift=-2.5pt]14-.1,1) rectangle ++(5pt, 5pt)
([xshift=-2.5pt,yshift=-2.5pt]14-.1,2) rectangle ++(5pt, 5pt)
([xshift=-2.5pt,yshift=-2.5pt]14-.1,3) rectangle ++(5pt, 5pt)
([xshift=-2.5pt,yshift=-2.5pt]14-.1,4) rectangle ++(5pt, 5pt)
([xshift=-2.5pt,yshift=-2.5pt]14-.1,5) rectangle ++(5pt, 5pt)
([xshift=-2.5pt,yshift=-2.5pt]14-.2-.1,5) rectangle ++(5pt, 5pt)
([xshift=-2.5pt,yshift=-2.5pt]14-.1,6) rectangle ++(5pt, 5pt)
([xshift=-2.5pt,yshift=-2.5pt]14-.1,7) rectangle ++(5pt, 5pt)
([xshift=-2.5pt,yshift=-2.5pt]14-.1,8) rectangle ++(5pt, 5pt)
([xshift=-2.5pt,yshift=-2.5pt]15-.1,2) rectangle ++(5pt, 5pt)
([xshift=-2.5pt,yshift=-2.5pt]15-.1,6) rectangle ++(5pt, 5pt)
([xshift=-2.5pt,yshift=-2.5pt]15-.1,7) rectangle ++(5pt, 5pt)
([xshift=-2.5pt,yshift=-2.5pt]16-.1,3) rectangle ++(5pt, 5pt)
([xshift=-2.5pt,yshift=-2.5pt]16-.1,4) rectangle ++(5pt, 5pt)
([xshift=-2.5pt,yshift=-2.5pt]16-.1,5) rectangle ++(5pt, 5pt)
([xshift=-2.5pt,yshift=-2.5pt]16-.1,6) rectangle ++(5pt, 5pt)
([xshift=-2.5pt,yshift=-2.5pt]16-.1,7) rectangle ++(5pt, 5pt)
([xshift=-2.5pt,yshift=-2.5pt]16-.1,8) rectangle ++(5pt, 5pt)
([xshift=-2.5pt,yshift=-2.5pt]16-.1,9) rectangle ++(5pt, 5pt)
([xshift=-2.5pt,yshift=-2.5pt]17-.1,2) rectangle ++(5pt, 5pt)
([xshift=-2.5pt,yshift=-2.5pt]17-.1,3) rectangle ++(5pt, 5pt)
([xshift=-2.5pt,yshift=-2.5pt]17-.1,4) rectangle ++(5pt, 5pt)
([xshift=-2.5pt,yshift=-2.5pt]17-.2-.1,4) rectangle ++(5pt, 5pt)
([xshift=-2.5pt,yshift=-2.5pt]17-.1,5) rectangle ++(5pt, 5pt)
([xshift=-2.5pt,yshift=-2.5pt]17-.1,10) rectangle ++(5pt, 5pt)
([xshift=-2.5pt,yshift=-2.5pt]18-.1,3) rectangle ++(5pt, 5pt)
([xshift=-2.5pt,yshift=-2.5pt]18-.1,9) rectangle ++(5pt, 5pt)
([xshift=-2.5pt,yshift=-2.5pt]18-.1,10) rectangle ++(5pt, 5pt)
([xshift=-2.5pt,yshift=-2.5pt]18-.1,11) rectangle ++(5pt, 5pt)
([xshift=-2.5pt,yshift=-2.5pt]19-.1,4) rectangle ++(5pt, 5pt)
([xshift=-2.5pt,yshift=-2.5pt]19-.1,5) rectangle ++(5pt, 5pt)
([xshift=-2.5pt,yshift=-2.5pt]19-.1,6) rectangle ++(5pt, 5pt)
([xshift=-2.5pt,yshift=-2.5pt]20-.1,3) rectangle ++(5pt, 5pt)
([xshift=-2.5pt,yshift=-2.5pt]20-.1,5) rectangle ++(5pt, 5pt)
([xshift=-2.5pt,yshift=-2.5pt]21-.1,4) rectangle ++(5pt, 5pt)
([xshift=-2.5pt,yshift=-2.5pt]21-.1,8) rectangle ++(5pt, 5pt)
([xshift=-2.5pt,yshift=-2.5pt]21-.1,9) rectangle ++(5pt, 5pt)
([xshift=-2.5pt,yshift=-2.5pt]21-.1,10) rectangle ++(5pt, 5pt)
;

\draw (0+.1, 1) node [anchor=west] {$h_0$};
\draw (1+.1, 1) node [anchor=west] {$\tau h_1$};
\draw (3+.1, 1) node [anchor=west] {$\tau^2 h_2$};
\draw (6, 2+.15) node [anchor=west] {$\tau^4 h_2^2$};
\draw (7+.1, 1) node [anchor=west] {$\tau^4 h_3$};
\draw (8+.1, 3-.05) node [anchor=west] {$\tau^5c_0$};
\draw (9, 5) node [anchor=north] {$\tau^5 Ph_1$};
\draw (11, 5) node [anchor=north] {$\tau^6 Ph_2$};
\draw (14, 2+.1) node [anchor=north west] {$\tau^8h_3^2$};
\draw (14, 4+.1) node [anchor=north west] {$\tau^8d_0$};
\draw (15+.15, 1) node [anchor=west] {$\tau^8h_4$};
\draw (16+.1, 7+.1) node [anchor=north west] {$\tau^9Pc_0$};
\draw (17+.15, 4) node [anchor=west] {$\tau^{10}e_0$};
\draw (17, 9) node [anchor=north] {$\tau^{9}P^2h_1$};
\draw (18, 2) node [anchor=north] {$\tau^{10}h_2h_4$};
\draw (18+.45, 4) node [anchor=north] {$\tau^{10}f_0$};
\draw (19, 9) node [anchor=north] {$\tau^{10}P^2h_2$};
\draw (19, 3) node [anchor=north] {$\tau^{11}c_1$};
\draw (20, 4) node [anchor=north] {$\tau^{11}[\tau g]$};
\draw (20, 5) node [anchor=south west] {$\tau^{12}h_2e_0$};
\draw (21, 3) node [anchor=north] {$\tau^{12}h_2^2h_4$};
\draw (21+.15, 4) node [anchor=north] {$u\tau^{12}h_2c_1$};
\end{tikzpicture}
}
  \caption{$E_2$ page of MASS for $\bbF_q$ with $q \equiv 1 \bmod 4$, weight 0}
  \label{fig:1mod4}
\end{figure}
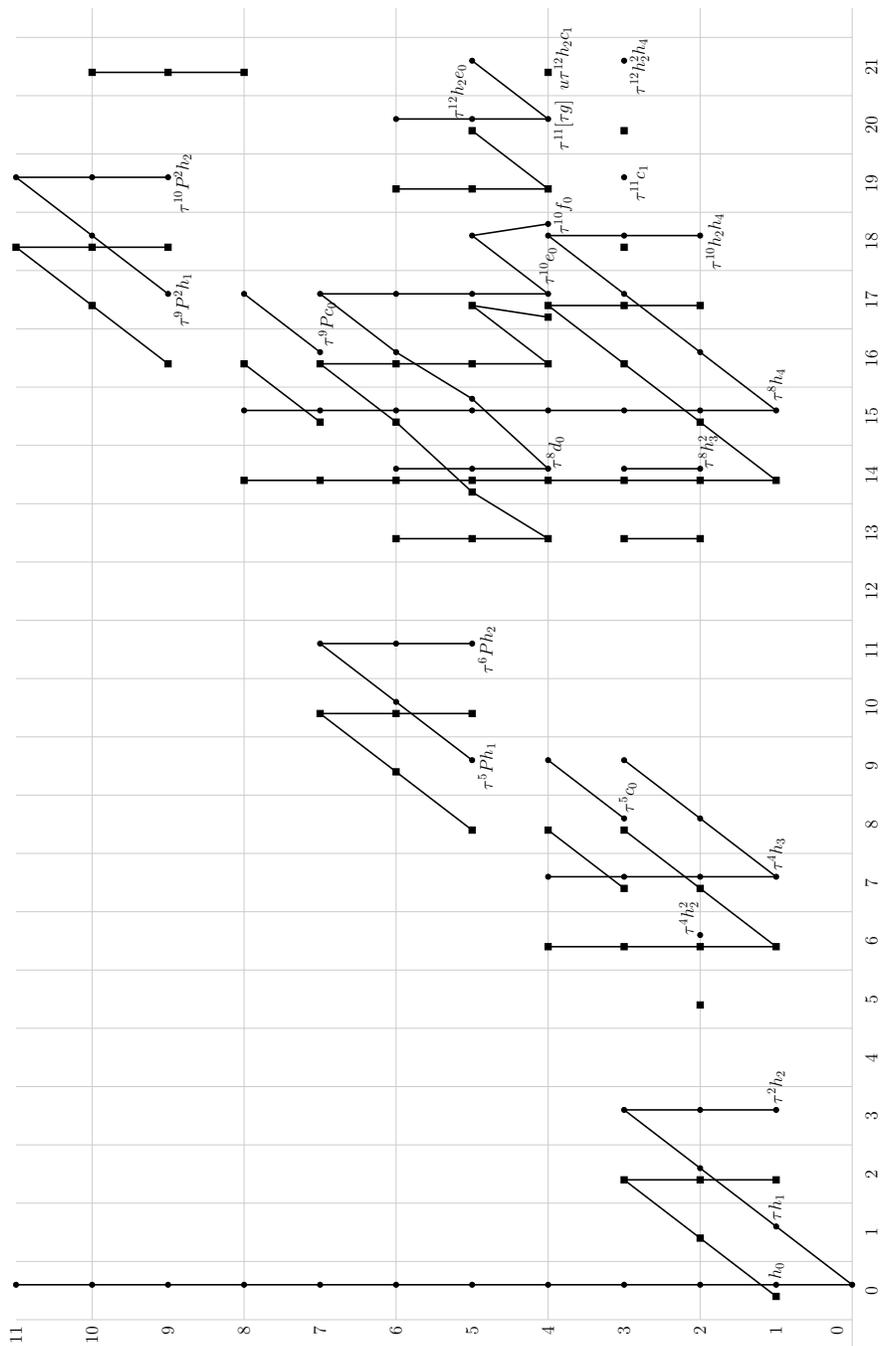

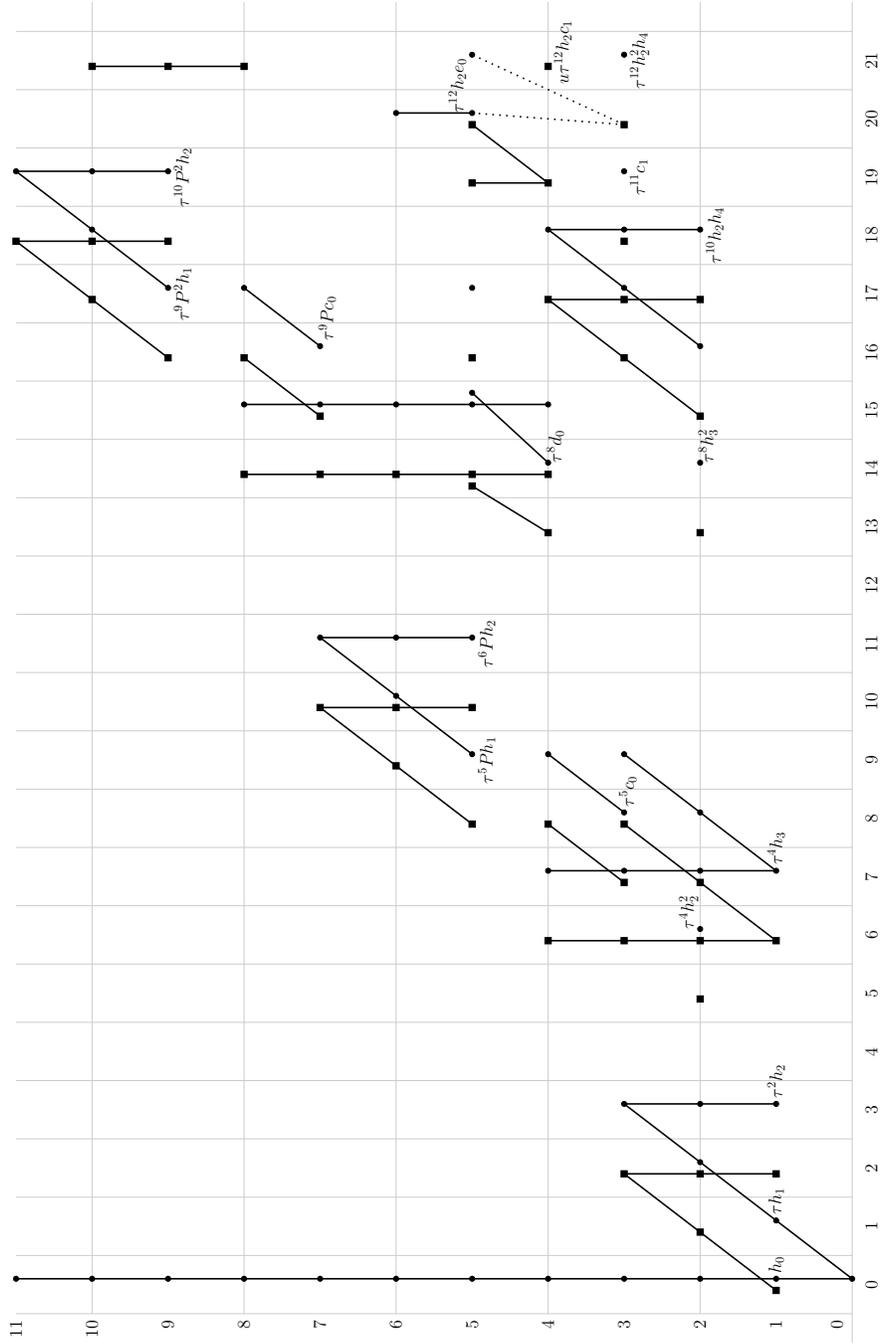
\begin{figure}[ht!]
  \centering
{\LARGE
\begin{tikzpicture}[x=1.65cm, y=2.15cm, scale=.47, rotate=90, transform shape]
\draw[gray!40] (-1, 0) -- (22, 0);
\foreach \n in {1,...,5}{\draw[line width=.25pt, gray!40] (-0.5, 2*\n) -- (22, 2*\n);}
\foreach \n in {0,...,22}{\draw[line width=.25pt, gray!40] (\n-.5,0) -- (\n-.5,11);}
\foreach \n in {0,..., 21}{\draw (\n, -.1) node [anchor=north] {\n};}
\draw (-0.5, .2) node [anchor=east] {0};
\foreach \n in {1,...,11}{\draw (-0.5, \n) node [anchor=east] {\n};}

\foreach \n in {0,...,10}{\draw [line width=.6pt] (.1,\n) -- (.1,{1+\n});}
\draw [line width=.6pt] (3+.1,1) -- (3+.1, 2);
\draw [line width=.6pt] (3+.1,2) -- (3+.1, 3);
\draw [line width=.6pt] (7+.1,1) -- (7+.1, 2);
\draw [line width=.6pt] (7+.1,2) -- (7+.1, 3);
\draw [line width=.6pt] (7+.1,3) -- (7+.1, 4);
\draw [line width=.6pt] (11+.1,5) -- (11+.1, 6);
\draw [line width=.6pt] (11+.1,6) -- (11+.1, 7);
\draw [line width=.6pt] (15+.1,4) -- (15+.1, 8);
\draw [line width=.6pt] (18+.1,2) -- (18+.1, 4);
\draw [line width=.6pt] (19+.1,9) -- (19+.1, 11);
\draw [line width=.6pt] (20+.1,5) -- (20+.1, 6);
\draw [line width=.6pt, dotted] (20-.1, 3) -- (20+.1, 5);

\draw [line width=.6pt] (0+.1,0) -- (3+.1,3);
\draw [line width=.6pt] (7+.1,1) -- (9+.1,3);
\draw [line width=.6pt] (8+.1,3) -- (9+.1,4);
\draw [line width=.6pt] (9+.1,5) -- (11+.1,7);
\draw [line width=.6pt] (14+.1,4) -- (15+.3,5);
\draw [line width=.6pt] (16+.1,2) -- (18+.1,4);
\draw [line width=.6pt] (16+.1,7) -- (17+.1,8);
\draw [line width=.6pt] (17+.1,9) -- (19+.1,11);
\draw [line width=.6pt, dotted] (20-.1,3) -- (21+.1, 5);

\filldraw [black] 
\foreach \n in {0,...,11}{(0+.1,\n) circle (2pt)};
\filldraw [black]
(1+.1,1) circle (2pt)
(2+.1,2) circle (2pt)
(3+.1,1) circle (2pt)
(3+.1,2) circle (2pt)
(3+.1,3) circle (2pt)
(6+.1,2) circle (2pt)
(7+.1,1) circle (2pt)
(7+.1,2) circle (2pt)
(7+.1,3) circle (2pt)
(7+.1,4) circle (2pt)
(8+.1,2) circle (2pt)
(8+.1,3) circle (2pt)
(9+.1,3) circle (2pt)
(9+.1,4) circle (2pt)
(9+.1,5) circle (2pt)
(10+.1,6) circle (2pt)
(11+.1,5) circle (2pt)
(11+.1,6) circle (2pt)
(11+.1,7) circle (2pt)
(14+.1,2) circle (2pt)
(14+.1,4) circle (2pt)
(15+.1,4) circle (2pt)
(15+.1,5) circle (2pt)
(15.2+.1,5) circle (2pt)
(15+.1,6) circle (2pt)
(15+.1,7) circle (2pt)
(15+.1,8) circle (2pt)
(16+.1,2) circle (2pt)
(16+.1,7) circle (2pt)
(17+.1,3) circle (2pt)
(17+.1,5) circle (2pt)
(17+.1,8) circle (2pt)
(17+.1,9) circle (2pt)
(18+.1,2) circle (2pt)
(18+.1,3) circle (2pt)
(18+.1,4) circle (2pt)
(18+.1,10) circle (2pt)
(19+.1,3) circle (2pt)
(19+.1,9) circle (2pt)
(19+.1,10) circle (2pt)
(19+.1,11) circle (2pt)
(20+.1,5) circle (2pt)
(20+.1,6) circle (2pt)
(21+.1,3) circle (2pt)
(21+.1,5) circle (2pt)
;

\draw [line width=.6pt] (2-.1,1) -- (2-.1, 2);
\draw [line width=.6pt] (2-.1,2) -- (2-.1, 3);
\draw [line width=.6pt] (6-.1,1) -- (6-.1, 2);
\draw [line width=.6pt] (6-.1,2) -- (6-.1, 3);
\draw [line width=.6pt] (6-.1,3) -- (6-.1, 4);
\draw [line width=.6pt] (10-.1,5) -- (10-.1, 6);
\draw [line width=.6pt] (10-.1,6) -- (10-.1, 7);
\draw [line width=.6pt] (14-.1,4) -- (14-.1, 8);
\draw [line width=.6pt] (17-.1,2) -- (17-.1, 4);
\draw [line width=.6pt] (18-.1,9) -- (18-.1, 11);
\draw [line width=.6pt] (19-.1,4) -- (19-.1, 5);
\draw [line width=.6pt] (21-.1,8) -- (21-.1, 10);

\draw [line width=.6pt] (0-.1,1) -- (2-.1,3);
\draw [line width=.6pt] (6-.1,1) -- (8-.1,3);
\draw [line width=.6pt] (7-.1,3) -- (8-.1,4);
\draw [line width=.6pt] (8-.1,5) -- (10-.1,7);
\draw [line width=.6pt] (13-.1,4) -- (14-.3,5);
\draw [line width=.6pt] (15-.1,2) -- (17-.1,4);
\draw [line width=.6pt] (15-.1,7) -- (16-.1,8);
\draw [line width=.6pt] (16-.1,9) -- (18-.1,11);
\draw [line width=.6pt] (19-.1,4) -- (20-.1,5);

\filldraw 
([xshift=-2.5pt,yshift=-2.5pt]0-.1,1) rectangle ++(5pt, 5pt)
([xshift=-2.5pt,yshift=-2.5pt]1-.1,2) rectangle ++(5pt, 5pt)
([xshift=-2.5pt,yshift=-2.5pt]2-.1,1) rectangle ++(5pt, 5pt)
([xshift=-2.5pt,yshift=-2.5pt]2-.1,2) rectangle ++(5pt, 5pt)
([xshift=-2.5pt,yshift=-2.5pt]2-.1,3) rectangle ++(5pt, 5pt)
([xshift=-2.5pt,yshift=-2.5pt]5-.1,2) rectangle ++(5pt, 5pt)
([xshift=-2.5pt,yshift=-2.5pt]6-.1,1) rectangle ++(5pt, 5pt)
([xshift=-2.5pt,yshift=-2.5pt]6-.1,2) rectangle ++(5pt, 5pt)
([xshift=-2.5pt,yshift=-2.5pt]6-.1,3) rectangle ++(5pt, 5pt)
([xshift=-2.5pt,yshift=-2.5pt]6-.1,4) rectangle ++(5pt, 5pt)
([xshift=-2.5pt,yshift=-2.5pt]7-.1,2) rectangle ++(5pt, 5pt)
([xshift=-2.5pt,yshift=-2.5pt]7-.1,3) rectangle ++(5pt, 5pt)
([xshift=-2.5pt,yshift=-2.5pt]8-.1,3) rectangle ++(5pt, 5pt)
([xshift=-2.5pt,yshift=-2.5pt]8-.1,4) rectangle ++(5pt, 5pt)
([xshift=-2.5pt,yshift=-2.5pt]8-.1,5) rectangle ++(5pt, 5pt)
([xshift=-2.5pt,yshift=-2.5pt]9-.1,6) rectangle ++(5pt, 5pt)
([xshift=-2.5pt,yshift=-2.5pt]10-.1,5) rectangle ++(5pt, 5pt)
([xshift=-2.5pt,yshift=-2.5pt]10-.1,6) rectangle ++(5pt, 5pt)
([xshift=-2.5pt,yshift=-2.5pt]10-.1,7) rectangle ++(5pt, 5pt)
([xshift=-2.5pt,yshift=-2.5pt]13-.1,2) rectangle ++(5pt, 5pt)
([xshift=-2.5pt,yshift=-2.5pt]13-.1,4) rectangle ++(5pt, 5pt)
([xshift=-2.5pt,yshift=-2.5pt]14-.1,4) rectangle ++(5pt, 5pt)
([xshift=-2.5pt,yshift=-2.5pt]14-.1,5) rectangle ++(5pt, 5pt)
([xshift=-2.5pt,yshift=-2.5pt]14-.2-.1,5) rectangle ++(5pt, 5pt)
([xshift=-2.5pt,yshift=-2.5pt]14-.1,6) rectangle ++(5pt, 5pt)
([xshift=-2.5pt,yshift=-2.5pt]14-.1,7) rectangle ++(5pt, 5pt)
([xshift=-2.5pt,yshift=-2.5pt]14-.1,8) rectangle ++(5pt, 5pt)
([xshift=-2.5pt,yshift=-2.5pt]15-.1,2) rectangle ++(5pt, 5pt)
([xshift=-2.5pt,yshift=-2.5pt]15-.1,7) rectangle ++(5pt, 5pt)
([xshift=-2.5pt,yshift=-2.5pt]16-.1,3) rectangle ++(5pt, 5pt)
([xshift=-2.5pt,yshift=-2.5pt]16-.1,5) rectangle ++(5pt, 5pt)
([xshift=-2.5pt,yshift=-2.5pt]16-.1,8) rectangle ++(5pt, 5pt)
([xshift=-2.5pt,yshift=-2.5pt]16-.1,9) rectangle ++(5pt, 5pt)
([xshift=-2.5pt,yshift=-2.5pt]17-.1,2) rectangle ++(5pt, 5pt)
([xshift=-2.5pt,yshift=-2.5pt]17-.1,3) rectangle ++(5pt, 5pt)
([xshift=-2.5pt,yshift=-2.5pt]17-.1,4) rectangle ++(5pt, 5pt)
([xshift=-2.5pt,yshift=-2.5pt]17-.1,10) rectangle ++(5pt, 5pt)
([xshift=-2.5pt,yshift=-2.5pt]18-.1,3) rectangle ++(5pt, 5pt)
([xshift=-2.5pt,yshift=-2.5pt]18-.1,9) rectangle ++(5pt, 5pt)
([xshift=-2.5pt,yshift=-2.5pt]18-.1,10) rectangle ++(5pt, 5pt)
([xshift=-2.5pt,yshift=-2.5pt]18-.1,11) rectangle ++(5pt, 5pt)
([xshift=-2.5pt,yshift=-2.5pt]19-.1,4) rectangle ++(5pt, 5pt)
([xshift=-2.5pt,yshift=-2.5pt]19-.1,5) rectangle ++(5pt, 5pt)
([xshift=-2.5pt,yshift=-2.5pt]20-.1,3) rectangle ++(5pt, 5pt)
([xshift=-2.5pt,yshift=-2.5pt]20-.1,5) rectangle ++(5pt, 5pt)
([xshift=-2.5pt,yshift=-2.5pt]21-.1,4) rectangle ++(5pt, 5pt)
([xshift=-2.5pt,yshift=-2.5pt]21-.1,8) rectangle ++(5pt, 5pt)
([xshift=-2.5pt,yshift=-2.5pt]21-.1,9) rectangle ++(5pt, 5pt)
([xshift=-2.5pt,yshift=-2.5pt]21-.1,10) rectangle ++(5pt, 5pt)
;

\filldraw 
(0-.1,1) circle (2pt)
(1-.1,2) circle (2pt)
(2-.1,1) circle (2pt)
(2-.1,2) circle (2pt)
(2-.1,3) circle (2pt)
(5-.1,2) circle (2pt)
(6-.1,1) circle (2pt)
(6-.1,2) circle (2pt)
(6-.1,3) circle (2pt)
(6-.1,4) circle (2pt)
(7-.1,2) circle (2pt)
(7-.1,3) circle (2pt)
(8-.1,3) circle (2pt)
(8-.1,4) circle (2pt)
(8-.1,5) circle (2pt)
(9-.1,6) circle (2pt)
(10-.1,5) circle (2pt)
(10-.1,6) circle (2pt)
(10-.1,7) circle (2pt)
(13-.1,2) circle (2pt)
(13-.1,4) circle (2pt)
(14-.1,4) circle (2pt)
(14-.1,5) circle (2pt)
(14-.2-.1,5) circle (2pt)
(14-.1,6) circle (2pt)
(14-.1,7) circle (2pt)
(14-.1,8) circle (2pt)
(15-.1,2) circle (2pt)
(15-.1,7) circle (2pt)
(16-.1,3) circle (2pt)
(16-.1,5) circle (2pt)
(16-.1,8) circle (2pt)
(16-.1,9) circle (2pt)
(17-.1,2) circle (2pt)
(17-.1,3) circle (2pt)
(17-.1,4) circle (2pt)
(17-.1,10) circle (2pt)
(18-.1,3) circle (2pt)
(18-.1,9) circle (2pt)
(18-.1,10) circle (2pt)
(18-.1,11) circle (2pt)
(19-.1,4) circle (2pt)
(19-.1,5) circle (2pt)
(20-.1,3) circle (2pt)
(20-.1,5) circle (2pt)
(21-.1,4) circle (2pt)
(21-.1,8) circle (2pt)
(21-.1,9) circle (2pt)
(21-.1,10) circle (2pt)
;

\draw (0+.1, 1) node [anchor=west] {$h_0$};
\draw (1+.1, 1) node [anchor=west] {$\tau h_1$};
\draw (3+.1, 1) node [anchor=west] {$\tau^2 h_2$};
\draw (6, 2+.15) node [anchor=west] {$\tau^4 h_2^2$};
\draw (7+.1, 1) node [anchor=west] {$\tau^4 h_3$};
\draw (8+.1, 3-.05) node [anchor=west] {$\tau^5c_0$};
\draw (9, 5) node [anchor=north] {$\tau^5 Ph_1$};
\draw (11, 5) node [anchor=north] {$\tau^6 Ph_2$};
\draw (14, 2+.1) node [anchor=north west] {$\tau^8h_3^2$};
\draw (14, 4+.1) node [anchor=north west] {$\tau^8d_0$};
\draw (16+.1, 7+.1) node [anchor=north west] {$\tau^9Pc_0$};
\draw (17, 9) node [anchor=north] {$\tau^{9}P^2h_1$};
\draw (18, 2) node [anchor=north] {$\tau^{10}h_2h_4$};
\draw (19, 9) node [anchor=north] {$\tau^{10}P^2h_2$};
\draw (19, 3) node [anchor=north] {$\tau^{11}c_1$};
\draw (20, 5) node [anchor=south west] {$\tau^{12}h_2e_0$};
\draw (21, 3) node [anchor=north] {$\tau^{12}h_2^2h_4$};
\draw (21+.15, 4) node [anchor=north] {$u\tau^{12}h_2c_1$};

\end{tikzpicture}
}
\caption{$E_{\infty}$ page of MASS for $\bbF_q$ with
  $q \equiv 1 \bmod 4$, weight 0}
  \label{fig:1mod4-infty}
\end{figure}

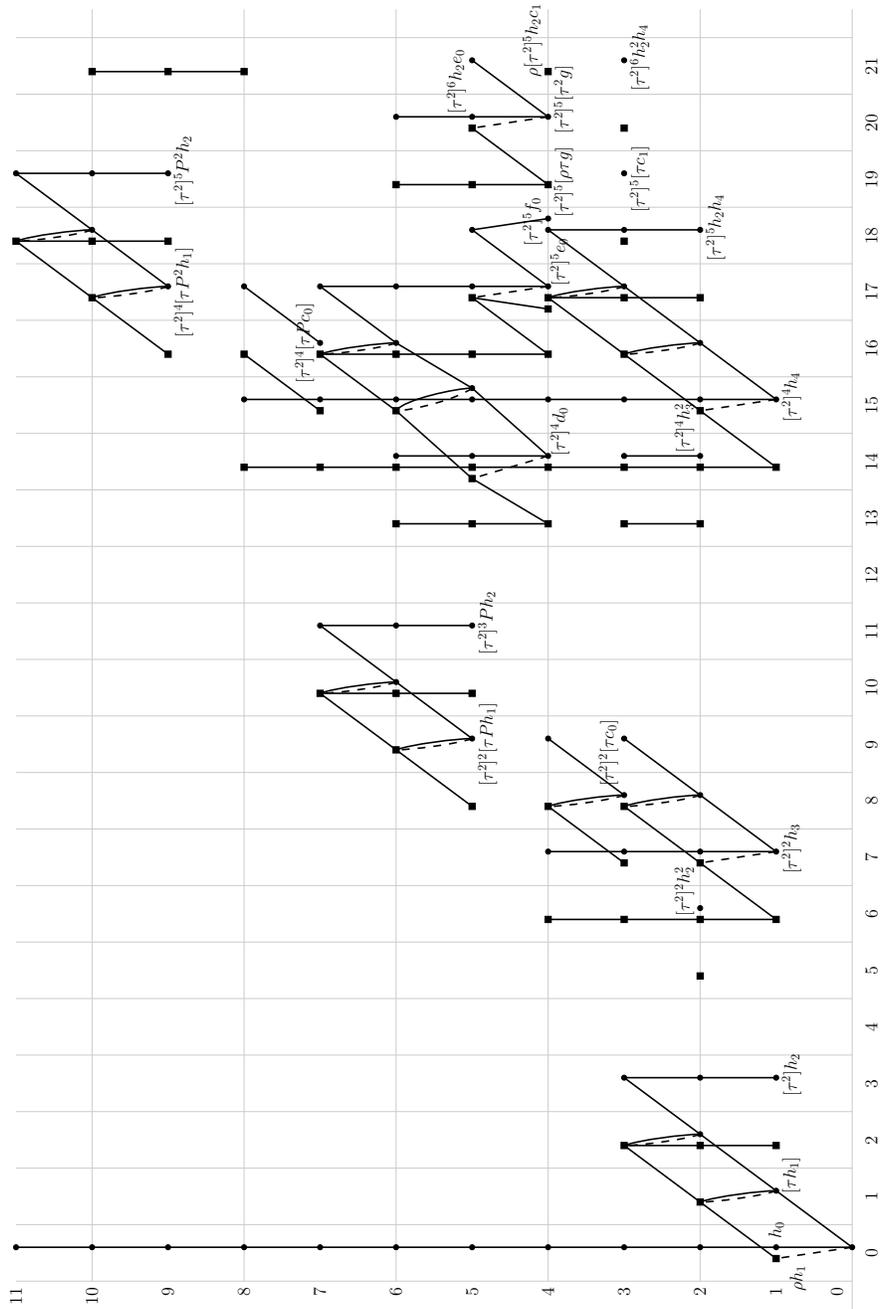
\begin{figure}[ht!]
  \centering
{\LARGE
\begin{tikzpicture}[x=1.6cm, y=2.15cm, scale=.47, rotate=90, transform shape]
\draw[gray!40] (-1, 0) -- (22, 0);
\foreach \n in {1,...,5}{\draw[line width=.25pt, gray!40] (-0.5, 2*\n) -- (22, 2*\n);}
\foreach \n in {0,...,22}{\draw[line width=.25pt, gray!40] (\n-.5,0) -- (\n-.5,11);}
\foreach \n in {0,..., 21}{\draw (\n, -.1) node [anchor=north] {\n};}
\draw (-0.5, .2) node [anchor=east] {0};
\foreach \n in {1,...,11}{\draw (-0.5, \n) node [anchor=east] {\n};}

\foreach \n in {0,...,10}{\draw [line width=.6pt] (.1,\n) -- (.1,{1+\n});}
\draw [line width=.6pt] (1+.1,1) .. controls (1+.1, 1+.2) and (1,2-.2) .. (1-.1,2);
\draw [line width=.6pt] (2+.1,2) .. controls (2+.1, 2+.2) and (2,3-.2) .. (2-.1,3);
\draw [line width=.6pt] (3+.1,1) -- (3+.1, 2);
\draw [line width=.6pt] (3+.1,2) -- (3+.1, 3);
\draw [line width=.6pt] (7+.1,1) -- (7+.1, 2);
\draw [line width=.6pt] (7+.1,2) -- (7+.1, 3);
\draw [line width=.6pt] (7+.1,3) -- (7+.1, 4);
\draw [line width=.6pt] (8+.1,2) .. controls (8+.1, 2+.2) and (8,3-.2) .. (8-.1,3);
\draw [line width=.6pt] (8+.1,3) .. controls (8+.1, 3+.2) and (8,4-.2) .. (8-.1,4);
\draw [line width=.6pt] (9+.1,5) .. controls (9+.1, 5+.2) and (9,6-.2) .. (9-.1,6);
\draw [line width=.6pt] (10+.1,6) .. controls (10+.1, 6+.2) and (10,7-.2) .. (10-.1,7);
\draw [line width=.6pt] (11+.1,5) -- (11+.1, 6);
\draw [line width=.6pt] (11+.1,6) -- (11+.1, 7);
\draw [line width=.6pt] (14+.1,2) -- (14+.1, 3);
\draw [line width=.6pt] (14+.1,4) -- (14+.1, 5);
\draw [line width=.6pt] (14+.1,5) -- (14+.1, 6);
\draw [line width=.6pt] (15+.1,1) -- (15+.1, 8);
\draw [line width=.6pt] (15+.3,5) .. controls (15+.3, 5+.2) and (15+.1,6-.1) .. (15-.1,6);
\draw [line width=.6pt] (16+.1,2) .. controls (16+.1, 2+.2) and (16,3-.2) .. (16-.1,3);
\draw [line width=.6pt] (16+.1,6) .. controls (16+.1, 6+.2) and (16,7-.2) .. (16-.1,7);
\draw [line width=.6pt] (17+.1,3) .. controls (17+.1, 3+.2) and (17,4-.2) .. (17-.1,4);
\draw [line width=.6pt] (17+.1,4) -- (17+.1, 7);
\draw [line width=.6pt] (17+.1,9) .. controls (17+.1, 9+.2) and (17,10-.2) .. (17-.1,10);
\draw [line width=.6pt] (18+.1,2) -- (18+.1, 4);
\draw [line width=.6pt] (18+.3,4) -- (18+.1, 5);
\draw [line width=.6pt] (18+.1,10) .. controls (18+.1, 10+.2) and (18,11-.2) .. (18-.1,11);
\draw [line width=.6pt] (19+.1,9) -- (19+.1, 11);
\draw [line width=.6pt] (20+.1,4) -- (20+.1, 6);

\draw [line width=.6pt,  dashed] (0+.1,0) -- (0-.1,1);
\draw [line width=.6pt,  dashed] (1+.1,1) .. controls (1+.1-.1,1+.2) and (1-.1,2-.2) .. (1-.1,2);
\draw [line width=.6pt,  dashed] (2+.1,2) .. controls (2+.1-.1,2+.2) and (2-.1,3-.2) .. (2-.1,3);
\draw [line width=.6pt,  dashed] (7+.1,1) -- (7-.1,2);
\draw [line width=.6pt,  dashed] (8+.1,2) .. controls (8+.1-.1,2+.2) and (8-.1,3-.2) .. (8-.1,3);
\draw [line width=.6pt,  dashed] (8+.1,3) .. controls (8+.1-.1,3+.2) and (8-.1,4-.2) .. (8-.1,4);
\draw [line width=.6pt,  dashed] (9+.1,5) .. controls (9+.1-.1,5+.2) and (9-.1,6-.2) .. (9-.1,6);
\draw [line width=.6pt,  dashed] (10+.1,6) .. controls (10+.1-.1,6+.2) and (10-.1,7-.2) .. (10-.1,7);
\draw [line width=.6pt,  dashed] (14+.1,4) -- (14-.3,5);
\draw [line width=.6pt,  dashed] (15+.1,1) -- (15-.1,2);
\draw [line width=.6pt,  dashed] (15+.3,5) .. controls (15+.3-.2,5+.2) and (15-.1,6-.2) .. (15-.1,6);
\draw [line width=.6pt,  dashed] (16+.1,2) .. controls (16+.1-.1,2+.2) and (16-.1,3-.2) .. (16-.1,3);
\draw [line width=.6pt,  dashed] (16+.1,6) .. controls (16+.1-.1,6+.2) and (16-.1,7-.2) .. (16-.1,7);
\draw [line width=.6pt,  dashed] (17+.1,3) .. controls (17+.1-.1,3+.2) and (17-.1,4-.2) .. (17-.1,4);
\draw [line width=.6pt,  dashed] (17+.1,4) -- (17-.1,5);
\draw [line width=.6pt,  dashed] (17+.1,9) .. controls (17+.1-.1,9+.2) and (17-.1,10-.2) .. (17-.1,10);
\draw [line width=.6pt,  dashed] (18+.1,10) .. controls (18+.1-.1,10+.2) and (18-.1,11-.2) .. (18-.1,11);
\draw [line width=.6pt,  dashed] (20+.1,4) -- (20-.1,5);

\draw [line width=.6pt] (0+.1,0) -- (3+.1,3);
\draw [line width=.6pt] (7+.1,1) -- (9+.1,3);
\draw [line width=.6pt] (8+.1,3) -- (9+.1,4);
\draw [line width=.6pt] (9+.1,5) -- (11+.1,7);
\draw [line width=.6pt] (14+.1,4) -- (15+.3,5);
\draw [line width=.6pt] (15+.3,5) -- (16+.1,6);
\draw [line width=.6pt] (16+.1,6) -- (17+.1,7);
\draw [line width=.6pt] (15+.1,1) -- (18+.1,4);
\draw [line width=.6pt] (16+.1,7) -- (17+.1,8);
\draw [line width=.6pt] (17+.1,9) -- (19+.1,11);
\draw [line width=.6pt] (17+.1,4) -- (18+.1,5);
\draw [line width=.6pt] (20+.1,4) -- (21+.1,5);

\filldraw [black] 
\foreach \n in {0,...,11}{(0+.1,\n) circle (2pt)};
\filldraw [black]
(1+.1,1) circle (2pt)
(2+.1,2) circle (2pt)
(3+.1,1) circle (2pt)
(3+.1,2) circle (2pt)
(3+.1,3) circle (2pt)
(6+.1,2) circle (2pt) 
(7+.1,1) circle (2pt)
(7+.1,2) circle (2pt)
(7+.1,3) circle (2pt)
(7+.1,4) circle (2pt)
(8+.1,2) circle (2pt)
(8+.1,3) circle (2pt)
(9+.1,3) circle (2pt)
(9+.1,4) circle (2pt)
(9+.1,5) circle (2pt)
(10+.1,6) circle (2pt)
(11+.1,5) circle (2pt)
(11+.1,6) circle (2pt)
(11+.1,7) circle (2pt)
(14+.1,2) circle (2pt)
(14+.1,3) circle (2pt)
(14+.1,4) circle (2pt)
(14+.1,5) circle (2pt)
(14+.1,6) circle (2pt)
(15+.1,1) circle (2pt)
(15+.1,2) circle (2pt)
(15+.1,3) circle (2pt)
(15+.1,4) circle (2pt)
(15+.1,5) circle (2pt)
(15.2+.1,5) circle (2pt)
(15+.1,6) circle (2pt)
(15+.1,7) circle (2pt)
(15+.1,8) circle (2pt)
(16+.1,2) circle (2pt)
(16+.1,6) circle (2pt)
(16+.1,7) circle (2pt)
(17+.1,3) circle (2pt)
(17+.1,4) circle (2pt)
(17+.1,5) circle (2pt)
(17+.1,6) circle (2pt)
(17+.1,7) circle (2pt)
(17+.1,8) circle (2pt)
(17+.1,9) circle (2pt)
(18+.1,2) circle (2pt)
(18+.1,3) circle (2pt)
(18+.1,4) circle (2pt)
(18.2+.1,4) circle (2pt)
(18+.1,5) circle (2pt)
(18+.1,10) circle (2pt)
(19+.1,3) circle (2pt)
(19+.1,9) circle (2pt)
(19+.1,10) circle (2pt)
(19+.1,11) circle (2pt)
(20+.1,4) circle (2pt)
(20+.1,5) circle (2pt)
(20+.1,6) circle (2pt)
(21+.1,3) circle (2pt)
(21+.1,5) circle (2pt)
;

\draw [line width=.6pt] (2-.1,1) -- (2-.1, 2);
\draw [line width=.6pt] (2-.1,2) -- (2-.1, 3);
\draw [line width=.6pt] (6-.1,1) -- (6-.1, 2);
\draw [line width=.6pt] (6-.1,2) -- (6-.1, 3);
\draw [line width=.6pt] (6-.1,3) -- (6-.1, 4);
\draw [line width=.6pt] (10-.1,5) -- (10-.1, 6);
\draw [line width=.6pt] (10-.1,6) -- (10-.1, 7);
\draw [line width=.6pt] (13-.1,2) -- (13-.1, 3);
\draw [line width=.6pt] (13-.1,4) -- (13-.1, 5);
\draw [line width=.6pt] (13-.1,5) -- (13-.1, 6);
\draw [line width=.6pt] (14-.1,1) -- (14-.1, 8);
\draw [line width=.6pt] (16-.1,4) -- (16-.1, 7);
\draw [line width=.6pt] (17-.1,2) -- (17-.1, 4);
\draw [line width=.6pt] (17-.3,4) -- (17-.1, 5);
\draw [line width=.6pt] (18-.1,9) -- (18-.1, 11);
\draw [line width=.6pt] (19-.1,4) -- (19-.1, 6);
\draw [line width=.6pt] (21-.1,8) -- (21-.1, 10);

\draw [line width=.6pt] (0-.1,1) -- (2-.1,3);
\draw [line width=.6pt] (6-.1,1) -- (8-.1,3);
\draw [line width=.6pt] (7-.1,3) -- (8-.1,4);
\draw [line width=.6pt] (8-.1,5) -- (10-.1,7);
\draw [line width=.6pt] (13-.1,4) -- (14-.3,5);
\draw [line width=.6pt] (14-.3,5) -- (15-.1,6);
\draw [line width=.6pt] (15-.1,6) -- (16-.1,7);
\draw [line width=.6pt] (14-.1,1) -- (17-.1,4);
\draw [line width=.6pt] (15-.1,7) -- (16-.1,8);
\draw [line width=.6pt] (16-.1,9) -- (18-.1,11);
\draw [line width=.6pt] (16-.1,4) -- (17-.1,5);
\draw [line width=.6pt] (19-.1,4) -- (20-.1,5);

\filldraw 
([xshift=-2.5pt,yshift=-2.5pt]0-.1,1) rectangle ++(5pt, 5pt)
([xshift=-2.5pt,yshift=-2.5pt]1-.1,2) rectangle ++(5pt, 5pt)
([xshift=-2.5pt,yshift=-2.5pt]2-.1,1) rectangle ++(5pt, 5pt)
([xshift=-2.5pt,yshift=-2.5pt]2-.1,2) rectangle ++(5pt, 5pt)
([xshift=-2.5pt,yshift=-2.5pt]2-.1,3) rectangle ++(5pt, 5pt)
([xshift=-2.5pt,yshift=-2.5pt]5-.1,2) rectangle ++(5pt, 5pt)
([xshift=-2.5pt,yshift=-2.5pt]6-.1,1) rectangle ++(5pt, 5pt)
([xshift=-2.5pt,yshift=-2.5pt]6-.1,2) rectangle ++(5pt, 5pt)
([xshift=-2.5pt,yshift=-2.5pt]6-.1,3) rectangle ++(5pt, 5pt)
([xshift=-2.5pt,yshift=-2.5pt]6-.1,4) rectangle ++(5pt, 5pt)
([xshift=-2.5pt,yshift=-2.5pt]7-.1,2) rectangle ++(5pt, 5pt)
([xshift=-2.5pt,yshift=-2.5pt]7-.1,3) rectangle ++(5pt, 5pt)
([xshift=-2.5pt,yshift=-2.5pt]8-.1,3) rectangle ++(5pt, 5pt)
([xshift=-2.5pt,yshift=-2.5pt]8-.1,4) rectangle ++(5pt, 5pt)
([xshift=-2.5pt,yshift=-2.5pt]8-.1,5) rectangle ++(5pt, 5pt)
([xshift=-2.5pt,yshift=-2.5pt]9-.1,6) rectangle ++(5pt, 5pt)
([xshift=-2.5pt,yshift=-2.5pt]10-.1,5) rectangle ++(5pt, 5pt)
([xshift=-2.5pt,yshift=-2.5pt]10-.1,6) rectangle ++(5pt, 5pt)
([xshift=-2.5pt,yshift=-2.5pt]10-.1,7) rectangle ++(5pt, 5pt)
([xshift=-2.5pt,yshift=-2.5pt]13-.1,2) rectangle ++(5pt, 5pt)
([xshift=-2.5pt,yshift=-2.5pt]13-.1,3) rectangle ++(5pt, 5pt)
([xshift=-2.5pt,yshift=-2.5pt]13-.1,4) rectangle ++(5pt, 5pt)
([xshift=-2.5pt,yshift=-2.5pt]13-.1,5) rectangle ++(5pt, 5pt)
([xshift=-2.5pt,yshift=-2.5pt]13-.1,6) rectangle ++(5pt, 5pt)
([xshift=-2.5pt,yshift=-2.5pt]14-.1,1) rectangle ++(5pt, 5pt)
([xshift=-2.5pt,yshift=-2.5pt]14-.1,2) rectangle ++(5pt, 5pt)
([xshift=-2.5pt,yshift=-2.5pt]14-.1,3) rectangle ++(5pt, 5pt)
([xshift=-2.5pt,yshift=-2.5pt]14-.1,4) rectangle ++(5pt, 5pt)
([xshift=-2.5pt,yshift=-2.5pt]14-.1,5) rectangle ++(5pt, 5pt)
([xshift=-2.5pt,yshift=-2.5pt]14-.2-.1,5) rectangle ++(5pt, 5pt)
([xshift=-2.5pt,yshift=-2.5pt]14-.1,6) rectangle ++(5pt, 5pt)
([xshift=-2.5pt,yshift=-2.5pt]14-.1,7) rectangle ++(5pt, 5pt)
([xshift=-2.5pt,yshift=-2.5pt]14-.1,8) rectangle ++(5pt, 5pt)
([xshift=-2.5pt,yshift=-2.5pt]15-.1,2) rectangle ++(5pt, 5pt)
([xshift=-2.5pt,yshift=-2.5pt]15-.1,6) rectangle ++(5pt, 5pt)
([xshift=-2.5pt,yshift=-2.5pt]15-.1,7) rectangle ++(5pt, 5pt)
([xshift=-2.5pt,yshift=-2.5pt]16-.1,3) rectangle ++(5pt, 5pt)
([xshift=-2.5pt,yshift=-2.5pt]16-.1,4) rectangle ++(5pt, 5pt)
([xshift=-2.5pt,yshift=-2.5pt]16-.1,5) rectangle ++(5pt, 5pt)
([xshift=-2.5pt,yshift=-2.5pt]16-.1,6) rectangle ++(5pt, 5pt)
([xshift=-2.5pt,yshift=-2.5pt]16-.1,7) rectangle ++(5pt, 5pt)
([xshift=-2.5pt,yshift=-2.5pt]16-.1,8) rectangle ++(5pt, 5pt)
([xshift=-2.5pt,yshift=-2.5pt]16-.1,9) rectangle ++(5pt, 5pt)
([xshift=-2.5pt,yshift=-2.5pt]17-.1,2) rectangle ++(5pt, 5pt)
([xshift=-2.5pt,yshift=-2.5pt]17-.1,3) rectangle ++(5pt, 5pt)
([xshift=-2.5pt,yshift=-2.5pt]17-.1,4) rectangle ++(5pt, 5pt)
([xshift=-2.5pt,yshift=-2.5pt]17-.2-.1,4) rectangle ++(5pt, 5pt)
([xshift=-2.5pt,yshift=-2.5pt]17-.1,5) rectangle ++(5pt, 5pt)
([xshift=-2.5pt,yshift=-2.5pt]17-.1,10) rectangle ++(5pt, 5pt)
([xshift=-2.5pt,yshift=-2.5pt]18-.1,3) rectangle ++(5pt, 5pt)
([xshift=-2.5pt,yshift=-2.5pt]18-.1,9) rectangle ++(5pt, 5pt)
([xshift=-2.5pt,yshift=-2.5pt]18-.1,10) rectangle ++(5pt, 5pt)
([xshift=-2.5pt,yshift=-2.5pt]18-.1,11) rectangle ++(5pt, 5pt)
([xshift=-2.5pt,yshift=-2.5pt]19-.1,4) rectangle ++(5pt, 5pt)
([xshift=-2.5pt,yshift=-2.5pt]19-.1,5) rectangle ++(5pt, 5pt)
([xshift=-2.5pt,yshift=-2.5pt]19-.1,6) rectangle ++(5pt, 5pt)
([xshift=-2.5pt,yshift=-2.5pt]20-.1,3) rectangle ++(5pt, 5pt)
([xshift=-2.5pt,yshift=-2.5pt]20-.1,5) rectangle ++(5pt, 5pt)
([xshift=-2.5pt,yshift=-2.5pt]21-.1,4) rectangle ++(5pt, 5pt)
([xshift=-2.5pt,yshift=-2.5pt]21-.1,8) rectangle ++(5pt, 5pt)
([xshift=-2.5pt,yshift=-2.5pt]21-.1,9) rectangle ++(5pt, 5pt)
([xshift=-2.5pt,yshift=-2.5pt]21-.1,10) rectangle ++(5pt, 5pt)
;

\draw (0-.1, 1-.1) node [anchor=north east] {$\rho h_1$};
\draw (0+.15, 1) node [anchor=west] {$h_0$};
\draw (1-.05, 1) node [anchor=north west] {$[\tau h_1]$};
\draw (3+.15, 1) node [anchor=north] {$[\tau^2]h_2$};
\draw (6+.4, 2) node [anchor=south ] {$[\tau^2]^2h_2^2$};
\draw (7+.15, 1) node [anchor=north] {$[\tau^2]^2h_3$};
\draw (8+.15, 3) node [anchor=south west] {$[\tau^2]^2[\tau c_0]$};
\draw (9, 5) node [anchor=north] {$[\tau^2]^2[\tau Ph_1]$};
\draw (11+.15, 5) node [anchor=north] {$[\tau^2]^3Ph_2$};
\draw (14, 4+.05) node [anchor=north west] {$[\tau^2]^4d_0$};
\draw (14+.05, 2) node [anchor=south west] {$[\tau^2]^4h_3^2$};
\draw (15+.15, 1) node [anchor=north] {$[\tau^2]^4h_4$};
\draw (16+.1, 7) node [anchor=south] {$[\tau^2]^4[\tau Pc_0]$};
\draw (17, 4+.05) node [anchor=north west] {$[\tau^2]^5e_0$};
\draw (17, 9) node [anchor=north] {$[\tau^2]^4[\tau P^2h_1]$};
\draw (18+.25, 4) node [anchor=south] {$[\tau^2]^5f_0$};
\draw (18+.15, 2) node [anchor=north] {$[\tau^2]^5h_2h_4$};
\draw (19+.15, 9) node [anchor=north] {$[\tau^2]^5P^2h_2$};
\draw (19, 3) node [anchor=north] {$[\tau^2]^5[\tau c_1]$};
\draw (19-.05, 4) node [anchor=north] {$[\tau^2]^5[\rho\tau g]$};
\draw (20-.3, 4) node [anchor=north west] {$[\tau^2]^5[\tau^2 g]$};
\draw (20+.05, 5) node [anchor=south west] {$[\tau^2]^6h_2e_0$};
\draw (21+.15, 3) node [anchor=north] {$[\tau^2]^6h_2^2h_4$};
\draw (21-.3, 4) node [anchor=south west] {$\rho[\tau^2]^5 h_2 c_1$};
\end{tikzpicture}
}
  \caption{$E_2$ page of MASS for $\bbF_q$ with $q \equiv 3 \bmod 4$, weight 0}
  \label{fig:3mod4}
\end{figure}

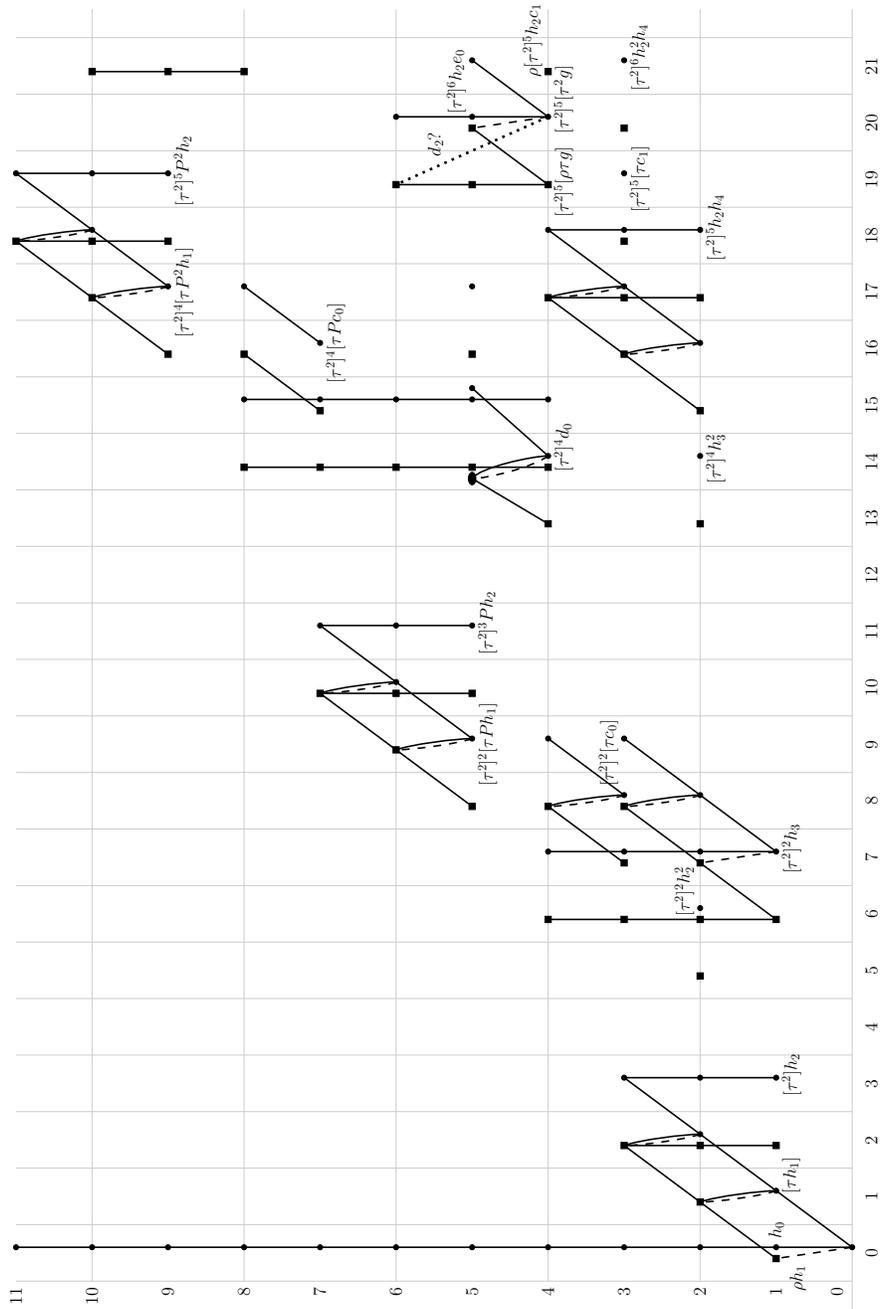
\begin{figure}[ht!]
  \centering
{\LARGE
\begin{tikzpicture}[x=1.6cm, y=2.15cm, scale=.47, rotate=90, transform shape]
\draw[gray!40] (-1, 0) -- (22, 0);
\foreach \n in {1,...,5}{\draw[line width=.25pt, gray!40] (-0.5, 2*\n) -- (22, 2*\n);}
\foreach \n in {0,...,22}{\draw[line width=.25pt, gray!40] (\n-.5,0) -- (\n-.5,11);}
\foreach \n in {0,..., 21}{\draw (\n, -.1) node [anchor=north] {\n};}
\draw (-0.5, .2) node [anchor=east] {0};
\foreach \n in {1,...,11}{\draw (-0.5, \n) node [anchor=east] {\n};}

\foreach \n in {0,...,10}{\draw [line width=.6pt] (.1,\n) -- (.1,{1+\n});}
\draw [line width=.6pt] (1+.1,1) .. controls (1+.1, 1+.2) and (1,2-.2) .. (1-.1,2);
\draw [line width=.6pt] (2+.1,2) .. controls (2+.1, 2+.2) and (2,3-.2) .. (2-.1,3);
\draw [line width=.6pt] (3+.1,1) -- (3+.1, 2);
\draw [line width=.6pt] (3+.1,2) -- (3+.1, 3);
\draw [line width=.6pt] (7+.1,1) -- (7+.1, 2);
\draw [line width=.6pt] (7+.1,2) -- (7+.1, 3);
\draw [line width=.6pt] (7+.1,3) -- (7+.1, 4);
\draw [line width=.6pt] (8+.1,2) .. controls (8+.1, 2+.2) and (8,3-.2) .. (8-.1,3);
\draw [line width=.6pt] (8+.1,3) .. controls (8+.1, 3+.2) and (8,4-.2) .. (8-.1,4);
\draw [line width=.6pt] (9+.1,5) .. controls (9+.1, 5+.2) and (9,6-.2) .. (9-.1,6);
\draw [line width=.6pt] (10+.1,6) .. controls (10+.1, 6+.2) and (10,7-.2) .. (10-.1,7);
\draw [line width=.6pt] (11+.1,5) -- (11+.1, 6);
\draw [line width=.6pt] (11+.1,6) -- (11+.1, 7);
\draw [line width=.6pt] (14+.1,4) .. controls (14+.1, 4+.2) and (14-.1,5-.1) .. (14-.3,5);
\draw [line width=.6pt] (15+.1,4) -- (15+.1, 8);
\draw [line width=.6pt] (16+.1,2) .. controls (16+.1, 2+.2) and (16,3-.2) .. (16-.1,3);
\draw [line width=.6pt] (17+.1,3) .. controls (17+.1, 3+.2) and (17,4-.2) .. (17-.1,4);
\draw [line width=.6pt] (17+.1,9) .. controls (17+.1, 9+.2) and (17,10-.2) .. (17-.1,10);
\draw [line width=.6pt] (18+.1,2) -- (18+.1, 4);
\draw [line width=.6pt] (18+.1,10) .. controls (18+.1, 10+.2) and (18,11-.2) .. (18-.1,11);
\draw [line width=.6pt] (19+.1,9) -- (19+.1, 11);
\draw [line width=.6pt] (20+.1,4) -- (20+.1, 6);

\draw [line width=.6pt,  dashed] (0+.1,0) -- (0-.1,1);
\draw [line width=.6pt,  dashed] (1+.1,1) .. controls (1+.1-.1,1+.2) and (1-.1,2-.2) .. (1-.1,2);
\draw [line width=.6pt,  dashed] (2+.1,2) .. controls (2+.1-.1,2+.2) and (2-.1,3-.2) .. (2-.1,3);
\draw [line width=.6pt,  dashed] (7+.1,1) -- (7-.1,2);
\draw [line width=.6pt,  dashed] (8+.1,2) .. controls (8+.1-.1,2+.2) and (8-.1,3-.2) .. (8-.1,3);
\draw [line width=.6pt,  dashed] (8+.1,3) .. controls (8+.1-.1,3+.2) and (8-.1,4-.2) .. (8-.1,4);
\draw [line width=.6pt,  dashed] (9+.1,5) .. controls (9+.1-.1,5+.2) and (9-.1,6-.2) .. (9-.1,6);
\draw [line width=.6pt,  dashed] (10+.1,6) .. controls (10+.1-.1,6+.2) and (10-.1,7-.2) .. (10-.1,7);
\draw [line width=.6pt,  dashed] (14+.1,4) .. controls (14+.1-.2,4+.2) and (14-.3,5-.2) .. (14-.3,5);
\draw [line width=.6pt,  dashed] (16+.1,2) .. controls (16+.1-.1,2+.2) and (16-.1,3-.2) .. (16-.1,3);
\draw [line width=.6pt,  dashed] (17+.1,3) .. controls (17+.1-.1,3+.2) and (17-.1,4-.2) .. (17-.1,4);
\draw [line width=.6pt,  dashed] (17+.1,9) .. controls (17+.1-.1,9+.2) and (17-.1,10-.2) .. (17-.1,10);
\draw [line width=.6pt,  dashed] (18+.1,10) .. controls (18+.1-.1,10+.2) and (18-.1,11-.2) .. (18-.1,11);
\draw [line width=.6pt,  dashed] (20+.1,4) -- (20-.1,5);

\draw [line width=.6pt] (0+.1,0) -- (3+.1,3);
\draw [line width=.6pt] (7+.1,1) -- (9+.1,3);
\draw [line width=.6pt] (8+.1,3) -- (9+.1,4);
\draw [line width=.6pt] (9+.1,5) -- (11+.1,7);
\draw [line width=.6pt] (14+.1,4) -- (15+.3,5);
\draw [line width=.6pt] (16+.1,2) -- (18+.1,4);
\draw [line width=.6pt] (16+.1,7) -- (17+.1,8);
\draw [line width=.6pt] (17+.1,9) -- (19+.1,11);
\draw [line width=.6pt] (20+.1,4) -- (21+.1,5);

\filldraw [black] 
\foreach \n in {0,...,11}{(0+.1,\n) circle (2pt)};
\filldraw [black]
(1+.1,1) circle (2pt)
(2+.1,2) circle (2pt)
(3+.1,1) circle (2pt)
(3+.1,2) circle (2pt)
(3+.1,3) circle (2pt)
(6+.1,2) circle (2pt) 
(7+.1,1) circle (2pt)
(7+.1,2) circle (2pt)
(7+.1,3) circle (2pt)
(7+.1,4) circle (2pt)
(8+.1,2) circle (2pt)
(8+.1,3) circle (2pt)
(9+.1,3) circle (2pt)
(9+.1,4) circle (2pt)
(9+.1,5) circle (2pt)
(10+.1,6) circle (2pt)
(11+.1,5) circle (2pt)
(11+.1,6) circle (2pt)
(11+.1,7) circle (2pt)
(14+.1,2) circle (2pt)
(14+.1,4) circle (2pt)
(15+.1,4) circle (2pt)
(15+.1,5) circle (2pt)
(15.2+.1,5) circle (2pt)
(15+.1,6) circle (2pt)
(15+.1,7) circle (2pt)
(15+.1,8) circle (2pt)
(16+.1,2) circle (2pt)
(16+.1,7) circle (2pt)
(17+.1,3) circle (2pt)
(17+.1,5) circle (2pt)
(17+.1,8) circle (2pt)
(17+.1,9) circle (2pt)
(18+.1,2) circle (2pt)
(18+.1,3) circle (2pt)
(18+.1,4) circle (2pt)
(18+.1,10) circle (2pt)
(19+.1,3) circle (2pt)
(19+.1,9) circle (2pt)
(19+.1,10) circle (2pt)
(19+.1,11) circle (2pt)
(20+.1,4) circle (2pt)
(20+.1,5) circle (2pt)
(20+.1,6) circle (2pt)
(21+.1,3) circle (2pt)
(21+.1,5) circle (2pt)
;

\draw [line width=.6pt] (2-.1,1) -- (2-.1, 2);
\draw [line width=.6pt] (2-.1,2) -- (2-.1, 3);
\draw [line width=.6pt] (6-.1,1) -- (6-.1, 2);
\draw [line width=.6pt] (6-.1,2) -- (6-.1, 3);
\draw [line width=.6pt] (6-.1,3) -- (6-.1, 4);
\draw [line width=.6pt] (10-.1,5) -- (10-.1, 6);
\draw [line width=.6pt] (10-.1,6) -- (10-.1, 7);
\draw [line width=.6pt] (14-.1,4) -- (14-.1, 8);
\draw [line width=.6pt] (17-.1,2) -- (17-.1, 4);
\draw [line width=.6pt] (18-.1,9) -- (18-.1, 11);
\draw [line width=.6pt] (19-.1,4) -- (19-.1, 6);
\draw [line width=.6pt] (21-.1,8) -- (21-.1, 10);

\draw [line width=.6pt] (0-.1,1) -- (2-.1,3);
\draw [line width=.6pt] (6-.1,1) -- (8-.1,3);
\draw [line width=.6pt] (7-.1,3) -- (8-.1,4);
\draw [line width=.6pt] (8-.1,5) -- (10-.1,7);
\draw [line width=.6pt] (13-.1,4) -- (14-.3,5);
\draw [line width=.6pt] (15-.1,2) -- (17-.1,4);
\draw [line width=.6pt] (15-.1,7) -- (16-.1,8);
\draw [line width=.6pt] (16-.1,9) -- (18-.1,11);
\draw [line width=.6pt] (19-.1,4) -- (20-.1,5);

\draw [line width=1pt, dotted] (20+.1, 4) -- (19-.1, 6);

\filldraw 
([xshift=-2.5pt,yshift=-2.5pt]0-.1,1) rectangle ++(5pt, 5pt)
([xshift=-2.5pt,yshift=-2.5pt]1-.1,2) rectangle ++(5pt, 5pt)
([xshift=-2.5pt,yshift=-2.5pt]2-.1,1) rectangle ++(5pt, 5pt)
([xshift=-2.5pt,yshift=-2.5pt]2-.1,2) rectangle ++(5pt, 5pt)
([xshift=-2.5pt,yshift=-2.5pt]2-.1,3) rectangle ++(5pt, 5pt)
([xshift=-2.5pt,yshift=-2.5pt]5-.1,2) rectangle ++(5pt, 5pt)
([xshift=-2.5pt,yshift=-2.5pt]6-.1,1) rectangle ++(5pt, 5pt)
([xshift=-2.5pt,yshift=-2.5pt]6-.1,2) rectangle ++(5pt, 5pt)
([xshift=-2.5pt,yshift=-2.5pt]6-.1,3) rectangle ++(5pt, 5pt)
([xshift=-2.5pt,yshift=-2.5pt]6-.1,4) rectangle ++(5pt, 5pt)
([xshift=-2.5pt,yshift=-2.5pt]7-.1,2) rectangle ++(5pt, 5pt)
([xshift=-2.5pt,yshift=-2.5pt]7-.1,3) rectangle ++(5pt, 5pt)
([xshift=-2.5pt,yshift=-2.5pt]8-.1,3) rectangle ++(5pt, 5pt)
([xshift=-2.5pt,yshift=-2.5pt]8-.1,4) rectangle ++(5pt, 5pt)
([xshift=-2.5pt,yshift=-2.5pt]8-.1,5) rectangle ++(5pt, 5pt)
([xshift=-2.5pt,yshift=-2.5pt]9-.1,6) rectangle ++(5pt, 5pt)
([xshift=-2.5pt,yshift=-2.5pt]10-.1,5) rectangle ++(5pt, 5pt)
([xshift=-2.5pt,yshift=-2.5pt]10-.1,6) rectangle ++(5pt, 5pt)
([xshift=-2.5pt,yshift=-2.5pt]10-.1,7) rectangle ++(5pt, 5pt)
([xshift=-2.5pt,yshift=-2.5pt]13-.1,2) rectangle ++(5pt, 5pt)
([xshift=-2.5pt,yshift=-2.5pt]13-.1,4) rectangle ++(5pt, 5pt)
([xshift=-2.5pt,yshift=-2.5pt]14-.1,4) rectangle ++(5pt, 5pt)
([xshift=-2.5pt,yshift=-2.5pt]14-.1,5) rectangle ++(5pt, 5pt)
(14-.2-.1,5) ellipse (5pt and 3pt)
([xshift=-2.5pt,yshift=-2.5pt]14-.1,6) rectangle ++(5pt, 5pt)
([xshift=-2.5pt,yshift=-2.5pt]14-.1,7) rectangle ++(5pt, 5pt)
([xshift=-2.5pt,yshift=-2.5pt]14-.1,8) rectangle ++(5pt, 5pt)
([xshift=-2.5pt,yshift=-2.5pt]15-.1,2) rectangle ++(5pt, 5pt)
([xshift=-2.5pt,yshift=-2.5pt]15-.1,7) rectangle ++(5pt, 5pt)
([xshift=-2.5pt,yshift=-2.5pt]16-.1,3) rectangle ++(5pt, 5pt)
([xshift=-2.5pt,yshift=-2.5pt]16-.1,5) rectangle ++(5pt, 5pt)
([xshift=-2.5pt,yshift=-2.5pt]16-.1,8) rectangle ++(5pt, 5pt)
([xshift=-2.5pt,yshift=-2.5pt]16-.1,9) rectangle ++(5pt, 5pt)
([xshift=-2.5pt,yshift=-2.5pt]17-.1,2) rectangle ++(5pt, 5pt)
([xshift=-2.5pt,yshift=-2.5pt]17-.1,3) rectangle ++(5pt, 5pt)
([xshift=-2.5pt,yshift=-2.5pt]17-.1,4) rectangle ++(5pt, 5pt)
([xshift=-2.5pt,yshift=-2.5pt]17-.1,10) rectangle ++(5pt, 5pt)
([xshift=-2.5pt,yshift=-2.5pt]18-.1,3) rectangle ++(5pt, 5pt)
([xshift=-2.5pt,yshift=-2.5pt]18-.1,9) rectangle ++(5pt, 5pt)
([xshift=-2.5pt,yshift=-2.5pt]18-.1,10) rectangle ++(5pt, 5pt)
([xshift=-2.5pt,yshift=-2.5pt]18-.1,11) rectangle ++(5pt, 5pt)
([xshift=-2.5pt,yshift=-2.5pt]19-.1,4) rectangle ++(5pt, 5pt)
([xshift=-2.5pt,yshift=-2.5pt]19-.1,5) rectangle ++(5pt, 5pt)
([xshift=-2.5pt,yshift=-2.5pt]19-.1,6) rectangle ++(5pt, 5pt)
([xshift=-2.5pt,yshift=-2.5pt]20-.1,3) rectangle ++(5pt, 5pt)
([xshift=-2.5pt,yshift=-2.5pt]20-.1,5) rectangle ++(5pt, 5pt)
([xshift=-2.5pt,yshift=-2.5pt]21-.1,4) rectangle ++(5pt, 5pt)
([xshift=-2.5pt,yshift=-2.5pt]21-.1,8) rectangle ++(5pt, 5pt)
([xshift=-2.5pt,yshift=-2.5pt]21-.1,9) rectangle ++(5pt, 5pt)
([xshift=-2.5pt,yshift=-2.5pt]21-.1,10) rectangle ++(5pt, 5pt)
;

\draw (0-.1, 1-.1) node [anchor=north east] {$\rho h_1$};
\draw (0+.15, 1) node [anchor=west] {$h_0$};
\draw (1-.05, 1) node [anchor=north west] {$[\tau h_1]$};
\draw (3+.15, 1) node [anchor=north] {$[\tau^2]h_2$};
\draw (6+.4, 2) node [anchor=south ] {$[\tau^2]^2h_2^2$};
\draw (7+.15, 1) node [anchor=north] {$[\tau^2]^2h_3$};
\draw (8+.15, 3) node [anchor=south west] {$[\tau^2]^2[\tau c_0]$};
\draw (9, 5) node [anchor=north] {$[\tau^2]^2[\tau Ph_1]$};
\draw (11+.15, 5) node [anchor=north] {$[\tau^2]^3Ph_2$};
\draw (14+.25, 4) node [anchor=north] {$[\tau^2]^4d_0$};
\draw (14+.05, 2) node [anchor=north] {$[\tau^2]^4h_3^2$};
\draw (16+.1, 7) node [anchor=north] {$[\tau^2]^4[\tau Pc_0]$};
\draw (17, 9) node [anchor=north] {$[\tau^2]^4[\tau P^2h_1]$};
\draw (18+.15, 2) node [anchor=north] {$[\tau^2]^5h_2h_4$};
\draw (19+.15, 9) node [anchor=north] {$[\tau^2]^5P^2h_2$};
\draw (19, 3) node [anchor=north] {$[\tau^2]^5[\tau c_1]$};
\draw (19-.05, 4) node [anchor=north] {$[\tau^2]^5[\rho\tau g]$};
\draw (20-.3, 4) node [anchor=north west] {$[\tau^2]^5[\tau^2 g]$};
\draw (20+.05, 5) node [anchor=south west] {$[\tau^2]^6h_2e_0$};
\draw (21+.15, 3) node [anchor=north] {$[\tau^2]^6h_2^2h_4$};
\draw (21-.3, 4) node [anchor=south west] {$\rho[\tau^2]^5 h_2 c_1$};

\draw (19+.3, 5+.3) node [anchor=south west] {$d_2 ?$};

\end{tikzpicture}
}
\caption{$E_{\infty}$ page of MASS for $\bbF_q$ with
  $q \equiv 3 \bmod 4$, weight 0}
  \label{fig:3mod4-infty}
\end{figure}

\bibliography{bibliography} 
\bibliographystyle{gtart}

\end{document}